\documentclass[leqno]{article}
\usepackage{amsmath}
\usepackage{amssymb}
\usepackage{amsthm}
\usepackage{enumerate}
\usepackage{mathrsfs}
\usepackage[dvipdfmx]{hyperref}
\usepackage{newtxtext}
\usepackage{newtxmath}

\usepackage{geometry}
\usepackage[nottoc]{tocbibind}

\makeatletter
    
    \@addtoreset{equation}{section}
\makeatother

\makeatletter
\def\address#1#2{\begingroup
\noindent\parbox[t]{7.8cm}{%
\small{\scshape\ignorespaces#1}\par\vskip1ex
\noindent\small{\itshape E-mail address}%
\/: #2\par\vskip4ex}\hfill%
\endgroup}%
\makeatother

\newtheorem{theorem}{Theorem}[section]
\newtheorem{lemma}[theorem]{Lemma}

\newtheorem{proposition}[theorem]{Proposition}

\theoremstyle{definition}
\newtheorem{definition}[theorem]{\rm Definition}
\newtheorem{remark}[theorem]{\rm Remark}
\newtheorem{hypothesis}[theorem]{\rm Hypothesis}

\makeatletter
\renewenvironment{proof}[1][\proofname]{\par
  \pushQED{\qed}%
  \normalfont \topsep6\p@\@plus6\p@\relax
  \trivlist
  \item[\hskip\labelsep
        \upshape
    #1\@addpunct{.}]\ignorespaces
}{%
  \popQED\endtrivlist\@endpefalse
}
\makeatletter

\newcommand{\lref}[2][]
	{\ifx#1""{Lemma~\textup{\ref{#2}}}%
	\else{Lemmas~\textup{\ref{#1}} and \textup{\ref{#2}}}\fi}
\newcommand{\lrefs}[3][]
	{\ifx#1""{Lemmas~\textup{\ref{#2}} and \textup{\ref{#3}}}%
	\else{Lemmas~\textup{\ref{#1}}, \textup{\ref{#2}} and \textup{\ref{#3}}}\fi}
\newcommand{\tref}[2][]
	{\ifx#1""{Theorem~\textup{\ref{#2}}}%
	\else{Theorems~\textup{\ref{#1}} and \textup{\ref{#2}}}\fi}
\newcommand{\trefs}[4][]
	{\ifx#1""{Theorems~\textup{\ref{#2}}, \textup{\ref{#3}} and \textup{\ref{#4}}}%
	\else{Theorems~\textup{\ref{#1}}, \textup{\ref{#2}}, \textup{\ref{#3}} and \textup{\ref{#4}}}\fi}
\newcommand{\pref}[1]{Proposition~\textup{\ref{#1}}}
\newcommand{\prefs}[3][]
	{\ifx#1""{Propositions~\textup{\ref{#2}} and \textup{\ref{#3}}}%
	\else{Propositions~\textup{\ref{#1}}, \textup{\ref{#2}} and \textup{\ref{#3}}}\fi}
\newcommand{\rref}[1]{Remark~\textup{\ref{#1}}}

\newcommand{\hyporef}[1]{Hypothesis~\textup{\ref{#1}}}

\newcommand{\secref}[1]{Section~\textup{\ref{#1}}}

\newcommand{\RealNum}{\mathbf{R}}
\newcommand{\NaturalNum}{\mathbf{N}}
\newcommand{\SymmetricGroup}[1]{\mathfrak{S}_{#1}}

\newcommand{\indicator}[1]{\mathbf{1}_{#1}}

\newcommand{\hermitePoly}[1]{H_{#1}}

\newcommand{\probSp}{\Omega}
\newcommand{\sigmaField}{\mathcal{F}}
\newcommand{\prob}{\boldsymbol P}
\newcommand{\expect}{\boldsymbol E}

\newcommand{\Hurst}{H}
\newcommand{\CM}{\mathfrak{H}}
\newcommand{\WienerInt}[1]{I_{#1}}
\newcommand{\WienerChaos}[1]{\mathcal{H}_{#1}}
\newcommand{\WienerHermitePoly}[1]{H_{#1}}
\newcommand{\CN}{\Omega^{{\rm CN}(m)}}

\newcommand{\innerProdLarge}[3][]{\left\langle{#2},{#3}\right\rangle_{#1}}
\newcommand{\innerProd}[3][]{\langle{#2},{#3}\rangle_{#1}}

\newcommand{\App}{\bar{X}^{(m)}}
\newcommand{\TR}[2][]{\ifx#1""{#2}^{\textup{TR}}\else{#2}^{\textup{TR}(#1)}\fi}

\newcommand{\SmoothFunc}[4][]{C_{\textup{#1}}^{#2}(#3;#4)}
\newcommand{\HolFunc}[4][]{\mathscr{C}_{\textup{#1}}^{#2}(#3;#4)}
\newcommand{\HolExp}{\lambda}

\newcommand{\LebSp}[3]{\mathcal{L}^{#1}(#2;#3)}
\newcommand{\SobSp}[4]{\mathcal{D}^{#1,#2}(#3;#4)}

\newcommand{\hermiteVar}[2]{V_{#1}^{(#2)}}
\newcommand{\wHerVar}[2]{U_{#1}^{(#2)}}
\newcommand{\TRVar}[1]{\tilde{V}^{(#1)}}
\newcommand{\wTRVar}[2]{\tilde{U}_{#1}^{(#2)}}

\newcommand{\const}[1][]{C_{#1}}

\newcommand{\intPart}[1]{\lfloor{#1}\rfloor}

\newcommand{\dyadicPart}[2][]{{\ifx#1""\tau_{#2}\else\tau^{#1}_{#2}\fi}}

\newcommand{\lowerPart}[2][]{\eta_{-}^{#1}({#2})}

\newcommand{\Wronskian}[1][]{\ifx#1""{\mathsf{w}}\else{\mathsf{w}_{#1}}\fi}
\newcommand{\vecField}[1]{\mathcal{#1}}

\title{Error analysis for approximations to one-dimensional SDEs via the
perturbation method
\footnote
{
Accepted for publication in Osaka Journal of Mathematics.
}
	\footnote
	{
	 	MSC2010 subject classifications: 60F05, 60H35, 60G15.
	}
}
\author{Shigeki Aida and Nobuaki Naganuma}

\date{\empty}
\begin{document}

\maketitle
\begin{abstract}
	We study asymptotic error distributions
	associated with standard approximation scheme
	for one-dimensional stochastic differential equations
	driven by fractional Brownian motions.
	This problem was studied by,
	for instance, Gradinaru-Nourdin \cite{GradinaruNourdin2009},
	Neuenkirch and Nourdin~\cite{NeuenkirchNourdin2007}
	and the second named author~\cite{Naganuma2015}.
	The aim of this paper is to extend their results to the case where
	the equations contain drift terms and simplify the proof of
	estimates of the remainder terms in \cite{Naganuma2015}.
	To this end, we represent the approximation solution
	as the solution of the equation which is obtained
	by replacing the fractional Brownian path with a perturbed path.
	We obtain the asymptotic error distribution
	as a directional derivative of the solution by using this expression.
\end{abstract}


\section{Introduction}

For a one-dimensional fractional Brownian motion (fBm)
$B$ with the Hurst $1/3<\Hurst<1$,
we consider a one-dimensional stochastic differential equation (SDE)
\begin{align}\label{eq_20141122070455}
	X_t
	=
		\xi
		+
		\int_0^t
			b(X_s)\,
			ds
		+
		\int_0^t
			\sigma(X_s)\,
			d^\circ B_s,
	\quad
	t\in[0,1],
\end{align}
where $\xi\in\RealNum$ is a deterministic initial value
and $d^\circ B$ stands for the symmetric integral in the sense of
Russo-Vallois.
We may write $X_t(\xi,B)$, $X_t(B)$ to indicate the
dependence of the initial value and the driving path.
We consider three schemes to
approximate the solution to \eqref{eq_20141122070455}
and study asymptotic error distributions of them.
We treat the Euler scheme, the Milstein type scheme and the Crank-Nicolson scheme
as real-valued stochastic processes on the interval $[0,1]$.

There are several frameworks to treat SDEs driven by fBm.
For multidimensional case,
the Young integration theory and the rough path analysis are powerful tools \cite{Lyons1994,Lyons1998}.
We can however deal with SDEs in dimension one more easily
by using the theory of the symmetric integral \cite{Nourdin2004}.
The symmetric integral was proposed by Russo-Vallois \cite{RussoVallois1993a}
with a motivation to establish non-causal stochastic integration theory.
Recently, Nourdin and his coauthors developed
a theory of integration with respect to general integrators including fBm
\cite{Nourdin2004,GradinaruNourdinRussoVallois2005} with a spirit of
\cite{RussoVallois1993a}.
In the present article, we adopt the symmetric integral and
give a meaning to \eqref{eq_20141122070455}.

The Euler scheme, the Milstein type scheme and the Crank-Nicolson scheme
for SDEs driven by fBm are considered by many researchers.
In the consideration of approximation schemes,
they are interested in the sharp error bounds (convergence rates)
and the limits of errors normalized by the convergence rates
(asymptotic error distribution).
In multidimensional case, Mishura-Shevchenko \cite{MishuraShevchenko2008},
Friz-Riedel \cite{FrizRiedel2014} and Bayer et al.\ \cite{BayerFrizRiedelSchoenmakers2016} obtain
an almost sharp convergence rate of the Euler scheme and the Milstein type scheme, respectively.
Hu-Liu-Nualart \cite{HuLiuNualart2016} consider asymptotic error distributions
of the Euler scheme for SDEs driven by fBm with $1/2<\Hurst<1$.
Liu-Tindel \cite{LiuTindel2017arXiv} treat the same problem in the case $1/3<\Hurst<1/2$.
There are a lot of results on asymptotic error
distributions of schemes for one-dimensional SDEs.
For example, Neuenkirch-Nourdin \cite{NeuenkirchNourdin2007} show
the convergence of the normalized error of the Euler scheme
for an SDE with a drift term driven by fBm with $1/2<\Hurst<1$.
Gradinaru-Nourdin \cite{GradinaruNourdin2009} deal with
the Milstein type scheme for an SDE without a drift term, namely $b\equiv 0$ in \eqref{eq_20141122070455},
and prove that the normalized error of it converges to some random
variable.

We next explain preceding results on the Crank-Nicolson scheme for one
dimensional SDE.
The first result on the error of it is obtained in
\cite{NeuenkirchNourdin2007};
the authors obtain an almost sharp convergence rate.
In \cite{GradinaruNourdin2009},
the authors treat the error of the Crank-Nicolson scheme
for an SDE without a drift term driven by a standard Brownian motion
and obtain the convergence of the normalized error.
The second named author~\cite{Naganuma2015} in the present paper
shows the convergence of the normalized error for fBm with $1/3<\Hurst<1/2$.
It is crucial to these studies that
the solution is given by a function of
$B_t$ as
$X_t(\xi,B)=\phi(\xi, B_t)$, where $\phi$ is a certain
smooth increasing function depending only on $\sigma$.
This is a Doss-Sussmann type representation of
the solution.
Let denote the approximation solution by $\bar{X}^{(m)}_t(\xi, B)$,
where $m$ is a positive integer.
Let $B^{(m)}_t$
be the dyadic polygonal approximation of
the fBm $B$ such that $B^{(m)}_{\dyadicPart[m]{k}}=B_{\dyadicPart[m]{k}}$ for every $k=0,\dots,2^m$,
where $\dyadicPart[m]{k}=k2^{-m}$.
For the Wong-Zakai approximation,
$\bar{X}^{(m)}_t(\xi,B)=\phi(\xi,B^{(m)}_t)$ holds.
 Hence the analysis of the error $\bar{X}^{(m)}_t-X_t$
is almost similar to that of $B-B^{(m)}$ itself.
Clearly, this simple relation does not hold any more for other
approximation schemes such as the Euler, Milstein and Crank-Nicolson
schemes.
However, if the dispersion coefficient $\sigma$ is strictly
positive, there exists unique $B$-dependent
random variable $h^{(m)}$ such that
$\bar{X}^{(m)}_{\tau^m_k}(\xi,B)=\phi(\xi, B+h^{(m)})$
for all $k$.
After obtaining this formula, it is clear that the analysis
of $\{h^{(m)}\}$ is important to the study of the error
$X_t(\xi,B)-\bar{X}^{(m)}_t(\xi,B)$.
This is one of main ideas of the proof in \cite{NeuenkirchNourdin2007,
Naganuma2015}.

Even if the equations contain the drift terms,
the Doss-Sussmann representation still holds and the solution mapping
$B\mapsto X(\xi,B)$ is Lipschitz continuous in the uniform convergence
topology in one dimensional cases.
Further, under the nondegeneracy assumption of $\sigma$,
we can show that there exists a unique
piecewise linear $h^{(m)}$ such that
$\bar{X}^{(m)}_{\dyadicPart[m]{k}}(\xi,B)=X_{\dyadicPart[m]{k}}(\xi,B+h^{(m)})$
~$(0\le k\le 2^m)$ hold.
By this perturbation representation of
the approximate solutions and the analysis of
$h^{(m)}$, we can show the convergence of
the normalized error distribution.
Hence, the present paper is a natural extension
of the preceding studies.
We use central limit theorem for the Hermite variation process
to see the asymptotic behavior of
the normalized error
similarly to
\cite{NeuenkirchNourdin2007, Naganuma2015}.
The proof that the remainder term is negligible in \cite{Naganuma2015}
was done by a long calculation.
In this paper, we give simpler and shorter argument for estimates
of remainder terms.

The organization of this paper is as follows.
In \secref{sec_201708102140},
we explain three approximation schemes, that is,
Euler, Milstein and Crank-Nicolson scheme.
We next state our main theorems which determine the asymptotic error
distributions in Theorem~\ref{thm_20141204013919},
Theorem~\ref{thm_20141204015019} and
Theorem~\ref{thm_20141204015044}.
The next two sections are preliminaries for the proofs of
these theorems.
In \secref{sec_20150623023608},
we recall the definition of Russo-Vallois symmetric integral.
We consider the solutions to
SDEs driven by fractional Brownian motions with the Hurst parameter
$1/3<\Hurst\le 1/2$.
In this case, the solution has a Doss-Sussmann representation
and the Russo-Vallois integral is the same as
the symmetric Riemman-Stieltjes integral as Stratonovich integral.
By using this, we obtain estimates of iterated integrals.
Also we prepare lemmas for directional derivative of the solution
with respect to the driving path.
In \secref{sec_20150519021523},
we collect necessary results for convergence of
variation functionals.
These are essential for the proof of our main theorems.
We give the proof of these results in Appendixes~\ref{sec_20150107074711} and \ref{sec_20170519021523}.
In \secref{sec_201707212021},
we consider the Crank-Nicolson scheme and
prove Theorem~\ref{thm_20141204015044}.
For the reader's convenience, we give a skecth of
the proof by using the perturbation path $h^{(m)}$
in \rref{rem_201708102135}.
The proof of other two theorems are essentially similar to
that of this theorem. We give the sketch of the proof
for other two schemes, Euler scheme and Milstein type scheme
in \secref{sec_201707212051}.
In Appendix~\ref{sec_20150623014415},
we prepare the Gaussian analysis and Malliavin calculus.
In Appendixes~\ref{sec_20150107074711} and \ref{sec_20170519021523}, we prove the results stated in \secref{sec_20150519021523}.

Throughout this paper, we use the following notaion.
For $m\in\NaturalNum$, we denote by $\{\dyadicPart[m]{k}\}_{k=0}^{2^m}$
the $m$-th dyadic rationals,
that is, $\dyadicPart[m]{k}=k2^{-m}$ for $k=0,\dots,2^m$.
For $n\in\{0\}\cup\NaturalNum\cup\{\infty\}$,
$\SmoothFunc{n}{\RealNum^d}{\RealNum}$ denotes
the set of all $n$-times continuously differentiable
$\RealNum$-valued functions defined on $\RealNum^d$.
For $n\in\{0\}\cup\NaturalNum\cup\{\infty\}$, $\SmoothFunc[bdd]{n}{\RealNum^d}{\RealNum}$ (resp. $\SmoothFunc[poly]{n}{\RealNum^d}{\RealNum}$)
stands for the set of all functions $f\in\SmoothFunc{n}{\RealNum^d}{\RealNum}$
which are bounded (resp. polynomial growth) together with all their derivatives.
For $k,l\in\{0\}\cup\NaturalNum$, $\SmoothFunc{k,l}{\RealNum^2}{\RealNum}$
denotes the set of all functions $f:\RealNum^2\to\RealNum$
which is $k$-times (resp. $l$-times) continuously differentiable with respect to the first (resp. second) variable.
We denote the set of right continuous paths on $\RealNum^d$ whose left limit exist by $D([0,1];\RealNum^d)$.
For $\HolExp\in(0,1]$,
$\HolFunc{\HolExp}{[0,1]}{\RealNum}$ stands for
the set of all $\HolExp$-H\"{o}lder continuous functions from $[0,1]$ to $\RealNum$.
The space $\HolFunc[0]{\HolExp}{[0,1]}{\RealNum}$ is the set of all functions
$g\in\HolFunc{\HolExp}{[0,1]}{\RealNum}$ starting from zero.
For $g\in\HolFunc{\HolExp}{[0,1]}{\RealNum}$ and $0\leq t\leq 1$,
we define the uniform norm by
$
	\|g\|_{\infty,[0,t]}
	=
		\sup_{0\leq s\leq t}
			|g_s|
$.
We simply write $\|g\|_{\infty}=\|g\|_{\infty,[0,1]}$.
For fixed $0<s<1$, we define the shift operator $\theta_s$ by
$(\theta_s g)(t)=g_{t+s}-g_s$ for $0\leq t\leq 1-s$.

\section{Main results}\label{sec_201708102140}

We state our main result.
For $b,\sigma\in\SmoothFunc[bdd]{\infty}{\RealNum}{\RealNum}$,
we consider an SDE \eqref{eq_20141122070455}.
Throughout this paper, we consider a solution $X$ to \eqref{eq_20141122070455} given by \eqref{eq_1545722900}.
We refer the meaning of SDEs driven by
fBm to \secref{sec_20150623023608}.
To state our main results, we recall the definitions
of three approximation schemes.
%

\begin{definition}[The Euler scheme]
 For every $m\in\NaturalNum$, the Euler scheme
 $\bar{X}^{(m)}:[0,1]\to\RealNum$ is defined by
	\begin{align*}
		\left\{
			\begin{aligned}
				\bar{X}^{(m)}_0
				&=
					\xi,\\
								\bar{X}^{(m)}_t
				&=
					\bar{X}^{(m)}_{\dyadicPart[m]{k-1}}
					+
					b(\bar{X}^{(m)}_{\dyadicPart[m]{k-1}})
					(t-\dyadicPart[m]{k-1})
					+
			 		\sigma(\bar{X}^{(m)}_{\dyadicPart[m]{k-1}})
					(B_t-B_{\dyadicPart[m]{k-1}})
					\quad
				\text{for $\dyadicPart[m]{k-1}<t\leq\dyadicPart[m]{k}$.}
			\end{aligned}
		\right.
	\end{align*}
\end{definition}

\begin{definition}[The Milstein type scheme]
 For every $m\in\NaturalNum$, the Milstein type scheme
 $\bar{X}^{(m)}:[0,1]\to\RealNum$ is defined by
	\begin{align*}
		\left\{
			\begin{aligned}
				\bar{X}^{(m)}_0
				&=
					\xi,\\
				\bar{X}^{(m)}_t
				&=
					\App_{\dyadicPart[m]{k-1}}
					+
					b(\App_{\dyadicPart[m]{k-1}})
					(t-\dyadicPart[m]{k-1})
					+
					\frac{1}{2}
					bb'(\App_{\dyadicPart[m]{k-1}})
					(t-\dyadicPart[m]{k-1})^2\\
				&\phantom{=}\quad
					+
					\frac{1}{2}
					[\sigma b'+\sigma'b]
					(\App_{\dyadicPart[m]{k-1}})
					(t-\dyadicPart[m]{k-1})
					(B_t-B_{\dyadicPart[m]{k-1}})\\
				&\phantom{=}\quad
					+
					\sigma(\App_{\dyadicPart[m]{k-1}})
					(B_t-B_{\dyadicPart[m]{k-1}})
					+
					\frac{1}{2}
					\sigma\sigma'(\App_{\dyadicPart[m]{k-1}})
					(B_t-B_{\dyadicPart[m]{k-1}})^2
				\quad
				\text{for $\dyadicPart[m]{k-1}<t\leq\dyadicPart[m]{k}$}.
			\end{aligned}
		\right.
	\end{align*}
\end{definition}

\begin{definition}[The Crank-Nicolson scheme]
	For every $m\in\NaturalNum$, the Crank-Nicolson scheme $\App :[0,1]\to\RealNum$ is defined by
	a solution to an equation
	\begin{align*}
		\left\{
			\begin{aligned}
				\App_0
				&=
					\xi,\\
				\App_t
				&=
					\App_{\dyadicPart[m]{k-1}}
					+
					\frac{1}{2}
					\left\{
						b(\App_{\dyadicPart[m]{k-1}})
						+
						b(\App_t)
					\right\}
					(t-\dyadicPart[m]{k-1})\\
				&\phantom{=}\quad\qquad
					+
					\frac{1}{2}
					\left\{
						\sigma(\App_{\dyadicPart[m]{k-1}})
						+
						\sigma(\App_t)
					\right\}
					(B_t-B_{\dyadicPart[m]{k-1}})
				\quad
				\text{for $\dyadicPart[m]{k-1}<t\leq\dyadicPart[m]{k}$.}
			\end{aligned}
		\right.
	\end{align*}
\end{definition}
Since the Crank-Nicolson scheme is an implicit scheme,
we need to restrict the domain of it
and assure an existence of a solution to the equation above.
Roughly speaking, the existence of the solution is ensured for large $m$.

In order to state our main results concisely,
we set $\Wronskian=\sigma b'-\sigma'b$ and
\begin{align}\label{eq_20141226020239}
	J_t
	=
		\exp
			\left(
				\int_0^t
					b'(X_u)\,
					du
				+
				\int_0^t
					\sigma'(X_u)\,
					d^\circ B_u
			\right).
\end{align}
We assume the following hypothesis
in order to obtain an expression of the error of the scheme;
\begin{hypothesis}\label{hypo_ellipticity}
	$\inf\sigma>0$.
\end{hypothesis}

The following are our main results.

\begin{theorem}[Euler scheme]\label{thm_20141204013919}
 We consider the Euler scheme.
 Assume that \hyporef{hypo_ellipticity} is satisfied.
	For $1/2<\Hurst<1$, we have
	\begin{align*}
		\lim_{m\to\infty}
			2^{m(2\Hurst-1)}
			\{\bar{X}^{(m)}-X\}
		=
			\sigma(X)
			U
			+
			J
			\int_0^\cdot
				J_s^{-1}
				\Wronskian(X_s)
				U_s\,
				ds
	\end{align*}
	in probability with respect to the uniform norm.
	Here $U$ is defined by
	\begin{align*}
		U_t
		=
			\int_0^t
				f_2(X_u)\,
				du,
	\end{align*}
	where $f_2=-\sigma'/2$.
\end{theorem}

In this theorem, the limit is a continuous stochastic process
indexed by the elements of the interval $[0,1]$.
When we emphasize the time parameter $t$, we express the limit process as
$
	\sigma(X_t)
	U_t
	+
	J_t
	\int_0^t
		J_s^{-1}
		\Wronskian(X_s)
		U_s\,
		ds
$.

\begin{theorem}[Milstein type scheme]\label{thm_20141204015019}
 Assume that \hyporef{hypo_ellipticity} is satisfied.
 We consider Milsten type scheme.
	For $1/3<\Hurst<1/2$ (resp. $\Hurst=1/2$), we have
	\begin{align*}
		\lim_{m\to\infty}
			2^{m(4\Hurst-1)}
			\{\App-X\}
		=
			\sigma(X)
			U
			+
			J
			\int_0^\cdot
				J_s^{-1}
				\Wronskian(X_s)
				U_s\,
				ds
	\end{align*}
	in probability (resp. weakly) with respect to the uniform norm.
	Here $U$ is a stochastic process defined as follows;
	we set
	\begin{gather*}
		\begin{aligned}
			\psi
			&=
				-
				\frac{1}{4}
				\left[
					\frac{\sigma'(\sigma b'+\sigma'b)+\sigma(\sigma''b+\sigma b'')}{\sigma}
				\right],
			&
			f_3
			&=
				-
				\frac{1}{3!}
				[(\sigma')^2+\sigma\sigma''],
		\end{aligned}\\
		\begin{aligned}
			f^\dagger_4
			&=
				\frac{1}{24}
				[
					\sigma^2\sigma'''
					+6\sigma\sigma'\sigma''
					+3(\sigma')^3
				],
			&
			g_1
			&=
				\frac{\Wronskian}{\sigma}.
		\end{aligned}
	\end{gather*}
	\begin{enumerate}
		\item	For $1/3<\Hurst<1/2$, we set
				\begin{align*}
					U_t
					&=
						3
						\int_0^t
							f^\dagger_4(X_u)\,
							du.
				\end{align*}
		\item	For $\Hurst=1/2$, we set
				\begin{multline*}
					U_t
					=
						\int_0^t
							\psi(X_u)\,
							du
						+
						\sqrt{6}
						\int_0^t
							f_3(X_u)\,
							dW_u
						+
						3
						\int_0^t
							f_3(X_u)
							\circ
							dB_u\\
						+
						3
						\int_0^t
							f^\dagger_4(X_u)\,
							du
						+
						\frac{1}{\sqrt{12}}
						\int_0^t
							g_1(X_u)\,
							d\tilde{W}_u,
				\end{multline*}
				where $W$ and $\tilde{W}$ are standard
    			Brownian motions and $B$, $W$ and $\tilde{W}$ are independent.
    			Also $dW_u, d\tilde{W}_u$ and ${}\circ dB_u$ stand
                for the It\^o integral and the Stratonovich integral, respectively.
	\end{enumerate}
\end{theorem}

\begin{theorem}[Crank-Nicolson scheme]\label{thm_20141204015044}
	Assume that \hyporef{hypo_ellipticity} is satisfied.
	For $1/3<\Hurst\leq 1/2$, we have
	\begin{align*}
		\lim_{m\to\infty}
			2^{m(3\Hurst-1/2)}
			\{\App-X\}
		=
			\sigma(X)
			U
			+
			J
			\int_0^\cdot
				J_s^{-1}
				\Wronskian(X_s)
				U_s\,
				ds
	\end{align*}
	weakly with respect to the uniform norm.
	Here $U$ is a stochastic process defined as follows;
	we set
	\begin{align*}
		\psi
		&=
			\frac{1}{4}
			[\sigma'b'+\sigma''b],
		&
		f_3
		&=
			\frac{1}{12}
			[(\sigma')^2+\sigma\sigma''],
		&
		g_1
		&=
			\frac{\Wronskian}{\sigma}.
	\end{align*}
	\begin{enumerate}
		\item	For $1/3<\Hurst<1/2$, we set
				\begin{align*}
					U_t
					=
						\sigma_{3,\Hurst}
						\int_0^t
							f_3(X_u)\,
							dW_u,
				\end{align*}
				where $\sigma_{3,\Hurst}$ is a positive constant defined by \eqref{eq_const_sigma_q_Hurst}
				and $W$ is a standard Brownian motion independent of $B$.
		\item	For $\Hurst=1/2$, we set
				\begin{align*}
					U_t
					=
						\int_0^t
							\psi(X_u)\,
							du
						+
						\sqrt{6}
						\int_0^t
							f_3(X_u)\,
							dW_u
						+
						3
						\int_0^t
							f_3(X_u)
							\circ
							dB_u
						+
						\frac{1}{\sqrt{12}}
						\int_0^t
							g_1(X_u)\,
							d\tilde{W}_u,
				\end{align*}
				where $W$ and $\tilde{W}$ are standard Brownian motions
				and $B$, $W$ and $\tilde{W}$ are independent.
	\end{enumerate}
\end{theorem}

We make remarks on our main results.
\begin{enumerate}
	\item	We explain how we derive $f_i, g_1, \varphi_{\boldsymbol{i}}, \psi, f_4^{\dagger}$
	 		$(i=2,3,4, \boldsymbol{i}=011,101,110)$.
			Since \trefs{thm_20141204013919}{thm_20141204015019}{thm_20141204015044}
			are proved by the same method, we explain the case of the Crank-Nicolson scheme
			(\tref{thm_20141204015044}) as an example.
			In the first step of our proof, we need to calculate
			one-step error $\hat{\kappa}_{k}$ of each approximation scheme as in \eqref{eq_1545356684}.
			In that calculation, the functions
			$\hat{f}_i, \hat{g}_1, \hat{\varphi}, \hat{\varphi}_{\boldsymbol{i}}$,
			which are defined by $\sigma$ and $b$, appear as the coefficients
			of the monomials of the increments of $\Delta B_k=B_{\dyadicPart[m]{k}}-B_{\dyadicPart[m]{k-1}}$
			and $\Delta=2^{-m}$ and iterated integrals of $B_t$ and $t$
			(\lref{lem_201707121721}).
			We define the functions $f_i, g_1, \varphi, \varphi_{\boldsymbol{i}}$
			by using $\hat{f}_i, \hat{g}_1, \hat{\varphi}, \hat{\varphi}_{\boldsymbol{i}}$
			and express main part of the piecewise linear function $h^{(m)}$
			in terms of $f_i, g_1, \varphi, \varphi_{\boldsymbol{i}}$ (\lref{lem_201708091518}).
			Finally, we study asymptotic of $h^{(m)}$ and then define $\psi=\phi+(\varphi_{011}+\varphi_{110})/4$
			(\lref{lem_201708091528}).
			In the case of the Euler and Milstein scheme,
			we show lemmas corresponding to \lrefs[lem_201707121721]{lem_201707121721}{lem_201708091528}.
			The function $f_4^\dagger$ in the Milstein scheme appears
			in studying in asymptotic of $h^{(m)}$.
	\item	\tref{thm_20141204013919} is an extension of \cite{NeuenkirchNourdin2007},
			but the proof is completely different and comparatively more simple.
	\item	In \cite{GradinaruNourdin2009}, the authors consider higher order schemes for SDEs without drift terms.
			\tref{thm_20141204015019} coresponds to the second order scheme for an SDE containing a drift term.
	\item	\tref{thm_20141204015044} is an extension of \cite{GradinaruNourdin2009,Naganuma2015}.
			To our knowledge, the convergence of the
			approximation solution itself is not unknown for
			$1/6<\Hurst\le 1/3$ (\cite{Nourdin2008a}).
			When $\sigma(x)^2$ is a quadratic function of $x$,
			\tref{thm_20141204015044}
			is proved in \cite{NeuenkirchNourdin2007} for
			$1/6<\Hurst<1/2$.
			In the case where $\Hurst>1/3$,
			the convergence of the approximation solution
			is a pathwise result, that is, the result holds
			for SDEs driven by H\"older continuous paths with
			H\"older exponent which is greater than $1/3$.
			However, the proof of \cite{NeuenkirchNourdin2007}
			is due to a central limit theorem and it is not clear
			that this is also a pathwise result.
\end{enumerate}

\section{ODEs driven by H\"{o}lder continuous functions and SDEs}
\label{sec_20150623023608}

In this section, we define the symmetric integral in the sense of Russo-Vallois
and discuss a unique existence and properties of a solution to
an ordinary differential equation (ODE).

Let $1/3<\HolExp<1$.
For a $\HolExp$-H\"older continuous function $g:[0,1]\to\RealNum$,
we consider an ODE
\begin{align}\label{eq_20140928070616}
	x_t
	=
		\xi
		+
		\int_0^t
			b(x_u)\,
			du
		+
		\int_0^t
			\sigma(x_u)\,
			d^\circ
			g_u,
	\quad
	t\in[0,1],
\end{align}
where $\xi\in\RealNum$ and $d^\circ g$ denotes the symmetric integral.
We shall also write $x_t(\xi,g)$, $x(\xi)$, or $x(g)$
for the solution $x$
to emphasize dependence on the initial value $\xi$ and/or the driver $g$.
Since fBm with the Hurst $1/3<\Hurst<1$ is
$(\Hurst-\epsilon)$-H\"{o}lder continuous with probability one,
we can deal with SDE \eqref{eq_20141122070455} in pathwise sense
by using the theory of ODEs \eqref{eq_20140928070616}.
We have $\HolExp=\Hurst-\epsilon$ in mind.
See \secref{sec_201708051641}.

We prepare notation.
For $g\in\HolFunc{\HolExp}{[0,1]}{\RealNum}$,
we use the symbol $\const[g]$, which may change line by line,
to denote a constant which has a bound
\begin{align*}
	\const[1]
	\left\{
		1
		+
		\sup_{0\leq s<t\leq 1}
			\frac{|g_t-g_s|}{(t-s)^\HolExp}
	\right\}^{\const[2]}
\end{align*}
for some constants $\const[1]$ and $\const[2]$,
which may depend on the H\"{o}lder exponent $\HolExp$ but not on $g$.

\subsection{Existence and uniqueness}
We collect facts on the symmetric integral
and a solution to an ODE \eqref{eq_20140928070616}.
In what follows, we assume $1/3<\HolExp<1$.
\begin{definition}
	For continuous functions $f,g:[0,1]\to\RealNum$,
	we define the symmetric integral in the sense of Russo-Vallois by
	\begin{align*}
		\int_0^t
			f_u\,
			d^\circ g_u
		=
		\lim_{\epsilon\downarrow 0}
			\int_0^t
				\frac{f_{(u+\epsilon)\wedge t}+f_u}{2}
				\frac{g_{(u+\epsilon)\wedge t}-g_u}{\epsilon}\,
				du
	\end{align*}
	if the limit of the right-hand side exists.
\end{definition}

\begin{proposition}[{\cite[Theorem~4.1.7]{Nourdin2004}}]\label{prop_20140928075311}
	Let $a\in\HolFunc{1}{[0,1]}{\RealNum}$
	and $g\in\HolFunc[0]{\HolExp}{[0,1]}{\RealNum}$.
	Then, for any $f\in\SmoothFunc{1,3}{\RealNum^2}{\RealNum}$,
	$
		\int_0^t
			\partial_2 f(a_u,g_u)\,
			d^\circ g_u
	$
	exists and it holds that
	\begin{align*}
		f(a_t,g_t)
		=
			f(a_0,g_0)
			+
			\int_0^t
				\partial_1 f(a_u,g_u)\,
				da_u
			+
			\int_0^t
				\partial_2 f(a_u,g_u)\,
				d^\circ g_u.
	\end{align*}
\end{proposition}

\begin{remark}\label{rem_20150304061109}
	Let $a\in\HolFunc{1}{[0,1]}{\RealNum}$ and $g\in\HolFunc[0]{\HolExp}{[0,1]}{\RealNum}$.
	Let $f\in\SmoothFunc{1,2}{\RealNum^2}{\RealNum}\cap \SmoothFunc{1}{\RealNum^2}{\RealNum}$.
	Then, we can choose a primitive function
	$F\in\SmoothFunc{1,3}{\RealNum^2}{\RealNum}\cap\SmoothFunc{1}{\RealNum^2}{\RealNum}$
	with respect to the second variable,
	that is, $f(x,y)=\partial_2 F(x,y)$ for any $x,y\in\RealNum$.
	Indeed, $F(x,y)=\int_0^y f(x,\eta)\,d\eta$ is a primitive function and the continuity of $\partial_1 f$ implies
	$
		\partial_1 F(x,y)
		=
			\int_0^y \partial_1 f(x,\eta)\,d\eta
	$.
	Hence, from \pref{prop_20140928075311}, we see
	$
	\int_0^t
		f(a_u,g_u)\,
		d^\circ g_u
	$
	exists
	and it holds that
	\begin{align*}
		\int_0^t
			f(a_u,g_u)\,
			d^\circ g_u
		=
			F(a_t,g_t)
			-
			F(a_0,g_0)
			-
			\int_0^t
				\partial_1 F(a_u,g_u)\,
				da_u.
	\end{align*}
\end{remark}

The next proposition asserts that a symmetric integral is a limit of a modified Riemann sum.
\begin{proposition}\label{prop_20150424100658}
	Let $a\in\HolFunc{1}{[0,1]}{\RealNum}$ and $g\in\HolFunc[0]{\HolExp}{[0,1]}{\RealNum}$.
	Let $0=t_0<\dots<t_n=t$ be a partition of $[0,t]$.
	For any $f\in\SmoothFunc{1,2}{\RealNum^2}{\RealNum}$, we see that
	\begin{align*}
		\sum_{k=1}^n
			\frac{f(a_{t_{k-1}},g_{t_{k-1}})+f(a_{t_k},g_{t_k})}{2}
			(g_{t_k}-g_{t_{k-1}})
	\end{align*}
	converges to
	$
		\int_0^t
			f(a_u,g_u)\,
			d^\circ g_u
	$
	as $\max\{t_k-t_{k-1};k=1,\dots,n\}$ tends to $0$.
\end{proposition}

\begin{proof}
	We use the formula in \rref{rem_20150304061109}.
	We have
	\begin{align*}
		\int_s^t
			f(a_u,g_u)\,
			d^{\circ}g_u
		&=
			F(a_t,g_t)-F(a_s,g_s)-\int_s^t\partial_1F(a_u,g_u)da_u\\
		&=
			\left\{
				F(a_t,g_t)-F(a_s,g_t)
				-
				\int_s^t
					\partial_1F(a_u,g_u)\,
					da_u
			\right\}
			+\{F(a_s,g_t)-F(a_s,g_s)\}\\
		&=
			\int_s^t
				\{\partial_1F(a_u,g_t)-\partial_1F(a_u,g_u)\}\,
				da_u\\
		&\phantom{=}\quad
			+f(a_s,g_s)(g_t-g_s)+\partial_2f(a_s,g_s)\frac{1}{2}(g_t-g_s)^2
			+\partial_2^2f(a_s,g_s+\theta(g_t-g_s))\frac{(g_t-g_s)^3}{3!}\\
		&=
			f(a_s,g_s)(g_t-g_s)
			+\partial_2f(a_s,g_s)\frac{1}{2}(g_t-g_s)^2
			+O(|t-s|^{1+\lambda})
			+O(|t-s|^{3\lambda}),
	\end{align*}
	where we used the Taylor formula and the H\"older continuity of $g$.
	On the other hand, by using the Taylor formula again, we have
	\begin{align*}
		\frac{f(a_s,g_s)+f(a_t,g_t)}{2}(g_t-g_s)
		&=
			f(a_s,g_s)(g_t-g_s)
			+
			\frac{1}{2}
			\left\{f(a_s,g_t)-f(a_s,g_s)\right\}
			(g_t-g_s)\\
		&\phantom{=}\quad
			+
			\frac{1}{2}
			\left\{f(a_t,g_t)-f(a_s,g_t)\right\}
			(g_t-g_s)\\
		&=
			f(a_s,g_s)(g_t-g_s)
			+\frac{1}{2}\partial_2f(a_s,g_s)(g_t-g_s)^2\\
		&\phantom{=}\quad
			+
			\frac{1}{4}
			\partial_2^2f(a_s,g_s+\theta(g_t-g_s))
			(g_t-g_s)^3\\
		&\phantom{=}\quad
			+
			\frac{1}{2}
			\partial_1f(a_s+\theta'(g_t-g_s),g_t)
			(a_t-a_s)(g_t-g_s).
	\end{align*}
	Therefore, we obtain
	\begin{align*}
		\int_s^tf(a_u,g_u)\,d^{\circ}g_u
		=
			\frac{f(a_t,g_t)+f(a_s,g_s)}{2}(g_t-g_s)
			+R(s,t),
	\end{align*}
	where $|R(s,t)|\le \const[g]|t-s|^{(1+\lambda)\wedge (3\lambda)}$.
	By the additivity property of the integral,
	$\int_s^tf(a_u,g_u)\,d^{\circ}g_u+\int_t^vf(a_u,g_u)\,d^{\circ}g_u=
	\int_s^vf(a_u,g_u)\,d^{\circ}g_u$~$(s<t<v)$ and a limiting argument,
	we obtain the desired result.
 \end{proof}

Next we consider properties of \eqref{eq_20140928070616}.
Let us start our discussion with properties of the flow $\phi$ associated to $\sigma$,
that is, $\phi$ is a unique solution $\phi$ to an ODE
\begin{align}\label{eq_20140928162150}
	\phi(\alpha,\beta)
	=
		\alpha
		+
		\int_0^\beta
			\sigma(\phi(\alpha,\eta))\,
			d\eta,
	\quad
	\beta
	\in
		\RealNum.
\end{align}
\begin{proposition}[{\cite[Lemma~2]{Doss1977}}]\label{prop_20140929101727}
	Let $n\geq 1$.
	For any $\sigma\in\SmoothFunc[bdd]{n}{\RealNum}{\RealNum}$ and
	an initial point $\alpha\in\RealNum$,
	there exists a unique solution to \eqref{eq_20140928162150}.
	The unique solution $\phi$ satisfies the following:
	\begin{enumerate}
		\item\label{item_1545964854}
				$\phi\in\SmoothFunc{n,n+1}{\RealNum^2}{\RealNum}\cap\SmoothFunc{n}{\RealNum^2}{\RealNum}$,
		\item\label{item_1545380688}
				$
					\phi(\alpha,\beta)
					=
						\phi(\phi(\alpha,\beta'),\beta-\beta')
				$,
		\item\label{item_1545964864}
				$
					\partial_1\phi(\alpha,\beta)
					=
						\exp
							\left(
								\int_0^\beta
									\sigma'(\phi(\alpha,\eta))\,
									d\eta
							\right)
				$.
	\end{enumerate}
\end{proposition}

To state assertion about uniqueness of solutions to \eqref{eq_20140928070616},
we introduce a class $\mathfrak{C}$ of the solutions by
\begin{align*}
	\mathfrak{C}
	=
		\left\{
			x\in\HolFunc{\HolExp}{[0,1]}{\RealNum};\,
			\begin{minipage}{20em}
				there exist $f\in\SmoothFunc{1,3}{\RealNum^2}{\RealNum}$ and $k\in\HolFunc{1}{[0,1]}{\RealNum}$\\
				such that $x_t=f(k_t,g_t)$ for all $t\in[0,1]$
			\end{minipage}
		\right\}.
\end{align*}
Note that $\mathfrak{C}$ depends on $g\in\HolFunc[0]{\HolExp}{[0,1]}{\RealNum}$.

\begin{proposition}[{\cite[Theorem~4.3.1]{Nourdin2004}}, {\cite[Section~3]{NourdinSimon2007}}]\label{prop_20140928080940}
	Let $g\in\HolFunc[0]{\HolExp}{[0,1]}{\RealNum}$.
	Assume that $b\in\SmoothFunc[bdd]{1}{\RealNum}{\RealNum}$
	and $\sigma\in\SmoothFunc[bdd]{2}{\RealNum}{\RealNum}$.
	Then, a unique solution to \eqref{eq_20140928070616} in the class $\mathfrak{C}$ exists
	and it is given by
	\begin{align}\label{eq_1545722900}
		x_t=\phi(a_t,g_t),
	\end{align}
	where $\phi$ and $a\equiv a(\xi,g)$ are given by solutions to
	\eqref{eq_20140928162150} and
	\begin{align*}
		a_t
		=
			\xi
			+
			\int_0^t
				f_{\sigma,b}(a_u,g_u)\,
				du,
		\quad
		t\in[0,1],
	\end{align*}
	respectively. Here $f_{\sigma,b}=f_1f_2$ with
	\begin{align*}
		f_1(x,y)
		&=
			\exp
				\left(
					-
					\int_0^y
						\sigma'(\phi(x,\eta))\,
						d\eta
				\right),
		&
		f_2(x,y)
		&=
			b(\phi(x,y)).
	\end{align*}
\end{proposition}
\begin{proof}
	It is easily shown that $x$ given by \eqref{eq_1545722900} belongs to $\mathfrak{C}$
	and satisfy \eqref{eq_20140928070616}.
	Indeed, \pref{prop_20140929101727}~(\ref{item_1545964854}) implies
	$
		\phi
		\in
			\SmoothFunc{2,3}{\RealNum^2}{\RealNum}
		\subset
			\SmoothFunc{1,3}{\RealNum^2}{\RealNum}
	$
	and $a\in\HolFunc{1}{[0,1]}{\RealNum}$.
	From \pref{prop_20140928075311} and \pref{prop_20140929101727}~(\ref{item_1545964864}),
	we see that $x$ satisfies \eqref{eq_20140928070616}.
	To prove the uniqueness, we borrow results from \cite[Section~3]{NourdinSimon2007}.
	Let $x$ be a solution in the class $\mathfrak{C}$ and given by
	$x=f(k,g)$ for $f\in\SmoothFunc{1,3}{\RealNum^2}{\RealNum}$	and $k\in\HolFunc{1}{[0,1]}{\RealNum}$.
	Since $\int_s^t x_u\,d^\circ g_u=\int_s^t f(k_u,g_u)\,d^\circ g_u$ is well-defined from \rref{rem_20150304061109},
	set
	$
		A_{st}
		=
			\int_s^t x_u\,d^\circ g_u
			-
			\frac{1}{2}
			(x_t+s_s)
			(g_t-g_s)
	$.
	Then, we deduce that $(x,A)$ is a solution to \eqref{eq_20140928070616}
	in the sense of \cite[Definition~3.1]{NourdinSimon2007}
	from \cite[Lemma~3.4 and Proposition~3.5]{NourdinSimon2007}.
	Finally, \cite[Corollary~3.7]{NourdinSimon2007} implies $x_t=\phi(a_t,g_t)$.
\end{proof}

\begin{proposition}\label{prop_20140928154144}
	Let $x$ be the solution to \eqref{eq_20140928070616} given by \eqref{eq_1545722900}.
	For fixed $0<s<1$, we have $x_{s+t}(\xi,g)=
	x_t(x_s(\xi,g),\theta_sg)$ for any $0\leq t\leq 1-s$.
\end{proposition}
\begin{proof}
	We first prove
	$a_t\left(x_s(\xi,g),\theta_s g\right)
		=
			\tilde{a}_t
		:=
			\phi(a_{s+t}(\xi,g),g_s)
	$.
	From \pref{prop_20140929101727}, we see
	\begin{align*}
		\frac{1}{f_1(x,y')}
		\cdot
		f_1(x,y)
		&=
			f_1(\phi(x,y'),y-y'),
		&
		f_2(x,y)
		&=
			f_2(\phi(x,y'),y-y').
	\end{align*}
	Hence, it holds that
	\begin{align*}
		\frac{d}{dt}
			\tilde{a}_t
		&=
			\partial_1\phi(a_{s+t}(\xi,g),g_s)
			\frac{d}{dt}a_{s+t}(\xi,g)\\
		&=
			\frac{1}{f_1(a_{s+t}(\xi,g),g_s)}
			\cdot
			f_1(a_{s+t}(\xi,g),g_{s+t})
			f_2(a_{s+t}(\xi,g),g_{s+t})\\
		&=
			[f_1f_2](\phi(a_{s+t}(\xi,g), g_s),g_{s+t}-g_s)\\
		&=
			f_{\sigma,b}(\tilde{a}_t,(\theta_s g)_t).
	\end{align*}
	By the definition of $\tilde{a}$ and \pref{prop_20140928080940},
	we have
	$
		\tilde{a}_0
		=
			\phi(a_s(\xi,g),g_s)
		=
			x_s(\xi,g)
	$.
	It follows from the uniquness of a solution that
	$
		a_t\left(x_s(\xi,g),\theta_s g\right)
		=
 		\tilde{a}_t
 	$.

	Combining \pref{prop_20140929101727} (\ref{item_1545380688}),
	\pref{prop_20140928080940} and this equality, we obtain
	\begin{align*}
		x_{s+t}(\xi,g)
		&=
			\phi(a_{s+t}(\xi,g),g_{s+t})
		=
			\phi(\phi(a_{s+t}(\xi,g_s),g_{s+t}-g_s)\\
		&=
			\phi(a_t(x_s(\xi,g),\theta_s g),
			(\theta_s g)_t)
		=
			x_t\left(x_s(\xi,g),\theta_sg\right),
	\end{align*}
	which completes the proof.
\end{proof}

\begin{remark}\label{rem_201708111136}
	We assume the same assumption as in \pref{prop_20140928080940}
	and consider the solution $x$ to \eqref{eq_20140928070616} given by \eqref{eq_1545722900}.
	In the proposition, we consider H\"older continuous paths.
	However it is easy to check that
	the mapping $g\mapsto x(g)$ can be extended to a continuous mapping
	on $C([0,1];\RealNum)$ with the uniform convergence
	norm $\|\cdot\|_{\infty}$.
	Further, by \rref{rem_20150304061109},
	for any $f\in\SmoothFunc{1,2}{\RealNum^2}{\RealNum}\cap\SmoothFunc{1}{\RealNum^2}{\RealNum}$,
	we have the continuity of
	the mapping in the uniform convergence topology :
	\begin{align*}
		C([0,1];\RealNum)\ni g
		\mapsto
			\int_0^{\cdot}f(a_s(g),x_s(g))\,d^{\circ}g_s
		\in C([0,1];\RealNum).
	\end{align*}
\end{remark}

\subsection{The Taylor expansion and its remainder estimates}
\label{sec_20141211055401}
For notational convenience,
we set $g^0_t=t$, $g^1_t=g_t$ for $0\leq t\leq 1$.
Let $x$ be the solution to \eqref{eq_20140928070616} given by \eqref{eq_1545722900}.
Assume that $b\in\SmoothFunc[bdd]{1}{\RealNum}{\RealNum}$
and $\sigma\in\SmoothFunc[bdd]{2}{\RealNum}{\RealNum}$.
For
$0\leq s\leq t\leq 1$ and $f\in\SmoothFunc[bdd]{2}{\RealNum}{\RealNum}$,
we can define
    \begin{align*}
        I^0_{st}
        &=
            \int_s^t
                f(x_u)\,
                dg^0_u,
        &
        I^1_{st}(f)
        =
            \int_s^t
                f(x_u)\,
                d^{\circ}g^1_u.
\end{align*}
Here, $I^0_{st}(f)$ is a usual Riemann integral.
As for $I^1_{st}(f)$, the reasoning is as follows.
By using functions $\phi$ and $a$ given in \pref{prop_20140928080940},
we have $f(x_u)=[f\circ\phi](a_u,g_u)$
and $f\circ\phi\in\SmoothFunc{1,2}{\RealNum^2}{\RealNum}\cap \SmoothFunc{1}{\RealNum^2}{\RealNum}$.
From \rref{rem_20150304061109}, we see $F(x,y)=\int_0^y f(x,\eta)\,d\eta$ belongs to $\SmoothFunc{1,3}{\RealNum^2}{\RealNum}\cap \SmoothFunc{1}{\RealNum^2}{\RealNum}$
and it holds that
\begin{align*}
	I^{1}_{st}(f)
	=
		\int_s^t
			[f\circ\phi](a_u,g_u)\,
			d^\circ g_u
	=
		F(a_t,g_t)
			-
			F(a_s,g_s)
			-
			\int_s^t
				\partial_1 F(a_u,g_u)\,
				da_u.
\end{align*}
Hence we see $I^{1}_{st}(f)$ is well-defined.
Further, for any $\alpha_1,\dots,\alpha_n\in\{0,1\}$, we can define the iterated integral
\begin{align*}
	I^{\alpha_1\cdots\alpha_n}_{st}(f)
	=
		\int_s^t
		I^{\alpha_1\cdots\alpha_{n-1}}_{su}(f)\,
			d^\circ g^{\alpha_n}_{u}
\end{align*}
inductively in the same way.
For $f\equiv 1$, we set
$
	g^{\alpha_1\cdots\alpha_n}_{st}
	=
		I^{\alpha_1\cdots\alpha_n}_{st}(f)
$.
We set $V_0=b$, $V_1=\sigma$
and define a vector field by $\vecField{V}_\alpha f=V_\alpha f'$.

From \rref{rem_20150304061109}, we see the following estimate.
\begin{lemma}\label{lem_201708052140}
	Assume that $b\in\SmoothFunc[bdd]{1}{\RealNum}{\RealNum}$
	and $\sigma\in\SmoothFunc[bdd]{2}{\RealNum}{\RealNum}$.
	Let $f\in\SmoothFunc[bdd]{2}{\RealNum}{\RealNum}$.
	Let $\alpha_1,\ldots,\alpha_n\in \{0,1\}$ and set $r_i=\sharp\{k=1,\dots,n;\alpha_k=i\}$.
	Then,
	there exists a constant $C=\const[f,g,\alpha_1,\ldots,\alpha_n]$
	which depends only on $f$, the H\"older constant of $g$ and
	$\alpha_1,\ldots,\alpha_n$
	such that, for any $0\leq s<t\leq 1$,
	\begin{align*}
		|I^{\alpha_1\cdots\alpha_n}_{st}(f)|
		\leq
			C
			(t-s)^{r_0+r_1\HolExp}.
	\end{align*}
\end{lemma}

We use the above Taylor expansion and the estimate of iterated integrals in the calculation below.
Using \pref{prop_20140928075311}, we can prove the following by induction on $n$;
\begin{proposition}\label{prop_20141107094620}
	Let $n\geq 0$.
	Assume that $b,\sigma\in\SmoothFunc[bdd]{n+2}{\RealNum}{\RealNum}$.
	Then, for any $0\leq s<t\leq 1$, we have
	\begin{align*}
		x_t-x_s
		&=
			\sum_{k=1}^n
			\sum_{\alpha_1,\dots,\alpha_k\in\{0,1\}}
				\left[
					\vecField{V}_{\alpha_1}
					\cdots
					\vecField{V}_{\alpha_{k-1}}
					V_{\alpha_k}
				\right]
					(x_s)
				g^{\alpha_1\cdots\alpha_k}_{st}\\
		&\phantom{=}\quad\qquad
			+
			\sum_{\alpha_1,\dots,\alpha_n,\alpha_{n+1}\in\{0,1\}}
				I^{\alpha_{1}\alpha_2\cdots\alpha_{n+1}}_{st}
					\left(
						\vecField{V}_{\alpha_{1}}
						\vecField{V}_{\alpha_2}
						\cdots
						\vecField{V}_{\alpha_n}
						V_{\alpha_{n+1}}
					\right).
	\end{align*}
\end{proposition}

We calculate each terms in \pref{prop_20141107094620}.
We first note that the $p$-th iterated integral $g^{\alpha\cdots\alpha}_{st}$
is equal to $(g^\alpha_t-g^\alpha_s)^p/p!$.
This can be checked by a direct calculation.
\begin{proposition}\label{prop_20141010094511}
	Assume that $b,\sigma\in\SmoothFunc[bdd]{6}{\RealNum}{\RealNum}$.
	Then, for any $0\leq s<t\leq 1$, we have
	\begin{align*}
		x_t-x_s
		&=
			b(x_s)(t-s)
			+\sigma(x_s)(g_t-g_s)
			+\frac{1}{2}\left[\sigma\sigma'\right](x_s)(g_t-g_s)^2\\
		&\phantom{=}\quad
			+
			\frac{1}{3!}\left[\sigma(\sigma\sigma')'\right](x_s)(g_t-g_s)^3
			+
			\frac{1}{4!}\left[\sigma(\sigma(\sigma\sigma')')'\right](x_s)(g_t-g_s)^4\\
		&\phantom{=}\quad
			+[b\sigma'](x_s)(g_t-g_s)(t-s)
			+[\sigma b'-b\sigma'](x_s)g^{10}_{st}
			+\frac{1}{2}[b'b](x_s)(t-s)^2\\
		&\phantom{=}\quad
			+
			[b\left(\sigma\sigma'\right)'](x_s)
			g^{011}_{st}
			+
			[\sigma\left(b\sigma'\right)'](x_s)
			g^{101}_{st}
			+
			[\sigma\left(\sigma b'\right)'](x_s)
			g^{110}_{st}
			+
			r_{st},
	\end{align*}
 where $|r_{st}|\leq C_g (t-s)^{\min\{2+\lambda,1+3\lambda,5\lambda\}}$.
\end{proposition}

\begin{proof}
	Set
	\begin{align*}
		J^k_{st}
		&=
			\sum_{\alpha_1,\dots,\alpha_k\in\{0,1\}}
				\left[
					\vecField{V}_{\alpha_1}
					\cdots
					\vecField{V}_{\alpha_{k-1}}
					V_{\alpha_k}
				\right]
					(x_s)
				g^{\alpha_1\cdots\alpha_k}_{st},\\
		\tilde{J}^k_{st}
		&=
			\sum_{\alpha_1,\dots,\alpha_k\in\{0,1\}}
				I^{\alpha_1\cdots\alpha_k}_{st}
					\left(
						\vecField{V}_{\alpha_1}
						\cdots
						\vecField{V}_{\alpha_{k-1}}
						V_{\alpha_k}
					\right).
	\end{align*}
	Then we see
	$
		x_t-x_s
		=
			J^1_{st}
			+\dots
			+J^4_{st}
			+\tilde{J}^5_{st}
	$
	and
	\begin{align*}
		J^1_{st}
		&=
			b(x_s)
			g^0_{st}
			+
			\sigma(x_s)
			g^1_{st},\\
		J^2_{st}
		&=
			[bb'](x_s)
			g^{00}_{st}
			+
			[\sigma b'](x_s)
			g^{10}_{st}
			+
			[b\sigma'](x_s)
			g^{01}_{st}
			+
			[\sigma\sigma'](x_s)
			g^{11}_{st},\\
		J^3_{st}
		&=
			[\sigma(\sigma b')'](x_s)
			g^{110}_{st}
			+
			[\sigma(b\sigma')'](x_s)
			g^{101}_{st}
			+
			[b(\sigma\sigma')'](x_s)
			g^{011}_{st}
			+
			[\sigma(\sigma\sigma')'](x_s)
			g^{111}_{st}
			+
			r^{(3)}_{st},\\
		J^4_{st}
		&=
			[\sigma(\sigma(\sigma\sigma')')'](x_s)
			g^{1111}_{st}
			+
			r^{(4)}_{st},
	\end{align*}
	where $r^{(3)}_{st}$ and $r^{(4)}_{st}$
	satisfy
	$
		|r^{(3)}_{st}|
		\leq
			\const[g]
			(t-s)^{2+\HolExp}
	$
	and
	$
		|r^{(4)}_{st}|
		\leq
			\const[g]
			(t-s)^{1+3\HolExp}
	$, respectively.
	In addition, we have
	$
		|\tilde{J}^5_{st}|
		\leq
			\const[g]
			(t-s)^{5\HolExp}
 $.
 Noting $[\sigma b'](x_s)g^{10}_{s,t}+[b\sigma'](x_s)g^{01}_{st}=
 [b\sigma'](x_s)(g_t-g_s)(t-s)+[\sigma b'-b\sigma'](x_s)g^{10}_{st}$,
we complete the proof.
\end{proof}

\subsection{Directional derivatives of solutions}
In what follows, we assume that \hyporef{hypo_ellipticity} is satisfied
and find expressions of the solution $x\equiv x(g)$ to \eqref{eq_20140928070616} given by \eqref{eq_1545722900}
and its directional derivatives.
We follow the approach employed in \cite{DetempleGarciaRindisbacher2005}
in order to do so.

For $g\in\HolFunc[0]{\HolExp}{[0,1]}{\RealNum}$, we set
\begin{align}\label{eq_20150218112958}
	J_t(g)
	=
		\exp
			\left(
				\int_0^t
					b'(x_u(g))\,
					du
				+
				\int_0^t
					\sigma'(x_u(g))\,
					d^\circ g_u
			\right).
\end{align}
This is a deterministic version of \eqref{eq_20141226020239}.
Note that $J_t(g)$ is expressed by
\begin{align}\label{eq_20141226020204}
	J_t(g)
	=
		\frac{\sigma(x_t(g))}{\sigma(x_0(g))}
		\exp
			\left(
				\int_0^t
					\left[
						\frac{\Wronskian}{\sigma}
					\right]
					(x_u(g))\,
					du
			\right).
\end{align}
Indeed, we see
\begin{align*}
	\log \sigma(x_t(g))
	&=
		\log (\sigma\circ\phi)(a_t(g),g_t)\\
	&=
		\log \sigma(x_0)
		+
		\int_0^t
			\left[
				\frac{\sigma'b}{\sigma}
			\right]
				(x_u(g))\,
			du
		+
		\int_0^t
			\sigma'(x_u(g))\,
			d^\circ g_u
\end{align*}
from \pref{prop_20140928075311}.
This implies
\begin{align*}
	\sigma(x_t(g))
	=
		\sigma(x_0)
		\exp
			\left(
				\int_0^t
					\left[
						\frac{\sigma'b}{\sigma}
					\right]
						(x_u(g))\,
					du
				+
				\int_0^t
					\sigma'(x_u(g))\,
					d^\circ g_u
			\right).
\end{align*}
Substituting the above to \eqref{eq_20141226020204},
we obtain \eqref{eq_20150218112958}.

\begin{proposition}\label{prop_20140929105014}
	Let $b,\sigma\in\SmoothFunc[bdd]{n+1}{\RealNum}{\RealNum}$ for $n\geq 1$.
	Assume that \hyporef{hypo_ellipticity} is satisfied.
	Then, the functional $g\mapsto x_t(g)$ is
	$n$-times Fr\'{e}chet differentiable in $\HolFunc[0]{\HolExp}{[0,1]}{\RealNum}$.

	In particular, the derivatives satisfy the following;
	\begin{enumerate}
		\item	\label{item_20141228072212}
			 	For any $h^1,\dots,h^\nu\in\HolFunc[0]{\HolExp}{[0,1]}{\RealNum}$,
				we have
				\begin{align*}
					|
						\nabla_{h^\nu}
						\cdots
						\nabla_{h^1}
						x_t(g)
					|
					\leq
						\const[\nu]
						\|h^1\|_\infty
						\cdots
						\|h^\nu\|_\infty,
				\end{align*}
				where  $\const[\nu]$ is a positive constant depending only on $b,\sigma$ and $\nu$.
		\item	\label{item_20141106063937}
			The first derivative $\nabla_h x_t(g)$
			is expressed as
				\begin{align*}
					\nabla_h x_t(g)
					&=
						\sigma(x_t(g))h_t
						+
						\int_0^t
							J_t(g)(J_s(g))^{-1}
							\Wronskian(x_s(g))
							h_s\,
							ds.
				\end{align*}
		\item	\label{item_20141106063948}
			If $h$ is Lipschitz continuous,
			then $\nabla_h x_t(g)$ is expressed as
				\begin{align*}
					\nabla_h x_t(g)
					=
						\int_0^t
							\dot{h}_s
							\sigma(x_s(g))
							J_t(g)(J_s(g))^{-1}\,
							ds
					=
						\sigma(x_t(g))
						\int_0^t
							\exp
								\left(
									\int_s^t
										\left[
											\frac{\Wronskian}{\sigma}
										\right]
										(x_u(g))\,
										du
								\right)
							\dot{h}_s\,
							ds.
				\end{align*}
	\end{enumerate}
\end{proposition}

In order to prove \pref{prop_20140929105014}, we set
\begin{align*}
	F(x)
	&=
		\int_0^x
			\frac{d\xi}{\sigma(\xi)},
	&
	G&=F^{-1},
	&
	\tilde{b}
	&=
		\left[
			\frac{b}{\sigma}
		\right]
		\circ
		G,
	&
	y_0
	&=
		F(x_0).
\end{align*}
We consider a solution $y$ to an ODE
\begin{align}\label{eq_20140928074406}
	y_t
	=
		y_0
		+
		\int_0^t
			\tilde{b}(y_u)\,
			du
		+
		g_t.
\end{align}
Then we obtain an expression of the solution $x_t$ to \eqref{eq_20140928070616} as follows;
\begin{proposition}\label{prop_20140928105149}
	Let $y$ be a solution to \eqref{eq_20140928074406}.
	The solution $x$ to \eqref{eq_20140928070616} given by \eqref{eq_1545722900} is expressed by $x=G(y)$.
\end{proposition}
\begin{proof}
	Due to \pref{prop_20140928080940}, we see the assertion
	by showing $G(y)\in\mathfrak{C}$ and it satisfies \eqref{eq_20140928070616}.
	Note that the solution $y$ is given by
	$
		y_t=\tilde{a}_t+g_t
	$,
	where $\tilde{a}$ is a solution to
	$
		\tilde{a}_t
		=
			y_0
			+
			\int_0^t
				\tilde{b}(\tilde{a}_u+g_u)\,
				du
	$.
	Hence $G(y)\in\mathfrak{C}$.
	We prove that $G(y)$ satisfies \eqref{eq_20140928070616}.
	From \pref{prop_20140928075311}, we see
	\begin{align*}
		G(y_t)-x_0
		&=
			G(\tilde{a}_t+g_t)-G(\tilde{a}_0+g_0)\\
		&=
			\int_0^t
				G'(\tilde{a}_u+g_u)\,
				d\tilde{a}_u
			+
			\int_0^t
				G'(\tilde{a}_u+g_u)\,
				d^\circ g_u.
	\end{align*}
	The first term is equal to
	\begin{align*}
		\int_0^t
			\sigma(G(\tilde{a}_u+g_u))
			\tilde{b}(\tilde{a}_u+g_u)\,
			du
		&=
		\int_0^t
			\sigma(G(\tilde{a}_u+g_u))
			\left[
				\frac{b}{\sigma}
			\right]
			(G(\tilde{a}_u+g_u))\,
			du\\
		&=
		\int_0^t
			b(G(y_u))\,
			du
	\end{align*}
	and the second one is
	$
		\int_0^t
			\sigma(G(y_u))\,
			d^\circ g_u
	$.
	We see that $G(y)$ satisfies \eqref{eq_20140928070616}.
	The proof is completed.
\end{proof}
We see that the solution $y_t$ to \eqref{eq_20140928074406}
with any coefficient $\tilde{b}$ and initial point $y_0$ is differentiable.
\begin{proposition}\label{prop_20141030154436}
	Assume that $\tilde{b}\in\SmoothFunc[bdd]{n+1}{\RealNum}{\RealNum}$ for $n\geq 1$.
	The functional $g\mapsto y_t(g)$	is
	$n$-times Fr\'{e}chet differentiable in $\HolFunc[0]{\HolExp}{[0,1]}{\RealNum}$.

	In particular, the derivatives satisfy the following;
	\begin{enumerate}
		\item	For any $h^1,\dots,h^\nu\in\HolFunc[0]{\HolExp}{[0,1]}{\RealNum}$,
				we have
				\begin{align*}
					|
						\nabla_{h^\nu}
						\cdots
						\nabla_{h^1}
						y_t(g)
					|
					\leq
						\const[\nu]
						\|h^1\|_\infty
						\cdots
						\|h^\nu\|_\infty
				\end{align*}
				where $\const[\nu]$ is a positive constant depending only on $\tilde{b}$ and $\nu$.
		\item	\label{item_1505962351}The first derivative $\nabla_h y_t(g)$ is expressed by
				\begin{align*}
					\nabla_h y_t(g)
					=
						h_t
						+
						\int_0^t
							\exp
								\left(
									\int_s^t
										\tilde{b}'(y_u(g))\,
										du
								\right)
							\tilde{b}'(y_s(g))
							h_s\,
							ds.
				\end{align*}
	\end{enumerate}
\end{proposition}

For the sake of conciseness,
we omit the proof of the above proposition and show \pref{prop_20140929105014}.
\begin{proof}[Proof of \pref{prop_20140929105014}]
	The differentiability and Assertion~(\ref{item_20141228072212})
	follow from \prefs{prop_20140928105149}{prop_20141030154436}.
	Noting $\tilde{b}'(y_t(g))=[\Wronskian/\sigma](x_t(g))$,
	we see that Assertion~(\ref{item_20141106063937}) is true.
	Assertion (\ref{item_20141106063948}) follows from Assertion~(\ref{item_20141106063937})
	and the integration by parts formula.
\end{proof}

\subsection{SDEs driven by fBm}\label{sec_201708051641}
We consider existence and properties of a solution to an SDE
\eqref{eq_20141122070455}.
Let us start our discussion with the definition of fBm;
\begin{definition}
	A one-dimensional centered Gaussian process $B=\{B_t\}_{0\leq t<\infty}$ starting from zero is called
	fractional Brownian motion (fBm) with the Hurst $0<\Hurst<1$ if its covariance is given by
	\begin{align}\label{eq_20150219103359}
		\expect[B_sB_t]
		=
			R(s,t)
		=
			\frac{1}{2}
			\left\{
				s^{2\Hurst}
				+
				t^{2\Hurst}
				-
				|t-s|^{2\Hurst}
			\right\}.
	\end{align}
\end{definition}
It is well known that fBm $B$ has stationary increments in the sense of
$
	\expect
		[(B_t-B_s)(B_v-B_u)]
	=
		\expect
			[(B_{t+a}-B_{s+a})(B_{v+a}-B_{u+a})]
$
for any $0\leq s\leq t\leq u\leq v<\infty$ and $0\leq a<\infty$
and that it has self-similarity, namely,
for any $a>0$, $\{a^{-\Hurst} B_{at}\}_{0\leq t<\infty}$ is also fBm with the Hurst $\Hurst$.
In addition, it has a modulus of continuity of trajectories;
there exists a measurable subset $\probSp_0$ of $\probSp$ such that $\prob(\probSp_0)=1$
and for any $0<\epsilon<\Hurst$, there exists a nonnegative random variable
$G_\epsilon$ such that
$\expect[G_\epsilon^p]<\infty$ for any $p\geq 1$ and
\begin{align}\label{eq_20141224083943}
	|B_t(\omega)-B_s(\omega)|
	\leq
		G_\epsilon(\omega)
		|t-s|^{\Hurst-\epsilon}
\end{align}
for any $0\leq s,t<\infty$ and $\omega\in\probSp_0$.

Assume that $1/3<\Hurst<1$.
From \pref{prop_20140928080940}
and the H\"{o}lder continuity of fBm \eqref{eq_20141224083943},
we see existence of a unique solution to the SDE \eqref{eq_20141122070455}
in the pathwise sense.
More precisely, since $B(\omega)$ for any
$\omega\in\probSp_0$ is $(\Hurst-\epsilon)$-H\"{o}lder continuous,
a solution $X$ to \eqref{eq_20141122070455} is give by \eqref{eq_1545722900}
and it is unique in sense of \pref{prop_20140928080940}.
In the same way as $x$, we shall also write $X(\xi)$, $X(B)$, or $X(\xi,B)$
to emphasize dependence on the initial value $\xi$ and/or the driver $B$.
\begin{proposition}\label{prop_20150520004508}
	Assume that $b\in\SmoothFunc[bdd]{1}{\RealNum}{\RealNum}$
	and $\sigma\in\SmoothFunc[bdd]{2}{\RealNum}{\RealNum}$.
	Then there exists a unique solution $X$ to \eqref{eq_20141122070455}
	and the following are satisfied:
	\begin{enumerate}
		\item	\label{item_20150327035831}
				$X$ is adapted to the fBm filtration $\{\sigmaField_{t}\}_{0\leq t\leq1}$,
				where
				$
					\sigmaField_{t}
					=
						\sigma(B_u;0\leq u\leq t)
				$,
		\item	\label{item_20141102234451}
				$t\mapsto X_t$ is $(\Hurst-\epsilon)$-H\"{o}lder continuous a.s.\ for every $0<\epsilon<\Hurst$,
		\item	\label{item_20141103222159}
				for any $r\geq 1$, there exists a positive constant $\const$ such that
				\begin{align*}
					\expect[|X_t-X_s|^r]^{1/r}
					\leq
						\const
						(t-s)^\Hurst
				\end{align*}
				for any $0\leq s<t\leq 1$.
	\end{enumerate}
\end{proposition}
\begin{proof}
	The first assertion follows from \pref{prop_20140928080940}.
	We show the second and third assertion.
	We decompose $X_t-X_s$ into
	$
		\{\phi(a^B_t,B_t)-\phi(a^B_s,B_t)\}
		+\{\phi(a^B_s,B_t)-\phi(a^B_s,B_s)\}
	$.
	From \prefs{prop_20140929101727}{prop_20140928080940}, we have
	\begin{align*}
		|\phi(a^B_t,B_t)-\phi(a^B_s,B_t)|
		&\leq
			e^{c_1|B_t|}
			\int_s^t
				c_2e^{c_3|B_u|}\,
				du,\\
		|\phi(a^B_s,B_t)-\phi(a^B_s,B_s)|
		&\leq
			c_4|B_t-B_s|,
	\end{align*}
	where $c_1$, $c_2$, $c_3$, $c_4$ are positive constants.
	The proof is completed.
\end{proof}

\section{Convergence of variation functionals}\label{sec_20150519021523}
Let $B=\{B_t\}_{0\leq t\leq 1}$ be an fBm with the Hurst $1/3<\Hurst<1$
and $X=\{X_t\}_{0\leq t\leq 1}$ the solution to \eqref{eq_20141122070455} given by \eqref{eq_1545722900}.
We assume that $b,\sigma\in\SmoothFunc[bdd]{\infty}{\RealNum}{\RealNum}$.
For these processes, we define the weighted Hermite variations and
the trapezoidal error variations.
The purpose of this section is to present
necessary results for asymptotics of the variations.

Let $f\in\SmoothFunc[poly]{2q}{\RealNum}{\RealNum}$ for $q\geq 2$
and $g\in\SmoothFunc[poly]{2}{\RealNum}{\RealNum}$.
Let $\mu$ be a probability measure on $[0,1]$.
For every $0\leq s<t\leq 1$ and continuous path $x:[0,1]\to\RealNum$, define
\begin{align*}
	F_{st}(x)
	\equiv
		F^{f,\mu}_{st}(x)
	:=
		\int_0^1
			f(\theta x_t+(1-\theta)x_s)\,
			\mu(d\theta).
\end{align*}
We define the weighted Hermite variations
$
	\wHerVar{q}{m}(t)
	\equiv
		\wHerVar{q,f,\mu}{m}(t)
$
by
\begin{align*}
	\wHerVar{q}{m}(t)
	=
		\sum_{k=1}^{\intPart{2^m t}}
			F_{\dyadicPart[m]{k-1}\dyadicPart[m]{k}}(X)
			\hermitePoly{q}(2^{m\Hurst}B_{\dyadicPart[m]{k-1}\dyadicPart[m]{k}})
\end{align*}
and
the trapezoidal error variations
$
	\wTRVar{}{m}(t)
	\equiv
		\wTRVar{g}{m}(t)
$
by
\begin{align*}
	\wTRVar{}{m}(t)
	=
		\sum_{k=1}^{\intPart{2^m t}}
			g(X_{\dyadicPart[m]{k-1}})
			\left(
				\frac{1}{2\cdot 2^m}
				B_{\dyadicPart[m]{k-1}\dyadicPart[m]{k}}
				-
				\int_{\dyadicPart[m]{k-1}}^{\dyadicPart[m]{k}}
					B_{\dyadicPart[m]{k-1}u}\,
					du
			\right).
\end{align*}
Here, $B_{st}=B_t-B_s$ for $0\leq s<t\leq 1$ and $H_q$ is the $q$-th Hermite polynomial defined by
\begin{align*}
	\hermitePoly{q}(\xi)
	=
		(-1)^q
		e^{\xi^2/2}
		\frac{d^q}{d\xi^q}
			e^{-\xi^2/2}.
\end{align*}
The first few Hermite polynomials are
$\hermitePoly{1}(\xi)=\xi$,
$\hermitePoly{2}(\xi)=\xi^2-1$,
$\hermitePoly{3}(\xi)=\xi^3-3\xi$,
and $\hermitePoly{4}(\xi)=\xi^4-6\xi^2+3$.
We set $\hermitePoly{0}(\xi)=1$ by convention.

The following limit theorems are vital for our proof.
These results are proved in Appendixes~\ref{sec_20150107074711} and \ref{sec_20170519021523}.

\begin{theorem}\label{thm_20141203094309}
	Let $q\geq 2$ be even.
	We have
	\begin{align*}
		\lim_{m\to\infty}
			2^{m(q\Hurst-1)}
			\sum_{k=1}^{\intPart{2^m\cdot}}
				F_{\dyadicPart[m]{k-1}\dyadicPart[m]{k}}(X)
				(B_{\dyadicPart[m]{k-1}\dyadicPart[m]{k}})^q
		=
			\expect[Z^q]
			\int_0^\cdot
				f(X_s)\,
				ds
	\end{align*}
	in probability with respect to the uniform norm.
	Here $Z$ is a standard Gaussian random variable.
\end{theorem}

\begin{theorem}\label{thm_20141117071529}
	Let $q\geq 2$ and $1/2q<\Hurst<1-1/2q$.
	We have
	\begin{align*}
		\lim_{m\to\infty}
			\left(
				B,2^{-m/2}\wHerVar{q}{m}
			\right)
		=
			\left(
				B,
				\sigma_{q,\Hurst}
				\int_0^\cdot
					f(X_s)\,
					dW_s
			\right)
	\end{align*}
	weakly in the Skorokhod topology,
	where
	$\sigma_{q,\Hurst}$ is a constant defined by \eqref{eq_const_sigma_q_Hurst}
	and $W$ is a standard Brownian motion independent of $B$.
\end{theorem}

\begin{theorem}\label{thm_20141123062022}
	Let $q\geq 2$ and $\Hurst=1/2$.
	We have
	\begin{align*}
		\lim_{m\to\infty}
			\left(
				B,2^{-m/2}\wHerVar{q}{m},2^{m}\wTRVar{}{m}
			\right)
		=
			\left(
				B,
				\sqrt{q!}
				\int_0^\cdot
					f(X_s)\,
					dW_s,
				\frac{1}{\sqrt{12}}
				\int_0^\cdot
					g(X_s)\,
					d\tilde{W}_s
			\right)
	\end{align*}
	weakly in the Skorokhod topology,
	where
	$W$ and $\tilde{W}$ are standard Brownian motions
	and $B$, $W$ and $\tilde{W}$ are independent.
\end{theorem}

\begin{proposition}\label{prop_20141117062354}
	If $0<\Hurst<1/2$ (resp. $1/2\leq\Hurst<1$),
	then the process $2^{mr}\wTRVar{}{m}$ for $0<r<2\Hurst$ (resp. $0<r<1$)
	converges to the process $0$ in probability
	with respect to the uniform norm.
\end{proposition}

In order to prove \tref[thm_20141117071529]{thm_20141123062022},
we use a simplified version of them.
Let $q\geq 2$. We set
\begin{align*}
	\hermiteVar{q}{m}(t)
	&=
		2^{-m/2}
		\sum_{k=1}^{\intPart{2^m t}}
			\hermitePoly{q}
				(
					2^{m\Hurst} B_{\dyadicPart[m]{k-1}\dyadicPart[m]{k}}
				)
\end{align*}
and
\begin{align*}
	\TRVar{m}(t)
	&=
		2^{-m/2}
		\sum_{k=1}^{\intPart{2^m t}}
			2^{m(\Hurst+1)}
			\left(
				\frac{1}{2\cdot 2^m}
				B_{\dyadicPart[m]{k-1}\dyadicPart[m]{k}}
				-
				\int_{\dyadicPart[m]{k-1}}^{\dyadicPart[m]{k}}
					B_{\dyadicPart[m]{k-1}u}\,
					du
			\right).
\end{align*}
Then, we see
$
	\hermiteVar{q}{m}
	=
		2^{-m/2}
		\wHerVar{q,f,\mu}{m}
$
and
$
	\TRVar{m}
	=
		2^{m(\Hurst+1/2)}
		\wTRVar{g}{m}
$
for $f=g\equiv 1$
and the following:
\begin{proposition}\label{prop_20141102222020}
	Assume $q\geq 2$ and $0<\Hurst<1-1/2q$. Then we have
	\begin{align*}
		\lim_{m\to\infty}
		(
			B,\hermiteVar{q}{m},\TRVar{m}
		)
		=
			(
				B,\sigma_{q,\Hurst}W,\tilde{\sigma}_\Hurst \tilde{W}
			)
	\end{align*}
	weakly in the Skorokhod topology.
	Here $W$ and $\tilde{W}$ are independent standard Brownian motions independent of $B$,
	and $\sigma_{q,\Hurst}$ and $\sigma_\Hurst$ are positive constants given by
	\begin{align}
		\label{eq_const_sigma_q_Hurst}
		\sigma_{q,\Hurst}^2
		&=
			q!
			\left(
				1+2\sum_{l=1}^\infty \rho_\Hurst(l)^q
			\right),\\
		\notag
		\tilde{\sigma}_\Hurst^2
		&=
			\frac{1}{4}
			\frac{1-\Hurst}{1+\Hurst}
			+
			2
			\sum_{l=1}^\infty
				\tilde{\rho}_\Hurst(l)
	\end{align}
	with
	\begin{gather*}
		\rho_\Hurst(l)
		=
			\expect
				[B_1(B_{l+1}-B_l)]
		=
			\frac{1}{2}
			(
				|l+1|^{2\Hurst}
				+
				|l-1|^{2\Hurst}
				-
				2|l|^{2\Hurst}
			),\\
		\tilde{\rho}_\Hurst(l)
		=
			\expect
				\left[
					\left(
						\frac{1}{2}
						B_1
						-
						\int_0^1
							B_u\,
							du
					\right)
					\left(
						\frac{1}{2}
						(B_{l+1}-B_l)
						-
						\int_l^{l+1}
							(B_u-B_l)\,
							du
					\right)
				\right].
	\end{gather*}
\end{proposition}

We close this section with making remarks on results above:
\begin{remark}\label{rem_201708101002}
	\begin{enumerate}
		\item	In Appendix~\ref{sec_20150107074711}, we show \pref{prop_20141102222020}
				by showing relative compactness (\lref{lem_20150519013003})
				and convergence in the sense of finite-dimensional distributions (\lref{lem_20150519013014}).
				In the proof of \lref{lem_20150519013014}, we show independence of $B$, $W$ and $\tilde{W}$
				by using the multidimensional fourth moment theorem
				by Peccati and Tudor \cite{PeccatiTudor2005}.

		\item	In Appendix~\ref{sec_20170519021523},
				we show \trefs{thm_20141203094309}{thm_20141117071529}{thm_20141123062022}
				and \pref{prop_20141102222020}.
				In order to prove \trefs{thm_20141203094309}{thm_20141117071529}{thm_20141123062022},
				we use good properties of the solution $X$:
				for example, the continuity of the solution map $B\mapsto X$, the continuity of the map $t\mapsto X_t$
				and Malliavin differentiability of $X_t$.
				In addition, \pref{prop_20141102222020} is essential for \tref[thm_20141117071529]{thm_20141123062022}.
				Since \pref{prop_20141102222020} is a consequence of the fourth moment theorem,
				these theorems also one of it.

		\item	\tref[thm_20141203094309]{thm_20141117071529} are a slight extensions of
				\cite[Theorem~2.1]{GradinaruNourdin2009}, \cite[Theorem~1]{NourdinNualartTudor2010} and \cite[Theorem~15]{Naganuma2015}.
				In these references, the authors showed convergences of the weighted Hermite variations $\wHerVar{q}{m}$
				in which $F_{\dyadicPart[m]{k-1}\dyadicPart[m]{k}}(X)$ are replaced by
				$f(B_{\dyadicPart[m]{k-1}})$ or $F_{\dyadicPart[m]{k-1}\dyadicPart[m]{k}}(B)$,
				that is, they considered functionals which are expressed by fBm $B$ explicitly.
				On the other hand, we consider functionals of the solution $X$ to \eqref{eq_20141122070455}
				in \tref[thm_20141203094309]{thm_20141117071529}.
				\tref{thm_20141123062022} is an exention of \tref{thm_20141117071529} in the case $\Hurst=1/2$.

		\item	Since a standard Brownian motion has independent increments,
				we see $\rho_{1/2}(l)=0$ and $\tilde{\rho}_{1/2}(l)=0$ for $l\geq 1$.
				Hence we have $\sigma_{q,1/2}=\sqrt{q!}$ and $\sigma_{1/2}=1/\sqrt{12}$.
	\end{enumerate}
\end{remark}

\section{The Crank-Nicolson scheme}\label{sec_201707212021}
In this section, we show \tref{thm_20141204015044}.
Below, we fix sufficiently small $0<\epsilon<H$ and write $H^{-}=H-\epsilon$.
For $m\in\NaturalNum$, we may write  $\Delta=2^{-m}$,
$\Delta B_k=B_{\tau^m_{k-1}\tau^m_{k}}$~$(1\le k\le 2^m)$,
$\Delta(\Delta B_k)^n=\Delta\cdot(\Delta B_k)^n$ ($n=1,2,\dots$)
and $\Delta(\Delta B_k)=\Delta(\Delta B_k)^1$.
We use the notation $B^{\boldsymbol{i}}_{st}$ $(\boldsymbol{i}=10, 01, 011, 101, 110)$
to denotes the iterated integral introduced in \secref{sec_20141211055401}.
We denote by $O(\Delta^p)$ the term which is less than or equal to $C \Delta^p$,
where $C$ does not depend on $m$ and $\xi$.

\subsection{Well-definedness of the Crank-Nicolson scheme}
Since the Crank-Nicolson scheme is an implicit scheme,
we need to define the set on which the scheme can be defined.
Recall that
$(\probSp,\sigmaField,\prob)$ denotes the canonical probability
space which defines fBm $B(\omega)$ with the Hurst parameter $H$
and
\begin{align*}
	\Omega_0
	&=
		\bigcap_{0<\epsilon<H}
			\{
				\omega\in \probSp; B(\omega)
				\in
				\HolFunc[0]{H-\epsilon}{[0,1]}{\RealNum}
			\}.
\end{align*}
For every $m\in\NaturalNum$, we define
\begin{align*}
	 \CN
	 =
	 	\Omega_0
		\cap
	 	\left\{
			\omega\in \probSp ;
			\sup_{|t-s|\le 2^{-m}}
	 			\frac{|B_t(\omega)-B_s(\omega)|}{(t-s)^{H-\epsilon}}
			\le
				1
	 	\right\}.
\end{align*}
Note that $\CN\subset \Omega^{{\rm CN}(m+1)}$
for any $m$ and $\lim_{m\to\infty}P(\CN)=1$
for the fBm with the Hurst parameter $H$.
We show that the Crank-Nicolson scheme is
defined on $\CN$ for large $m$.

\begin{proposition}\label{prop_20141111045543}
	Suppose
	\begin{align}\label{eq_201708101103}
		m
		>
			\max
			\left\{
				1+\log_2(\sup |b'|),
					\frac{1+\log_2(\sup|\sigma'|)}{H-\epsilon}
			\right\}.
	\end{align}
	Let $0<s<t<1$ satisfy $|t-s|\le 2^{-m}$.
	Then for any $\xi\in\RealNum$ and $\omega\in \CN$, there exists a
	unique $\eta_t$ satisfying
	\begin{align*}
		\eta_t
		=
			\xi
			+\frac{b(\xi)+b(\eta_t)}{2}(t-s)
			+\frac{\sigma(\xi)+\sigma(\eta_t)}{2}(B_t(\omega)-B_s(\omega)).
	\end{align*}
\end{proposition}
\begin{proof}
	Set
	\begin{align*}
		F(\xi,\delta,\Delta;\eta)
		=
			\eta
			-
			\left[
				\xi
				+
				\frac{1}{2}
				\left\{
					b(\xi)
					+
					b(\eta)
				\right\}
				\delta
				+
				\frac{1}{2}
				\left\{
					\sigma(\xi)
					+
					\sigma(\eta)
				\right\}
				\Delta
			\right].
	\end{align*}
 If $|\delta|<1/(2\sup|b'|)$ and $|\Delta|<
 1/(2\sup|\sigma'|)$,
	then
	$
		[\partial F/\partial\eta](\xi,\delta,\Delta;\eta)
		=
			1
			-
			\{
				(1/2)
				b'(\eta)
				\delta
				+
				(1/2)
				\sigma'(\eta)
				\Delta
			\}
	$
	satisfies
	\begin{align*}
		\frac{\partial F}{\partial\eta}(\xi,\delta,\Delta;\eta)
		\geq
			1
			-
			\frac{1}{2}
			|b'(\eta)|
			|\delta|
			-
			\frac{1}{2}
			|\sigma'(\eta)|
			|\Delta|
		\geq
			\frac{1}{2},
	\end{align*}
	which implies that
	$\eta\mapsto F(\xi,\delta,\Delta;\eta)$ is strictly increasing.
	Hence there exists a unique value $f(\xi,\delta,\Delta)$ such that
	$F(\xi,\delta,\Delta;f(\xi,\delta,\Delta))=0$
	and $f(\xi,0,0)=\xi$.

	Under the assumption on $m$ and $s,t$, it holds that
	$t-s<1/(2\sup|b'|)$ and $|B_t(\omega)-B_s(\omega)|<1/(2\sup|\sigma'|)$
	~$(\omega\in \CN)$.
	Hence $\eta_t$ is uniquely defined as
	$
		\eta_t
		=f(\xi,t-s,B_t(\omega)-B_s(\omega)).
	$
	\end{proof}

	\begin{remark}\label{rem_201708101101}
	 Clearly, the implicit function
	 $f(\xi,\delta,\Delta)$ $(\xi\in \RealNum, |\delta|<1/(2\sup|b'|), |\Delta|<1/(2\sup|\sigma'|)$
	 is a $C^{\infty}$ function.
	\end{remark}

\subsection{Proof of \tref{thm_20141204015044}}\label{sec_1506485494}

The Crank-Nicolson approximation solution $\bar{X}^{(m)}$ can be defined
on $\CN$ for $m$ in \eqref{eq_201708101103}.
From now on, we assume $m$ satisifes \eqref{eq_201708101103}.
For $\omega\notin \CN$, we always set
$\bar{X}^{(m)}_t(\xi,B)\equiv \xi$.

To study the error $\bar{X}^{(m)}-X$,
we prove that there exists a piecewise linear path
$h$ such that $X_{\dyadicPart[m]{k}}(\xi,B+h)=\bar{X}^{(m)}_{\dyadicPart[m]{k}}(\xi,B)$ for all $0\leq k\leq 2^m$.
Let $h$ be a piecewise linear path defined on $[0,1]$ with $h_0=0$ whose partition points
are dyadic points $\{\tau^m_k\}_{k=0}^{2^m}$.
Then $h$ can be identified with the set of values
at the partition points $\{h(\dyadicPart[k]{k})\}_{k=1}^{2^m}$.
We write $\kappa_k=h(\tau^m_k)-h(\tau^m_{k-1})$ $(1\le k\le 2^m)$.

\begin{lemma}\label{lem_201707240936}
	Let $\omega\in \Omega_0$.
	Then there exist unique $\kappa_k\in \RealNum$~$(1\le k\le 2^m)$
	such that
	\begin{align*}
		\bar{X}^{(m)}_{\dyadicPart[m]{k}}(\xi,B)
		=
			X_{\dyadicPart[m]{k}}(\xi,B+h),
		\qquad 1\le k\le 2^m.
	\end{align*}
\end{lemma}
We denote the above $h$ by $h^{(m)}$.
Although $\kappa_k$ depends on $m$ similarly,
we use the same notation $\kappa_k$ for simplicity.
$h^{(m)}(\omega)$ is defined for all $\omega\in \Omega_0$.
Of course, the definition of $\bar{X}^{(m)}$ on
$\Omega_0\setminus \CN$ is essentially meaningless
and the behavior of $h^{(m)}$ on $\Omega_0\setminus \CN$
has nothing to do with the asymptotics
of the error.
Before proving the existence of $h^{(m)}$,
we give a rough sketch how to prove \tref{thm_20141204015044}
by using $h^{(m)}$.

\begin{remark}[Rough sketch of the proof of \tref{thm_20141204015044}]\label{rem_201708102135}
	We decompose $h^{(m)}$ as $h^{(m)}=h_M^{(m)}+h_{R}^{(m)}$.
	Here, $h_M^{(m)}$ is the main term
	and we see
	\begin{align}\label{eq_1506055600}
		\lim_{m\to\infty}
			2^{m(3H-\frac{1}{2})}
			h_M^{(m)}
		=
			U
		\qquad
		\text{in law},
	\end{align}
	where $U$ is a random variable.
	The term $h_{R}^{(m)}$ is the remainder term satisfying that for small $\delta>0$,
	\begin{align}\label{eq_1506055256}
		\lim_{m\to\infty}
			2^{m(3H-\frac{1}{2}+\delta)}
			\|h_R^{(m)}\|_{\infty}
		=
			0
		\qquad \text{in probability}
	\end{align}
	By using the derivative of $X(\xi,B)$ with respect to $B$,
	we have
	$
		2^{m(3H-\frac{1}{2})}
		\{\bar{X}^{(m)}(\xi,B)-X(\xi,B)\}
		=
			I_1+I_2+I_3
	$,
	where
	\begin{align*}
		I_1
		&=
			\nabla_{(2^{m})^{3H-\frac{1}{2}}h^{(m)}_M}X(\xi,B),\\
		I_2
		&=
			(2^m)^{3H-\frac{1}{2}}
			\left\{\bar{X}^{(m)}(\xi,B)-X(\xi,B+h^{(m)}_M)\right\},\\
		I_3
		&=
			(2^m)^{3H-\frac{1}{2}}
			\left\{
				X(\xi,B+h^{(m)}_M)
				-X(\xi,B)
				-\nabla_{h^{(m)}_M}X(\xi,B)
			\right\}.
	\end{align*}
	By the convergence $2^{m(3H-\frac{1}{2})}h_M^{(m)}\to U$ in law,
	we have $I_1=\nabla_{2^{m(3H-\frac{1}{2})}h^{(m)}_M}X(\xi,B)\to
	\nabla_U X(\xi,B)$ in law.
	Since
	\begin{align*}
		I_2
		\thickapprox
			2^{m(3H-\frac{1}{2})}
			\left\{X(\xi,B+h^{(m)})-X(\xi,B+h^{(m)}_M)\right\}
		\thickapprox
			\nabla_{2^{m(3H-\frac{1}{2})}h^{(m)}_R}X(\xi,B),
	\end{align*}
	the middle term converges to 0 in probability.
	For the third term, considering the second derivative,
	we have
	\begin{align*}
		I_3
		\thickapprox
			2^{m(3H-\frac{1}{2})}
			\frac{1}{2}
			\nabla_{h^{(m)}_M}^2X(\xi,B).
	\end{align*}
	Therefore
	this term also converges to $0$ in probability because
	$h^{(m)}_M$ is of order $2^{-m(3H-\frac{1}{2})}$.
	In the following, $h^{(m)}_M$ and $h^{(m)}_R$ are
	piecewise linear paths corresponding
	to $\{\tilde{\kappa}_k\}$ and $\{R_k(\omega)\}$ in \lref{lem_201708091518}.

	We conclude this remark by making a comment on \eqref{eq_1506055600} and \eqref{eq_1506055256}.
	The convergence \eqref{eq_1506055600} of the main term is shown by \tref{thm_20141117071529} and so on in \lref{lem_201708091528}.
	By using this result, we see the convergence \eqref{eq_1506055256} of the remainder in \lref{lem_201708091518}.
	We should mention that the method used in \lref{lem_201708091518} makes estimate of the remainder simpler drastically
	than that of \cite{Naganuma2015}.
\end{remark}

We now prove the existence of $h^{(m)}$.
To this end, we need the bijectivity of
the map $\kappa\mapsto X_t(\xi,B+\kappa\ell)$
which follows from the following lemma.
Here $\ell_t=t$.
This lemma is an immediate consequence of
\pref{prop_20140929105014}~(\ref{item_20141106063948}).

\begin{lemma}\label{lem_201707251240}
	There exist positive numbers $C_1, C_2$ which are independent of
	 $B, \xi, t$ such that
	\begin{align*}
		C_1t
		\leq
			\frac{d}{d\kappa}X_t(\xi,B+\kappa\ell)
		\leq
			C_2t.
	\end{align*}
	In particular, the mapping $\RealNum\ni\kappa\mapsto X_t(\xi,B+\kappa\ell)$ is bijection on $\RealNum$.
\end{lemma}

We prove \lref{lem_201707240936}.
We write $\xi_k=\bar{X}^{(m)}_{\tau^m_k}(\xi,B)$.

\begin{proof}[Proof of \lref{lem_201707240936}]
	We prove this by an induction on $k$.
	Let $k=1$.
	It suffices to prove the existence $\kappa_1$ satisfying
	$\xi_1=X_{2^{-m}}(\xi,B+2^m\kappa_1\ell)$.
	Since $\kappa\mapsto X_{2^{-m}}(\xi,B+2^m\kappa\ell)$
	is a bijective mapping, $\kappa_1$ is uniquely determined.
	Suppose the equality holds upto $k$.
	Noting
	$\xi_{k+1}=X_{\tau^m_{k+1}}(\xi,B+h)$ is equivalent to
	$\xi_{k+1}=X_{2^{-m}}(\xi_k,\theta_{\tau^m_k}B+2^m\kappa_{k+1}\ell)$ and by applying \lref{lem_201707251240},
	the proof is completed.
\end{proof}

In the rest of this subsection, we state some key lemmas (\lrefs[lem_201707121721]{lem_201708091518}{lem_201708091528})
for \tref{thm_20141204015044} and show the theorem.
The key lemmas is shown in the next subsection.
In these lemmas, we calculate $\kappa_k$ and determine the main term of the error.
By the definition, $\kappa_{k}$ $(1\le k\le 2^m)$ satisfies the equation
\begin{align}\label{eq_1545356684}
	X_{2^{-m}}(\xi_{k-1},\theta_{\tau^m_{k-1}}B+2^m\kappa_{k}\ell)
	-X_{2^{-m}}(\xi_{k-1},\theta_{\tau^m_{k-1}}B)
	=
		\{
			\xi_{k}-\xi_{k-1}
		\}
		-
		\{
			X_{2^{-m}}(\xi_{k-1},\theta_{\tau^m_{k-1}}B)
			-
			\xi_{k-1}
		\}.
\end{align}
We set $\hat{\kappa}_{k}$ by the left-hand side of the above equality.
The quantity $\hat{\kappa}_k$ is the 1-step error of the Crank-Nicolson scheme.
We calculate $\hat{\kappa}_k$ and $\kappa_k$ with small remainder terms.
By this calculation and the H\"older continuity of $B$, we see
that $\max_{1\le k\le 2^m}|\bar{X}^{(m)}_{\tau^m_k}-X_{\tau^m_k}|$ converges to $0$ if $H>\frac{1}{3}$ (\lref{lem_201707121721}).
This is a rough estimate.
We improve it later by identifying the main term of the error (\lref{lem_201708091518}).

In order to express $\hat{\kappa}_k$, we introduce
\begin{align*}
	\hat{f}_3
	&=
		\frac{1}{12}
		[
			\sigma^2\sigma''+\sigma (\sigma')^2
		], \quad
	\hat{f}_4
	=
		\frac{1}{24}
		[
			\sigma^3\sigma'''
			+5\sigma^2\sigma'\sigma''
			+2\sigma(\sigma')^3
		],\quad
	\hat{g}_1
	=
		\Wronskian,\\
	\hat{\varphi}
	&=
		\frac{1}{4}
		\left[b(\sigma')^2+\sigma^2 b''\right]
		+
		\frac{1}{2}
		[
			b\sigma\sigma''+\sigma\sigma'b'
		],\\
	\hat{\varphi}_{011}
	&=
		-b(\sigma\sigma')', \quad
	\hat{\varphi}_{101}
	=
		-\sigma (b\sigma')',\quad
	\hat{\varphi}_{110}
	=
		-\sigma(\sigma b')'.
\end{align*}
Here, we recall $\Wronskian=\sigma b'-\sigma'b$.
We also see that the main term of $\kappa_k$ is expressed by the following functions:
\begin{align*}
	f_3
	&=
		\frac{1}{12}
		[\sigma\sigma''+(\sigma')^2], \quad
	f_4
	=
		\frac{1}{24}
		\sigma
		(\sigma\sigma'''+ 3\sigma'\sigma''), \quad
	g_1
	=
		\frac{\Wronskian}{\sigma},\\
	\varphi
	&=
		\frac{1}{4}
		\left[
			\frac{b(\sigma')^2}{\sigma}
			+\sigma b''
		\right]
		+\frac{1}{2}(b\sigma''+\sigma'b'),\\
	\varphi_{011}
	&=
		-\frac{b(\sigma\sigma')'}{\sigma}, \quad
	\varphi_{101}
	=
		-(b\sigma')',\quad
	\varphi_{110}
	=
		-(\sigma b')'.
\end{align*}
Note that $f_4=(\hat{f}_4-\sigma'\hat{f}_3)/\sigma$
and that $h=\hat{h}/\sigma$ for $h=f_3,g_1,\phi,\phi_{011},\phi_{101},\phi_{110}$.
By a simple calculation, we have $f_4=\sigma f_3'/2$.
This identity is a key for the convergence of the main term of the error
similarly to the case where $b\equiv 0$
(\cite{NeuenkirchNourdin2007,Naganuma2015}); see \lref{lem_201708091528}.

The expression of $\hat{\kappa}_{k}$ and the convergence of $\max_{1\le k\le 2^m}|\bar{X}^{(m)}_{\tau^m_k}-X_{\tau^m_k}|$
are obtained as follows:
\begin{lemma}\label{lem_201707121721}
	For any $\omega\in \CN$, the following hold.
	\begin{enumerate}
		\item	\label{item_1505965509}
				We have
				\begin{align*}
					\hat{\kappa}_{k}
					&=
						\hat{f}_3(\xi_{k-1})(\Delta B_k)^3
						+\hat{f}_4(\xi_{k-1})(\Delta B_k)^4
						+
							\hat{g}_1(\xi_{k-1})
							\left(
								\frac{\Delta}{2}\Delta B_k-B^{10}_{\tau^m_{k-1}\tau^m_k}
							\right)\\
					&\quad\quad
						+
							\hat{\varphi}(\xi_{k-1})\Delta (\Delta B_k)^2
						+
							\hat{\varphi}_{011}(\xi_{k-1})
							B^{011}_{\tau^m_{k-1}\tau^m_k}
						+
							\hat{\varphi}_{101}(\xi_{k-1})
							B^{101}_{\tau^m_{k-1}\tau^m_k}
						+
							\hat{\varphi}_{110}(\xi_{k-1})
							B^{110}_{\tau^m_{k-1}\tau^m_k}\\
					&\quad\quad
						+O(\Delta^{5H^{-}})+O(\Delta^{3H^{-}+1})
						+O(\Delta^{H^{-}+2}).
				\end{align*}
		\item	\label{item_1505964312}
				We have
				$\hat{\kappa}_k=O(\Delta^{3H^{-}})$,
				$\kappa_k=O(\Delta^{3H^{-}})$ and
				\begin{align*}
					\max_{1\le k\le 2^m}
						|
							X_{\tau^m_k}(\xi,B)
							-
							\bar{X}^{(m)}_{\tau^m_k}(\xi,B)
						|
					=
						O(\Delta^{3H^{-}-1}).
				\end{align*}
				In particular, the Crank-Nicolson approximation solution
				converges to the solution itself at the partition
				points
				uniformly if $H>\frac{1}{3}$.
		\item	\label{item_1547629361}
				We have
				\begin{align*}
					\max_{0\leq t\leq 1}
						|\bar{X}^{(m)}_t(\xi,B)-X_t(\xi,B+h^{(m)})|
					=
						O(\Delta^{3H^{-}}).
				\end{align*}
	\end{enumerate}
\end{lemma}

The next lemma asserts that $\tilde{\kappa}_k$ is the main term of $\kappa_k$.
As stated in \rref{rem_201708102135},
in order to prove it, we use not only the H\"older regularity of
$B$ but also the convergence in law of the main term of $h^{(m)}$.

\begin{lemma}\label{lem_201708091518}
	For $1\le k\le 2^m$, let
	\begin{align*}
		\tilde{\kappa}_{k}
		&=
			f_3(X_{\tau^m_{k-1}})(\Delta B_k)^3
			+f_4(X_{\tau^m_{k-1}})(\Delta B_k)^4
			+
			g_1(X_{\tau^m_{k-1}})
			\left(\frac{\Delta}{2}\Delta B_k-B^{10}_{\tau^m_{k-1}\tau^m_k}\right)\\
		&\quad
			+\varphi(X_{\tau^m_{k-1}})\Delta (\Delta B_k)^2
			+\varphi_{011}(X_{\tau^m_{k-1}})B^{011}_{\tau^m_{k-1}\tau^m_k}
			+\varphi_{101}(X_{\tau^m_{k-1}})B^{101}_{\tau^m_{k-1}\tau^m_k}
			+\varphi_{110}(X_{\tau^m_{k-1}})B^{110}_{\tau^m_{k-1}\tau^m_k}
	\end{align*}
	and set
	$R_k(\omega)=\kappa_k-\tilde{\kappa}_k$.
	Then there exists $\delta>0$ such that
	$
		\lim_{m\to\infty}
			(2^m)^{3H-\frac{1}{2}+\delta}
			\max_{1\le k\le 2^{m}}
				|\sum_{i=1}^{k}R_i|
		=
			0
	$
	in probability.
\end{lemma}

\begin{remark}
	Although $\tilde{\kappa}_k$ and $\kappa_k$ are defined on $\Omega_0$,
	the definition of $\kappa_k$ on $\Omega_0\setminus\CN$ is essentially meaningless.
	However, the statement of the convergence of $R_k$ makes sense
	because	$\lim_{m\to\infty}\prob(\CN)=1$.
\end{remark}

The following processes are candidates of the main term of $h^{(m)}$:
\begin{gather}\label{eq_1506569861}
	\left\{
		\begin{aligned}
			\Phi_1(t)
			&=
				\sum_{k=1}^{\intPart{2^mt}}
					\left\{
						f_3(X_{\tau^m_{k-1}})(\Delta B_k)^3
						+f_4(X_{\tau^m_{k-1}})(\Delta B_k)^4
					\right\},\\
			\Phi_2(t)
			&=
				\sum_{k=1}^{\intPart{2^mt}}
					g_1(X_{\tau^m_{k-1}})
					\left(
						\frac{\Delta}{2}\Delta B_k
						-
						B^{10}_{\tau^m_{k-1}\tau^m_k}
					\right),\\
			\Phi_3(t)
			&=
				\sum_{k=1}^{\intPart{2^mt}}
					\Biggl\{
						\varphi(X_{\tau^m_{k-1}})
						\Delta(\Delta B_k)^2
						+
						\varphi_{011}(X_{\tau^m_{k-1}})
						B^{011}_{\tau^m_{k-1}\tau^m_k}
						+
						\varphi_{101}(X_{\tau^m_{k-1}})
						B^{101}_{\tau^m_{k-1}\tau^m_k}
						+
						\varphi_{110}(X_{\tau^m_{k-1}})
						B^{110}_{\tau^m_{k-1}\tau^m_k}
					\Biggr\},\\
			\Phi_4(t)
			&=
				-\sum_{k=1}^{\intPart{2^mt}}
					[g_1\sigma'](X_{\tau^m_{k-1}})
					\Delta B_k
					\left(
						\frac{\Delta}{2}\Delta B_k
						-B^{10}_{\tau^m_{k-1}\tau^m_k}
					\right).
		\end{aligned}
	\right.
\end{gather}
\begin{remark}
	The processes $\Phi_1$, $\Phi_2$ and $\Phi_3$ are arising from the expression of $\tilde{\kappa}_{k}$.
	In order to prove \lref{lem_201708091518}, it is necessary to consider $\Phi_4$ together.
\end{remark}
By using \tref{thm_20141117071529}, \tref{thm_20141123062022} and \pref{prop_20141117062354},
we can show the next lemma, which gives us asymptotic of $\Phi_1$, $\Phi_2$, $\Phi_3$ and $\Phi_4$.
\begin{lemma}\label{lem_201708091528}
	Let $W$ and $\tilde{W}$ be standard Brownian motions.
	Assume that $B$, $W$ and $\tilde{W}$ are independent.
	The next assertions hold.
	\begin{enumerate}
		\item	Let $\frac{1}{3}<\Hurst<\frac{1}{2}$.
				Then
				$
					\left(
						B,
						(2^m)^{3H-\frac{1}{2}}
						(\Phi_1,\Phi_2,\Phi_3,\Phi_4)
					\right)
				$
				converges weakly to
				$
					\left(
						B,
						\sigma_{3,H}\int_0^{\cdot}f_3(X_t)dW_t,
						0,
						0,
						0
					\right)$
				in $D([0,1];\RealNum^4)$ with respect to the Skorokhod topology.
				Here, $\sigma_{3,H}$ is a constant defined by \eqref{eq_const_sigma_q_Hurst}.
		\item	Let $\Hurst=\frac{1}{2}$.
				Then
				$
					\left(
						B,
						2^m
						(\Phi_1,\Phi_2,\Phi_3,\Phi_4)
					\right)
				$
				converges weakly to
				\begin{align*}
					&\Biggl(
						B,
						\sqrt{6}\int_0^{\cdot}f_3(X_s)\,dW_s+3\int_0^{\cdot}f_3(X_s)\circ dB_s,
						\frac{1}{\sqrt{12}}\int_0^{\cdot}g_1(X_s)d\tilde{W}_s,\\
					&\qquad\qquad\qquad\qquad\qquad\qquad
						\int_0^{\cdot}\varphi(X_s)\,ds
						+
						\frac{1}{4}
						\int_0^{\cdot}
							\left\{\varphi_{011}(X_s)+\varphi_{110}(X_s)\right\}\,ds,
						0
					\Biggr)
				\end{align*}
				in $D([0,1];\RealNum^4)$ with respect to the Skorokhod topology.
	\end{enumerate}
\end{lemma}

We are in a position to show \tref{thm_20141204015044}.
Proofs of \lrefs[lem_201707121721]{lem_201708091518}{lem_201708091528} are postponed in \secref{sec_1506314800}.

\begin{proof}[Proof of \tref{thm_20141204015044}]
	We follow the idea in \rref{rem_201708102135}.
	Let $h^{(m)}_M$ and $h^{(m)}_R$ be piecewise linear paths associated
	with $\{\tilde{\kappa}_k\}$ and $\{R_k\}$, respectively, in \lref{lem_201708091518}.
	By \lref{lem_201708091528},
	we have the weak convergence in the Skorokhod topology in $D([0,1];\RealNum^2)$,
	\begin{align*}
		\left(B,(2^m)^{3H-\frac{1}{2}}(\Phi_1+\Phi_2+\Phi_3)\right)
		\to (B,U),
	\end{align*}
	where $U$ is the same process defined in \tref{thm_20141204015044}.
	Since $h^{(m)}_M$ is a piecewise linear and $\Phi_1+\Phi_2+\Phi_3$ is step function, we have
	\begin{align*}
		\|h^{(m)}_M-(\Phi_1+\Phi_2+\Phi_3)\|_{\infty}
		=
			O(\Delta^{3H^-}) \qquad \omega\in \CN.
	\end{align*}
	Hence
	$\lim_{m\to\infty}
	(2^m)^{3H-\frac{1}{2}}\|h^{(m)}_M-(\Phi_1+\Phi_2+\Phi_3)\|_{\infty}=0$
	in probability.
	Consequently,
	we have the weak convergence in the uniform convergence topology
	in $C([0,1];\RealNum^3)$:
	\begin{align}\label{eq_201708110956}
		\left(B,(2^m)^{3H-\frac{1}{2}}h^{(m)}_M\right)
		\to (B,U).
	\end{align}
	As stated in \rref{rem_201708102135}, we have
	$
		(2^m)^{3H-\frac{1}{2}}
		\{\bar{X}^{(m)}(\xi,B)-X(\xi,B)\}
		=
			I_1+I_2+I_3
	$,
	where
	\begin{align*}
		I_1
		&=
			\nabla_{(2^{m})^{3H-\frac{1}{2}}h^{(m)}_M}X(\xi,B),\\
		I_2
		&=
			(2^m)^{3H-\frac{1}{2}}
			\left\{\bar{X}^{(m)}(\xi,B)-X(\xi,B+h^{(m)}_M)\right\},\\
		I_3
		&=
			(2^m)^{3H-\frac{1}{2}}
			\left\{
				X(\xi,B+h^{(m)}_M)
				-X(\xi,B)
				-\nabla_{h^{(m)}_M}X(\xi,B)
			\right\}.
	\end{align*}
	We consider $I_2$ and $I_3$ first.
	By Taylor's theorem, we have
	\begin{align*}
		|\bar{X}^{(m)}_t(\xi,B)-X_t(\xi,B+h^{(m)}_M)|
		&\leq
			|
				\bar{X}^{(m)}_t(\xi,B)-X_t(\xi,B+h^{(m)})
			|
			+
			|
				X_t(\xi,B+h^{(m)})-X_t(\xi,B+h^{(m)}_M)
			|\\
		&\leq
			|
				\bar{X}^{(m)}_t(\xi,B)-X_t(\xi,B+h^{(m)})
			|
			+
			\left|
				\int_0^1\nabla_{h^{(m)}_R} X_t(\xi,B+\theta h^{(m)}_R)\,d\theta
			\right|.
	\end{align*}
	By using \lref{lem_201707121721}~(\ref{item_1547629361}) and the boundedness of the derivative, we have
	\begin{align*}
		\|\bar{X}^{(m)}(\xi,B)-X(\xi,B+h^{(m)}_M)\|_\infty
		&\le
			C
			\{
				\Delta^{3H^{-}}
				+
				\|h^{(m)}_R\|_{\infty}
			\}.
	\end{align*}
	Here $C$ is a constant independent of $m$.
	Combining this and \lref{lem_201708091518}, we have
	$\|I_2\|_{\infty}$ converges to $0$
	in probability.
	Similarly, we have
	\begin{align*}
		\|I_3\|_{\infty}\le C(2^m)^{3H-\frac{1}{2}}\|h^{(m)}_M\|_{\infty}^2
		\to 0\qquad \text{in probability}.
	\end{align*}
	We next consider the main term $I_1$.
	Let $J_t(g)$ be the continuous path defined by
	$g$ in \eqref{eq_20150218112958}.
	By \rref{rem_201708111136},	the mapping $g\mapsto J(g)$ is continuous on $C([0,1];\RealNum)$.
	From this, we have the continuity of the mapping
	\begin{align*}
		C([0,1];\RealNum^2)\ni(g,z)
		\mapsto
			\sigma(x(g))z+J(g)\int_0^{\cdot}J_s^{-1}(g)\Wronskian(x_s(g))z_s\,ds
			\in
			C([0,1];\RealNum).
	\end{align*}
	Combining \pref{prop_20140929105014}, \eqref{eq_201708110956} and the above,
	we complete the proof.
\end{proof}

\subsection{Proof of key lemmas}\label{sec_1506314800}
In the rest of this section, we show \lrefs[lem_201707121721]{lem_201708091518}{lem_201708091528}.
\lref{lem_201707121721} follows from the next lemma immediately:
\begin{lemma}\label{lem_1506077157}
	For any $\omega\in \CN$, the following hold.
	\begin{enumerate}
		\item	\label{item_1506567781}
				We have
				 \begin{align*}
					&
						\xi_k-\xi_{k-1}
						=
							b(\xi_{k-1})\Delta+\sigma(\xi_{k-1})\Delta B_k
							+\frac{1}{2}[\sigma'\sigma](\xi_{k-1})(\Delta B_k)^2
							+\frac{1}{4}\left[\sigma(\sigma')^2+\sigma^2\sigma''\right](\xi_{k-1})
							(\Delta B_k)^3\\
					&\quad
							+
							\left[
								\frac{1}{12}\sigma'''\sigma^3
								+\frac{3}{8}\sigma^2\sigma'\sigma''
								+\frac{1}{8}\sigma(\sigma')^3
							\right]
								(\xi_{k-1})
							(\Delta B_k)^4
							+
							\frac{1}{2}\left[\sigma'b+\sigma b'\right](\xi_{k-1})
							\Delta(\Delta B_k)\\
					&\quad
							+
								\frac{1}{4}
								\left[
									(b(\sigma')^2+\sigma^2b'')
									+
									2(\sigma b\sigma''+\sigma\sigma'b')
								\right]
									(\xi_{k-1})
								\Delta (\Delta B_k)^2
							+
								\frac{1}{2}[bb'](\xi_{k-1})\Delta^2\\
					&\quad
							+O(\Delta^{5H^{-}})
							+O(\Delta^{3H^{-}+1}).
				\end{align*}
		\item	\label{item_1506567805}
				We have
				\begin{multline*}
					X_\Delta(\xi_{k-1},\theta_{\tau^m_{k-1}}B)-\xi_{k-1}\\
					\begin{aligned}
						&=
								b(\xi_{k-1})\Delta+\sigma(\xi_{k-1})\Delta B_k
								+
								\frac{1}{2}
								\left[\sigma\sigma'\right](\xi_{k-1})(\Delta B_k)^2
								+
								\frac{1}{3!}
								\left[\sigma(\sigma\sigma')'\right](\xi_{k-1})
								(\Delta B_k)^3\\
						&\phantom{=}\quad
								+
								\frac{1}{4!}
								\left[
									\sigma(\sigma(\sigma\sigma')')'
								\right]
									(\xi_{k-1})
								(\Delta B_k)^4
								+[b\sigma'](\xi_{k-1})\Delta(\Delta B_k)
								+
									[\sigma b'-b\sigma'](\xi_{k-1})
									B^{10}_{\tau^m_{k-1}\tau^m_k}\\
						&\phantom{=}\quad
								+
								b(\sigma\sigma')'(\xi_{k-1})
								B^{011}_{\tau^m_{k-1}\tau^m_k}
								+
								\sigma(b\sigma')'(\xi_{k-1})
								B^{101}_{\tau^m_{k-1}\tau^m_k}
								+
								\sigma(\sigma b')'(\xi_{k-1})
								B^{110}_{\tau^m_{k-1}\tau^m_k}
								+\frac{1}{2}[b'b](\xi_{k-1})\Delta^2\\
						&\phantom{=}\quad
								+O(\Delta^{5H^{-}})+O(\Delta^{3H^{-}+1})
								+O(\Delta^{H^{-}+2}).
					\end{aligned}
				\end{multline*}
	\end{enumerate}
\end{lemma}
\begin{proof}
	(\ref{item_1506567781}) $\xi_k$ is determined by the equation
	\begin{align}\label{eq_20177111737}
		\xi_k
		=
			\xi_{k-1}
			+\frac{\sigma(\xi_{k-1})+\sigma(\xi_{k})}{2}\Delta B_k
			+\frac{b(\xi_{k-1})+b(\xi_{k})}{2}\Delta.
	\end{align}
	Since the implicit function is $C^{\infty}$ as in \rref{rem_201708101101},
	there exist constants $a_{1,0},\dots,a_{4,0}$, $a_{0,1}$, $a_{1,1}$, $a_{2,1}$ and $a_{0,2}$ such that
	\begin{align*}
		\xi_k-\xi_{k-1}
		=
			\sum_{i=1}^4a_{i,0}
				(\Delta	B_k)^i
			+a_{0,1}\Delta
			+a_{1,1}\Delta(\Delta B_k)
			+a_{2,1}\Delta(\Delta B_k)^2
			+a_{0,2}\Delta^2
			+O(\Delta^{3H^{-}+1})
			+O(\Delta^{5H^{-}}).
	\end{align*}
	Putting this expansion of $\xi_k$ into the equation
	(\ref{eq_20177111737}) and compare the coefficients of
	the both sides of equation, we obtain the desired formula.

	\noindent
	(\ref{item_1506567805}) This is a immediate consequence of \pref{prop_20141010094511}.
\end{proof}

\begin{proof}[Proof of \lref{lem_201707121721}]
	(\ref{item_1505965509}) The assertion follows from \lref{lem_1506077157} and the definition of $\hat{\kappa}_k$.

	\noindent
	(\ref{item_1505964312})
	The estimate $\hat{\kappa}_k=O(\Delta^{3H^{-}})$ follows from
	(\ref{item_1505965509}) and the H\"older continuity of $B$.
	It follows that $\kappa_k=O(\Delta^{3H^{-}})$
	from the estimate of $\hat{\kappa}_k$ and \lref{lem_201707251240}.
	By combining $\kappa_k=O(\Delta^{3H^{-}})$ and
	the Lipschitz continuity of the mapping $B\mapsto X(B)$,
	we obtain the last assertion.

	\noindent
	(\ref{item_1547629361})
	Since \lref{lem_1506077157} for $\Delta=t-\tau^m_{k-1}$ is still valid,
	for $\tau^m_{k-1}<t\leq\tau^m_k$, we have
	\begin{align*}
		\bar{X}^{(m)}_t(\xi,B)-X_t(\xi,B+h^{(m)})
		&=
			\{\bar{X}^{(m)}_t(\xi,B)-\xi_{k-1}\}
			-
			\{X_{t-\tau^m_{k-1}}(\xi_{k-1},\theta_{\tau^m_{k-1}}(B+h^{(m)}))-\xi_{k-1}\}\\
		&=
			O(h^{(m)}_t-h^{(m)}_{\tau^m_{k-1}}).
	\end{align*}
	Noting
	$
		O(h^{(m)}_t-h^{(m)}_{\tau^m_{k-1}})
		=
			O(\kappa_k)
		=
			O(\Delta^{3H^{-}})
	$,
	we see the assertion.
\end{proof}

Next we show \lref{lem_201708091528}.
To prove this lemma, we use the following results concerning the
Skorokhod topology.
\begin{proposition}\label{prop_201708091009}
	The following hold.
	\begin{enumerate}
		\item	\label{item_1506071756}
				The mapping
				$
					D([0,1];\RealNum^d)\ni(x_i)_{i=1}^d
					\mapsto
						(\sum_{i=1}^dx_i)\in D([0,1];\RealNum)
				$
				is continuous.
		\item	\label{item_1506071772}
				The mapping
				$
					D([0,1];\RealNum^d)\ni x
					\mapsto
						\sup_{0\le t\le 1}|x_t|\in \RealNum
				$
				is continuous.
		\item	\label{item_1506072032}
				We assume random variables in this statement are
				defined in the same probability space.
				Let $\{X_n\}_{n=1}^{\infty}$ and $\{Y_n\}_{n=1}^{\infty}$
				be random variables with values in $C([0,1];\RealNum^{d_1})$ and $D([0,1];\RealNum^{d_2})$, respectively.
				Let $\{Z_n\}_{n=1}^\infty$ be random variables with values in $D([0,1];\RealNum^{d_3})$.
				Let $\varphi : C([0,1];\RealNum^{d_1})\to C([0,1];\RealNum^{d_4})$ be a continuous mapping.
				Suppose that $(X_n, Y_n)\in D([0,1];\RealNum^{d_1+d_2})$
				converges to $(X,Y)$ in law with respect to the Skorokhod topology and $\|Z_n\|_{\infty}\to 0$ in probability.
				Then $(X_n,Y_n,\varphi(X_n),Z_n)$ converges in
				law in the Skorokhod topology
				to $(X,Y,\varphi(X),0)\in D([0,1];\RealNum^{d_1+d_2+d_3+d_4})$.
	\end{enumerate}
\end{proposition}

\begin{proof}[Proof of \lref{lem_201708091528}]
	First, we consider $\Phi_1$ and $\Phi_2$.
	Recalling $f_4=\sigma f_3'/2$, we have
	$
		f_3(X_{\tau^m_{k-1}})
		+
		f_4(X_{\tau^m_{k-1}})\Delta B_k
		=
			\{f_3(X_{\tau^m_{k-1}})+f_3(X_{\tau^m_k})\}/2
			+O(\Delta^{2H^-})
			+O(\Delta)
	$.
	Hence
	\begin{multline*}
		(2^m)^{3H-\frac{1}{2}}
		\left\{
			f_3(X_{\tau^m_{k-1}})(\Delta B_k)^3
			+
			f_4(X_{\tau^m_{k-1}})(\Delta B_k)^4
		\right\}\\
		=
			(2^m)^{-1/2}
			\frac{f_3(X_{\tau^m_{k-1}})+f_3(X_{\tau^m_k})}{2}
			\hermitePoly{3}(2^{mH}\Delta B_k)
			+
			(2^m)^{H-1/2}
			\frac{f_3(X_{\tau^m_{k-1}})+f_3(X_{\tau^m_k})}{2}3\Delta B_k
			+R_{m,k}(B)
	\end{multline*}
	where $R_{m,k}(B)=O(\Delta^{5H^{-}-3H+\frac{1}{2}})+O(\Delta^{3H^{-}-3H+\frac{3}{2}})$.
	Note that
	$
		\lim_{m\to\infty}
			\sum_{k=1}^{2^m}|R_{m,k}|
			=
		0
	$
	for any $\omega\in \bigcup_m\CN$.
	By \pref{prop_20150424100658}, we have
	\begin{align*}
		\left\|
			\sum_{k=1}^{\intPart{2^m\cdot}}
				\frac{f_3(X_{\tau^m_{k-1}})+f_3(X_{\tau^m_k})}{2}
				\Delta B_k
			-
				\int_0^{\cdot}f_3(X_s)\,d^\circ B_s
		\right\|_{\infty}
		\to
			0\qquad \omega\in \bigcup_{m}\CN.
	\end{align*}
	By \rref{rem_201708111136},	the mapping $B\mapsto\int_0^{\cdot}f_3(X_s)\,d^{\circ}B_s$ is
	continuous in the uniform norm.
	By \tref{thm_20141117071529}, \tref{thm_20141123062022}, \pref{prop_20141117062354} and \pref{prop_201708091009}~(\ref{item_1506072032}),
	\begin{multline*}
		\left(
			B,
			(2^m)^{3H-\frac{1}{2}}(\Phi_1,\Phi_2)
		\right)\\
		\begin{aligned}
			\to
				\begin{cases}
					\left(
						B,
						\displaystyle{\sqrt{6}\int_0^{\cdot}f_3(X_s)\,dW_s}
						+\displaystyle{3\int_0^{\cdot}f_3(X_s)\,d^\circ B_s},
						\displaystyle{\frac{1}{\sqrt{12}}\int_0^{\cdot}g_1(X_s)\,d\tilde{W}_s}
					\right), &
					\Hurst=1/2,\\
					\left(
						B,
						\displaystyle{\sigma_{3,H}\int_0^{\cdot}f_3(X_s)\,dW_s},
						0
					\right),&
					1/3<\Hurst<1/2
				\end{cases}
		\end{aligned}
	\end{multline*}
	weakly in the Skorokhod topology.
	Note that $\sigma_{3,\frac{1}{2}}=\sqrt{6}$.
	(See \rref{rem_201708101002}.)

	Next, we consider $\Phi_3$.
	Suppose $1/3<\Hurst<1/2$.
	By \lref{lem_201708052140}, for any $\omega\in\CN$,
	\begin{align*}
		(2^m)^{3H-\frac{1}{2}}
		\sum_{k=1}^{2^m}
			\left(
				|\Delta (\Delta B_k)^2|
				+|B^{011}_{\tau^m_{k-1}\tau^m_k}|
				+|B^{101}_{\tau^m_{k-1}\tau^m_k}|
				+|B^{110}_{\tau^m_{k-1}\tau^m_k}|
			\right)
		=
			O(\Delta^{2H^{-}-3H+\frac{1}{2}}).
	\end{align*}
	Hence $\|\Phi_3\|_{\infty}$ converges to $0$ in probability.
	We consider the case $\Hurst=\frac{1}{2}$.
	Then we have
	\begin{align*}
		B^{011}_{s,t}
		&=
			\int_s^t\left(\int_s^u(r-s)dB_r\right)\,dB_u+\frac{(t-s)^2}{4},\\
		B^{101}_{s,t}
		&=
			\int_s^t\left(\int_s^u(B_r-B_s)dr\right)\,dB_u,\\
		B^{110}_{s,t}
		&=
			\int_s^t\left(\int_s^u(B_r-B_s)dB_r\right)\,du+\frac{(t-s)^2}{4},
	\end{align*}
	where $dB_r$ is the It\^o integral.
	By the same reason as for $\Phi_3$,
	we see that for almost all $\omega$ uniformly,
	\begin{align*}
		\lim_{m\to\infty}2^m
			\sum_{k=1}^{\intPart{2^m\cdot}}
				\varphi_{\boldsymbol{i}}(X_{\tau^m_{k-1}})
				B_{\tau^m_{k-1}\tau^m_k}^{\boldsymbol{i}}
		=
		\begin{cases}
			\displaystyle{\frac{1}{4}\int_0^{\cdot}\varphi_{\boldsymbol{i}}(X_s)\,ds}, & \boldsymbol{i}=011, 110,\\
			0, & \boldsymbol{i}=101.
		\end{cases}
	\end{align*}
	By a similar calculation to the above, we have
	\begin{align*}
		\lim_{m\to\infty}
			2^m
			\sum_{k=1}^{\intPart{2^m\cdot}}
				\varphi(X_{\tau^m_{k-1}})
				\Delta(\Delta B_k)^2
		=
			\int_0^{\cdot}
				\varphi(X_s)\,ds
		\quad
		\text{a.s. $\omega$ uniformly}.
	\end{align*}
	Hence, we see that for almost all $\omega$ uniformly,
	\begin{align*}
		\lim_{m\to\infty}
			(2^m)^{3H-\frac{1}{2}}\Phi_3
		=
			\int_0^{\cdot}\varphi(X_s)\,ds
			+
			\frac{1}{4}
			\int_0^{\cdot}
				\left\{\varphi_{011}(X_s)+\varphi_{110}(X_s)\right\}\,ds.
	\end{align*}

	Finally, we consider the term $\Phi_4$.
	Suppose $1/3<\Hurst<1/2$.
	Then for any $\omega\in \CN$
	\begin{align}\label{eq_201707261039}
		(2^m)^{3H-\frac{1}{2}+\delta}
		\sum_{k=1}^{2^{m}}
		   \left|
				\Delta B_k
				\left(
					\frac{\Delta}{2}\Delta B_k
					-
					B^{10}_{\tau^m_{k-1}\tau^m_k}
				\right)
			\right|
		=O(\Delta^{2H^{-}-3H+\frac{1}{2}-\delta})
		=O\left(\Delta^{\frac{1}{2}-H-2\epsilon-\delta}\right).
	\end{align}
	Hence, if $\delta<\frac{1}{2}-H-2\epsilon$,
	$
		\lim_{m\to\infty}
			\|(2^m)^{3H-\frac{1}{2}+\delta}\Phi_4\|_{\infty}
		=
			0
	$
	in probability.
	We consider the case where $\Hurst=\frac{1}{2}$.
	In this case, $B_t$ is a standard Brownian motion and we have
	\begin{align*}
		\expect
		\left[
			\Delta B_k
			\left(
				\frac{\Delta}{2}\Delta B_k
				-
				B^{10}_{\tau^m_{k-1}\tau^m_k}
			\right)
		\right]
		&=
		0,
		&
		\expect
			\left[
				\left\{
					\Delta B_k
					\left(
						\frac{\Delta}{2}\Delta B_k-B^{10}_{\tau^m_{k-1}\tau^m_k}
					\right)
				\right\}^2
			\right]
		=
		\frac{\Delta^4}{3}.
	\end{align*}
	Since $X_t(\xi,B)$ is $\sigma(\{B_u~|~0\le u\le t\})$-adapted, by Doob's inequality, we have
	\begin{align*}
	\Delta^{-2} E\left[\sup_{0\le t\le 1}
	|\Phi_4(t)|^2\right]\le C\Delta.
	\end{align*}
	This implies that for any $\delta<\frac{1}{2}$,
	\begin{align}\label{eq_201707241515}
		\lim_{m\to\infty}\Delta^{-1-\delta}
		\sup_{0\le t\le 1}|\Phi_4(t)|=0 \qquad\text{a.s. $\omega$}.
	\end{align}

	From the calculation above, \rref{rem_201708111136} and \pref{prop_201708091009}~(\ref{item_1506072032}), we see the conclusion.
\end{proof}

The next lemma is a corollary of \lref{lem_201708091528} and \pref{prop_201708091009},
which is used in the proof of \lref{lem_201708091518}.
\begin{lemma}\label{lem_1506322555}
	Set
	\begin{align*}
		\Psi_{m,\delta}
		&=
			(2^m)^{3H-\frac{1}{2}-\delta}
			\max_{0\le t\le 1}
				\left|\sum_{i=1}^4\Phi_i(t)\right|.
	\end{align*}
	Then, for any $\delta>0$, $\lim_{m\to\infty}\Psi_{m,\delta}=0$ in probability.
\end{lemma}
\begin{proof}
	From \pref{prop_201708091009}~(\ref{item_1506071756}) and (\ref{item_1506071772}),
	we see that
	$
		\sup_{0\le t\le 1}
			\left|\sum_i(2^m)^{3H-\frac{1}{2}}\Phi_i(t)\right|
	$
	converges in law.
	Thus we obtain that
	$\lim_{m\to\infty}\Psi_{m,\delta}=0$
	in probability.
\end{proof}

Next, we show \lref{lem_201708091518}.
By using \lref[lem_201707121721]{lem_201708091528}, we obtain a representation of
the main term of $\kappa_{k}$ in terms of $\Delta$, $\Delta B_k$, $B^{\boldsymbol{i}}_{\tau^m_{k-1}\tau^m_{k}}$ and $X_{\tau^m_{k-1}}$.
We divide this calculation into two steps.
In the first step, we have the following.
This estimate is a pathwise estimate.
We use just H\"older continuity of the path of $B$.

\begin{lemma}\label{lem_201707121723}
    Let $\omega\in \CN$.
    For $k$ $(1\le k\le 2^m)$ and $x\in \RealNum$, let
    \begin{align*}
        F_k(x,B)
        &=
            f_3(x)(\Delta B_k)^3
            +f_4(x)(\Delta B_k)^4
            +g_1(x)
            \left(
                \frac{\Delta}{2}\Delta B_k-B^{10}_{\tau^m_{k-1}\tau^m_k}
            \right)\\
        &\phantom{=}\quad
            +\varphi(x)\Delta (\Delta B_k)^2
            +\varphi_{011}(x)B^{011}_{\tau^m_{k-1}\tau^m_k}
            +\varphi_{101}(x)B^{101}_{\tau^m_{k-1}\tau^m_k}
            +\varphi_{110}(x)B^{110}_{\tau^m_{k-1}\tau^m_k},\\
        G_k(x,B)
        &=
            -
            [g_1\sigma'](x)
            \Delta B_k
            \left(\frac{\Delta}{2}\Delta B_k-B^{10}_{\tau^m_{k-1}\tau^m_k}\right),\\
        r_k
        &=
            \kappa_{k}-F_k(\xi_{k-1},B)-G_k(\xi_{k-1},B).
    \end{align*}
	Then it holds that $r_k=O(\Delta^{3H^{-}+1})+O(\Delta^{5H^{-}})$.
\end{lemma}

\begin{proof}
	By the Taylor formula, there exists $0<\rho<1$ such that
	\begin{align*}
		\hat{\kappa}_k
		&=\xi_k
			-X_{2^{-m}}(\xi_{k-1}, \theta_{\tau^m_{k-1}}B)\\
		&=
			X_{2^{-m}}(\xi_{k-1},\theta_{\tau^m_{k-1}}B+2^m\kappa_{k} \ell)
			-X_{2^{-m}}(\xi_{k-1}, \theta_{\tau^m_{k-1}}B)\\
		&=
			\nabla_{2^m\kappa_{k}\ell} X_{2^{-m}}(\xi_{k-1},\theta_{\tau^m_{k-1}}B)
			+
			\frac{1}{2}\nabla_{2^m\kappa_{k}\ell}^2X_{2^{-m}}(\xi_{k-1},\theta_{\tau^m_{k-1}}B
			+\rho 2^m\kappa_{k}\ell).
	\end{align*}
	Applying the estimate $\kappa_k=O(\Delta^{3H^{-}})$ and
	\pref{prop_20140929105014}~(\ref{item_20141228072212}),
	we see that the second term of the right-hand side is $O(\Delta^{6H^{-}})$.
	As for the first term, \pref{prop_20140929105014}~(\ref{item_20141106063948}), \lref{lem_201707121721}~(\ref{item_1505964312})
	and \pref{prop_20141010094511} yield
	\begin{align*}
		\nabla_{2^m\kappa_{k}\ell} X_{2^{-m}}(\xi_{k-1},\theta_{\tau^m_{k-1}}B)
			&=
				\sigma\left(X_\Delta(\xi_{k-1},\theta_{\tau^m_{k-1}}B)\right)
				\int_0^{\Delta}
					\exp
						\left(
							\int_s^{\Delta}
								\left[\frac{\Wronskian}{\sigma}\right]
								\left(X_{u}(\xi_{k-1},\theta_{\tau^m_{k-1}}B)\right)\,
							du
						\right)
					\frac{\kappa_k}{\Delta}\,
					ds\\
			&=
				\sigma(\xi_{k-1})\kappa_{k}
				+
				\left\{
					\sigma\left(X_\Delta(\xi_{k-1},\theta_{\tau^m_{k-1}}B)\right)
					-
					\sigma(\xi_{k-1})
				\right\}
				\kappa_{k}
				+O(\Delta^{3H^{-}+1})\\
			&=
				\left\{
					\sigma(\xi_{k-1})
					+\sigma(\xi_{k-1})\sigma'(\xi_{k-1})\Delta B_k
				\right\}
				\kappa_{k}
				+O(\Delta^{5H^{-}})
				+O(\Delta^{3H^{-}+1}).
	\end{align*}
	Hence we see that $\hat{\kappa}_{k}$ and $\kappa_{k}$ satisfy
	\begin{align*}
		\hat{\kappa}_{k}
		=
			\sigma(\xi_{k-1})
			\left\{
				1+
				\sigma'(\xi_{k-1})\Delta B_k
			\right\}
			\kappa_{k}
			+O(\Delta^{3H^{-}+1})
			+O(\Delta^{5H^{-}}).
	\end{align*}
	Since $|\sigma'(\xi_{k-1})\Delta B_k|\le 1/2$ on $\CN$, we can solve this equation and using
    \lref{lem_201707121721}~(\ref{item_1505965509}),
    \begin{align*}
        \kappa_k
        &=
            \sigma(\xi_{k-1})^{-1}\{1-\sigma'(\xi_{k-1})\Delta B_k\}\hat{\kappa}_k
            +O(\Delta^{3H^{-}+1})
            +O(\Delta^{5H^{-}})\\
        &=
            F_k(\xi_{k-1},B)+G_k(\xi_{k-1},B)\\
        &\phantom{=}\quad
            -
            [\sigma^{-1}\sigma'](\xi_{k-1})
            \Delta B_k
            \left\{
                \hat{\kappa}_k
                -\hat{f}_3(\xi_{k-1})(\Delta B_k)^3
                -\hat{g}_1(\xi_{k-1})\left(\frac{\Delta}{2}\Delta B_k
                -B^{10}_{\tau^m_{k-1}\tau^m_k}\right)
            \right\}\\
        &\phantom{=}\quad
            +O(\Delta^{3H^{-}+1})
            +O(\Delta^{5H^{-}}).
    \end{align*}
    Since
    $
        \hat{\kappa}_k
        -\hat{f}_3(\xi_{k-1})(\Delta B_k)^3
        -\hat{g}_1(\xi_{k-1})\left(\frac{\Delta}{2}\Delta B_k
        -B^{10}_{\tau^m_{k-1}\tau^m_k}\right)
        =
            O(\Delta^{2H^{-}+1})
            +O(\Delta^{4H^{-}})
    $,
    we complete the proof.
\end{proof}

Now, we are in a position to prove \lref{lem_201708091518}.

\begin{proof}[Proof of \lref{lem_201708091518}]
	Let
	$
		\epsilon_m
		=
			\max_{1\le k\le 2^m}
				|
		X_{\tau^m_k}(\xi,B)-\bar{X}^{(m)}_{\tau^m_k}(\xi,B)
				|
	$.
    We proved that $\lim_{m\to\infty}(2^m)^{3H^{-}-1}\epsilon_m=0$ for $\omega\in \bigcup_m\CN$.
	Our first task is to improve this estimate as $\lim_{m\to\infty}(2^m)^{3H-1/2-\delta}\epsilon_m=0$
	in probability for any $\delta>0$ by using $\lim_{m\to\infty}\Psi_{m,\delta}=0$ in probability
    (recall \lref{lem_1506322555}).
    To this end, let
    \begin{align*}
        \kappa_{k,1}&=F_k(X_{\tau^m_{k-1}},B)+G_k(X_{\tau^m_{k-1}},B),\\
        \kappa_{k,2}&=F_k(\xi_{k-1},B)+G_k(\xi_{k-1},B)-
        \left(F_k(X_{\tau^m_{k-1}},B)+G_k(X_{\tau^m_{k-1}},B)\right),
    \end{align*}
    where $F_k$ and $G_k$ are the same functions as in \lref{lem_201707121723}.
    Then $\kappa_k=\kappa_{k,1}+\kappa_{k,2}+r_k$,
    $\tilde{\kappa}_k=F_k(X_{\tau^m_{k-1}},B)$ and
    $R_k=G_k(X_{\tau^m_{k-1}}, B)+\kappa_{k,2}+r_k$ hold.
 	Here, $r_k$ is defined in \lref{lem_201707121723}.
	Let $h^{(m)}_i$ $(i=1,2)$ be piecewise linear paths which are defined by $\{\kappa_{k,i}\}$.
	We define $h^{(m)}_r$ similarly by $\{r_k\}$.
	Note that
	$\|h_1^{(m)}\|_{\infty}=O(\Delta^{3H-1/2-\delta})\Psi_{m,\delta}$ holds.
    By the Lipschitz continuity of $F_k$ and $G_k$ with respect to $x$-variable,
 	we have
	\begin{align*}
		\|h^{(m)}_2\|_{\infty}\le \sum_{k=1}^{2^m}|\kappa_{k,2}|&\le
		K\epsilon_{m},
		\qquad \omega\in \CN
	\end{align*}
	where
	$
		K=O(\Delta^{3H^{-}-1}).
	$
	By \lref{lem_201707121723}, we have
	\begin{align}\label{eq_2017071724}
		\|h^{(m)}_r\|_{\infty}
		\le
			\sum_{k=1}^{2^m}
				|r_k|
		=
			O(\Delta^{3H^{-}})
			+O(\Delta^{5H^{-}-1}),
		\qquad
			\omega\in \CN.
	\end{align}
	By the Lipschitz continuity of $B\mapsto X(\xi,B)$ in the uniform
	norm, we have
	\begin{align}
		\epsilon_m
		&=
			\max_{1\le k\le 2^m}
				|X_{\tau^m_k}(\xi,B)-X_{\tau^m_k}(\xi,B+h^{(m)}_1+h^{(m)}_2+h^{(m)}_r)|\nonumber\\
		&\le
			C\sum_{i=1}^3\|h^{(m)}_i\|_{\infty}
		=
			\tilde{K}\epsilon_m+\hat{K},\qquad \omega\in \CN,\label{eq_201707121756}
	\end{align}
	where $\tilde{K}=CK=O(\Delta^{3H^{-}-1})$ and
	$\hat{K}=C(\|h^{(m)}_1\|_{\infty}+\|h^{(m)}_r\|_{\infty})$.
	By applying the inequality \eqref{eq_201707121756}, $n$-times and
	using the rough estimate $\epsilon_m=O(\Delta^{3H^{-}-1})$, we
	get
	\begin{align*}
		\epsilon_m
		&\le
			\tilde{K}^nO(\Delta^{3H^{-}-1})+\hat{K}\left(1+\sum_{j=1}^{n-1}
			\tilde{K}^j\right).
	\end{align*}
	From this, we conclude that for $\omega\in \CN$,
	$\epsilon_m=\Psi_{m,\delta}(\omega)O(\Delta^{3H-1/2-\delta})
	+O(\Delta^{3H^{-}})+O(\Delta^{5H^{-}-1})$ holds for any
	$\delta>0$.
    We now prove the estimate of
    the sum of $R_k$.
    Thanks for the the improved estimate of $\epsilon_m$,
    we obtain for any $\delta>0$
	\begin{align*}
		\sum_{k=0}^{2^m-1}|\kappa_{k,2}|=O(\Delta^{3H-1/2+3H^{-}-1-\delta})
			\Psi_{m,\delta}(\omega)+O(\Delta^{6H^{-}-1})+O(\Delta^{8H^{-}-2}),
		\qquad \text{$\omega\in \CN$.}
	\end{align*}
 We already proved the necessary estimates in
 (\ref{eq_2017071724}), (\ref{eq_201707261039}) and
 (\ref{eq_201707241515})
 for the sum of $r_k$ and $G_k(X_{\tau^m_{k-1}},B)$.
Thus, we complete the proof.
\end{proof}

\section{The Euler scheme and the Milstein scheme}\label{sec_201707212051}
In this section, we show \tref[thm_20141204013919]{thm_20141204015019},
which are concerning with the Euler-Maruyama scheme and the Milstein scheme, respectively.
Since the proofs are similar to one of \tref{thm_20141204015044},
we omit the detail and give key lemmas.
We denote by $\bar{X}^{(m)}$ the Euler scheme or the Milstein scheme
and set $\xi_k=\bar{X}^{(m)}_{\dyadicPart[m]{k}}$.

Note that \lref{lem_201707240936} holds for the Euler scheme and the Milstein scheme.
We see \lref{lem_201707240936} holds for the both of the schemes.
We denote by $h^{(m)}$ the piecewise linear function which appears in \lref{lem_201707240936}
and we write $\kappa_k=h^{(m)}(\dyadicPart[m]{k})-h^{(m)}(\dyadicPart[m]{k-1})$ for every $1\le k\le 2^m$.
Because analysis of 1-step error
$
	\hat{\kappa}_k
	=
		\{
			\xi_{k}-\xi_{k-1}
		\}
		-
			\{
				X_{2^{-m}}(\xi_{k-1},\theta_{\tau^m_{k-1}}B)
				-
				\xi_{k-1}
			\}
$
of the scheme and the main term $\tilde{\kappa}_k$ are essential in the proof,
we state assertions on them, that is, we give counterparts of \lrefs[lem_201707121721]{lem_201708091518}{lem_201708091528}.

\subsection{The Euler scheme}
In this subsection, we assume $1/2<\Hurst<1$ and show \tref{thm_20141204013919}.
To state assertions, we set $\hat{f}_2=-\sigma\sigma'/2$ and $f_2=-\sigma'/2$.
Then we see the following lemmas:
\begin{lemma}\label{lem_1506481850}
	For any $\omega\in\probSp_0$, the following hold:
	\begin{enumerate}
		\item	We have $\hat{\kappa}_k=\hat{f}_2(\xi_{k-1})(\Delta B_k)^2+O(\Delta^{\Hurst^{-}+1})$.
		\item	We have	$\hat{\kappa}_k=O(\Delta^{2H^{-}})$, $\kappa_k=O(\Delta^{2H^{-}})$ and
				\begin{align*}
					\max_{1\le k\le 2^m}
						|
							X_{\tau^m_k}(\xi,B)
							-
							\bar{X}^{(m)}_{\tau^m_k}(\xi,B)
						|
					=
						O(\Delta^{2H^{-}-1}).
				\end{align*}
		\item	We have
				\begin{align*}
					\max_{0\leq t\leq 1}
						|\bar{X}^{(m)}_t(\xi,B)-X_t(\xi,B+h^{(m)})|
					=
						O(\Delta^{2H^{-}}).
				\end{align*}
	\end{enumerate}
\end{lemma}
\begin{lemma}\label{lem_1506481781}
	For $1\le k\le 2^m$, let
	\begin{align*}
		\tilde{\kappa}_{k}
		&=
			f_2(X_{\dyadicPart[m]{k}})(\Delta B_k)^2
	\end{align*}
	and set
	$R_k(\omega)=\kappa_k-\tilde{\kappa}_k$.
    Then $R_k=O(\Delta^{4\Hurst^{-}-1})+O(\Delta^{\Hurst^{-}+1})$.
\end{lemma}

\begin{lemma}\label{lem_1506482751}
	Let
	\begin{align*}
		\Phi_1(t)
		&=
			\sum_{k=1}^{\intPart{2^mt}}
				f_2(X_{\tau^m_{k-1}})(\Delta B_k)^2.
	\end{align*}
	Then,
	$
		\left(
			B,
			2^{m(2\Hurst-1)}
			\Phi_1
		\right)
	$
	converges to
	$
		\left(
			B,
			\int_0^\cdot
				f_2(X_u)\,
				du
		\right)
	$
	in $D([0,1];\RealNum^2)$ with respect to the Skorokhod topology in probability.
\end{lemma}

Here we make comments on proof of the lemmas above:
\begin{itemize}
	\item	\lref{lem_1506481850} is seen by the similar way with \lref{lem_201707121721}.
	\item	\lref{lem_1506481781} follows from
			the equality $\hat{\kappa}_k=\sigma(\xi_{k-1})\kappa_k+O(\Delta^{3\Hurst^{-}})$
			and \lref{lem_1506481850} (note that we do not use \lref{lem_1506482751}).
	\item	\lref{lem_1506482751} is a direct consequence of \tref{thm_20141203094309}.
\end{itemize}
Combining the lemmas, we obtain \tref{thm_20141204013919}.

\subsection{The Milstein scheme}
In this subsection, we assume $1/3<\Hurst\leq 1/2$ and \tref{thm_20141204015019}.
We set
\begin{align*}
	\hat{f}_3
	&=
		-
		\frac{1}{3!}
		\sigma(\sigma\sigma')',
	&
	\hat{f}_4
	&=
		-
		\frac{1}{4!}
		\sigma(\sigma(\sigma\sigma')')',\\
	f_3
	&=
		-
		\frac{1}{3!}
		(\sigma\sigma')',
	&
	f_4
	&=
		-
		\frac{1}{4!}
		[
			\sigma^2\sigma'''-3(\sigma')^3
		],
	&
	f_4^\dagger
	&=
		\frac{1}{4!}
		[
			\sigma^2\sigma'''
			+6\sigma\sigma'\sigma''
			+3(\sigma')^3
		].
\end{align*}
Note that $f_4=(\hat{f}_4-\sigma'\hat{f}_3)/\sigma$ and $f_4^\dagger=f_4-\sigma f_3'/2$.
We set $\varphi=0$ and use functions $\hat{g}_1$, $g_1$,
$\hat{\varphi}_{\boldsymbol{i}}$,
$\varphi_{\boldsymbol{i}}$ $(\boldsymbol{i}=011,101,110)$
introduced in \secref{sec_1506485494}.
We define processes $\Phi_1,\dots,\Phi_4$ by \eqref{eq_1506569861} with the functions above.
Then we see the next lemmas:
\begin{lemma}
	For any $\omega\in\probSp$, the following hold.
	\begin{enumerate}
		\item	We have
				\begin{align*}
					\hat{\kappa}_{k}
					&=
						\hat{f}_3(\xi_{k-1})(\Delta B_k)^3
						+\hat{f}_4(\xi_{k-1})(\Delta B_k)^4
						+
							\hat{g}_1(\xi_{k-1})
							\left(
								\frac{\Delta}{2}\Delta B_k-B^{10}_{\tau^m_{k-1}\tau^m_k}
							\right)\\
					&\quad\quad
						+
							\hat{\varphi}_{011}(\xi_{k-1})
							B^{011}_{\tau^m_{k-1}\tau^m_k}
						+
							\hat{\varphi}_{101}(\xi_{k-1})
							B^{101}_{\tau^m_{k-1}\tau^m_k}
						+
							\hat{\varphi}_{110}(\xi_{k-1})
							B^{110}_{\tau^m_{k-1}\tau^m_k}\\
					&\quad\quad
						+O(\Delta^{5H^{-}})
						+O(\Delta^{3H^{-}+1})
						+O(\Delta^{H^{-}+2}).
				\end{align*}
		\item	We have
				$\hat{\kappa}_k=O(\Delta^{3H^{-}})$, $\kappa_k=O(\Delta^{3H^{-}})$ and
				\begin{align*}
					\max_{1\le k\le 2^m}
						|
							X_{\tau^m_k}(\xi,B)
							-
							\bar{X}^{(m)}_{\tau^m_k}(\xi,B)
						|
					=
						O(\Delta^{3H^{-}-1}).
				\end{align*}
		\item	We have
				\begin{align*}
					\max_{0\leq t\leq 1}
						|\bar{X}^{(m)}_t(\xi,B)-X_t(\xi,B+h^{(m)})|
					=
						O(\Delta^{3H^{-}}).
				\end{align*}
	\end{enumerate}
\end{lemma}

\begin{lemma}
	Let
	\begin{align*}
		\tilde{\kappa}_{k}
		&=
			f_3(X_{\tau^m_{k-1}})(\Delta B_k)^3
			+f_4(X_{\tau^m_{k-1}})(\Delta B_k)^4
			+
			g_1(X_{\tau^m_{k-1}})
			\left(\frac{\Delta}{2}\Delta B_k-B^{10}_{\tau^m_{k-1}\tau^m_k}\right)\\
		&\quad
			+\varphi_{011}(X_{\tau^m_{k-1}})B^{011}_{\tau^m_{k-1}\tau^m_k}
			+\varphi_{101}(X_{\tau^m_{k-1}})B^{101}_{\tau^m_{k-1}\tau^m_k}
			+\varphi_{110}(X_{\tau^m_{k-1}})B^{110}_{\tau^m_{k-1}\tau^m_k}
	\end{align*}
	and set
    $R_k(\omega)=\kappa_k-\tilde{\kappa}_k$.
    Then there exists $\delta>0$ such that
	$
		\lim_{m\to\infty}
			(2^m)^{4\Hurst-1+\delta}
			\max_{1\le k\le 2^{m}}
				|\sum_{i=1}^{k}R_i|
		=
			0
	$
	in probability.
\end{lemma}

\begin{lemma}\label{lem_1506499391}
	The following hold:
	\begin{enumerate}
		\item	Let $\frac{1}{3}<\Hurst<\frac{1}{2}$.
				Then
				$
					\left(
						B,
						(2^m)^{4\Hurst-1}
						(\Phi_1,\Phi_2,\Phi_3,\Phi_4)
					\right)
				$
				converges to
				$
					\left(
						B,
						3\int_0^{\cdot}f_4^\dagger(X_s)\, ds,
						0,
						0,
						0
					\right)$
				in $D([0,1];\RealNum^4)$ with respect to the Skorokhod topology
				in probability.
		\item	Let $\Hurst=\frac{1}{2}$.
				Then
				$
					\left(
						B,
						2^m
						(\Phi_1,\Phi_2,\Phi_3,\Phi_4)
					\right)
				$
				converges weakly to
				\begin{multline*}
					\Biggl(
						B,
						\sqrt{6}\int_0^{\cdot}f_3(X_s)\,dW_s
						+3\int_0^{\cdot}f_3(X_s)\circ dB_s
						+3\int_0^{\cdot}f_4^\dagger(X_s)\, ds,\\
						\frac{1}{\sqrt{12}}\int_0^{\cdot}g_1(X_s)d\tilde{W}_s,
						\frac{1}{4}
						\int_0^{\cdot}
							\left\{\varphi_{011}(X_s)+\varphi_{110}(X_s)\right\}\,ds,
						0
					\Biggr)
				\end{multline*}
				in $D([0,1];\RealNum^4)$ with respect to the Skorokhod topology.
	\end{enumerate}
\end{lemma}
Note that in proof \lref{lem_1506499391} we used the decomposition
\begin{align*}
	f_3(X_{\tau^m_{k-1}})
	+
	f_4(X_{\tau^m_{k-1}})\Delta B_k
	&=
		\left\{f_3(X_{\tau^m_{k-1}})+\frac{1}{2}f_3'\sigma(X_{\tau^m_k})\Delta B_k\right\}
		+f_4^\dagger(X_{\tau^m_k})\Delta B_k\\
	&=
		\frac{f_3(X_{\tau^m_{k-1}})+f_3(X_{\tau^m_k})}{2}
		+O(\Delta^{2H^-})
		+O(\Delta)
		+f_4^\dagger(X_{\tau^m_k})\Delta B_k
\end{align*}
and apply \trefs{thm_20141203094309}{thm_20141117071529}{thm_20141123062022}.

\appendix

\section{Gaussian analysis and Malliavin calculus}\label{sec_20150623014415}

We summarize basic results on Gaussian analysis and Malliavin calculus
which we use to estimate some terms of error.
For details, see \cite{Nualart2006}.

Let $(\probSp,\sigmaField,\prob)$ be the canonical probability space
for a one-dimensional centered continuous Gaussian
process $X=\{X_t\}_{0\leq t\leq 1}$
with the covariance $\expect[X_sX_t]=R(s,t)$,
that is,
$\probSp$ is the Banach space of continuous functions from
$[0,1]$ to $\RealNum$ starting at zero with the uniform norm $\|\cdot\|_\infty$,
$\sigmaField$ the $\sigma$-field generated by the cylindrical subsets of $\probSp$,
and $\prob$ a probability measure on $\probSp$ such that
the canonical process $X(\omega)=\omega$, $\omega\in\probSp$, is the Gaussian process.

We construct an abstract Wiener space $(\probSp,\CM,\prob)$
and an isonormal Gaussian process $\{X(h)\}_{h\in\CM}$.
The Hilbert space $\CM$ with the norm $\|\cdot\|_\CM$
and the inner product $\innerProd[\CM]{\cdot}{\ast}$ is defined by as follows;
set
$
	[\mathscr{R}\indicator{[0,t)}](\cdot)
	=
		R(t,\cdot)
	=
		\expect[X_tX_\cdot]
$
and let $\CM_0$ be the linear span of functions $\mathscr{R}\indicator{[0,t)}$
and $\CM$ the Hilbert space defined as the closure of $\CM_0$ with respect to the inner product
$
	\innerProd[\CM]
		{\mathscr{R}\indicator{[0,s)}}
		{\mathscr{R}\indicator{[0,t)}}
	=
		\expect[X_sX_t]
$.
We call the Hilbert space $\CM$ the Cameron-Martin subspace.
Note the map
$
	\CM_0\ni \mathscr{R}\indicator{[0,t)}
	\mapsto
	X(\indicator{[0,t)})\in\LebSp{2}{\probSp}{\RealNum}
$
is an isometry.
Hence if $\{h_n\}_{n=1}^\infty\subset\CM_0$ converges to $h\in\CM$,
then $\{X(h_n)\}_{n=1}^\infty$ converges to some element
$X(h)\in\LebSp{2}{\probSp}{\RealNum}$.
Hence we obtain the isonormal Gaussian process $\{X(h)\}_{h\in\CM}$.

Next, we define the $q$-th Wiener integral $\WienerInt{q}$
which is a map from the symmetric space $\CM^{\odot q}$
to the $q$-th Wiener chaos $\WienerChaos{q}$ for $q\in\NaturalNum$.

In order to define $\CM^{\odot q}$, $\WienerChaos{q}$ and $\WienerInt{q}$,
we denote by $\Lambda$ the set of sequences
$\lambda=(\lambda_1,\dots)\in(\NaturalNum\cup\{0\})^\infty$
such that all the elements vanish except a finite number of them
and set $\lambda!=\prod_{n=1}^\infty \lambda_n!$ for $\lambda\in\Lambda$.
We take an orthonormal basis $\{e_n\}_{n=1}^\infty$ of $\CM$.

We denote by $\otimes$ the tensor product
and by $\CM^{\otimes q}$ the tensor product space for $q\geq 2$.
For $q=0,1$, we set $\CM^{\otimes 0}=\RealNum$ and $\CM^{\otimes 1}=\CM$ by convention.
We define the symmetrization $\tilde{h}\in\CM^{\otimes q}$ for $h\in\CM^{\otimes q}$ as follows:
if $h$ has the form of $h=h_1\otimes\cdots\otimes h_q$ for $h_r\in\CM$, we set
\begin{align*}
	(h_1\otimes\cdots\otimes h_q)^\sim
	=
	\frac{1}{q!}
	\sum_{\sigma\in\SymmetricGroup{q}}
		h_{\sigma(1)}\otimes\cdots\otimes h_{\sigma(q)},
\end{align*}
where $\SymmetricGroup{q}$ is the symmetric group on $\{1,\dots,q\}$;
we also define the symmetrization for general elements in $\CM^{\otimes q}$
by linearity.
For notational simplicity, we set
$
	h_1\odot\cdots\odot h_q
	=
	(h_1\otimes\cdots\otimes h_q)^\sim
$.
An element $h\in\CM^{\otimes q}$ is said to be symmetric if $\tilde{h}=h$.
We denote by $\CM^{\odot q}$ the set of symmetric elements of
$\CM^{\otimes q}$. The space $\CM^{\odot q}$ forms a Hilbert space with
respect to the scaled norm $\sqrt{q!}\|\cdot\|_{\CM^{\otimes q}}$. For
$\lambda\in\Lambda$, set
\begin{align*}
	e^{\lambda}
	=
	\frac{1}{\sqrt{\lambda!}}
	e_1^{\odot \lambda_1}\odot
	e_2^{\odot \lambda_2}\odot\cdots.
\end{align*}
Then, $\{e^{\lambda};|\lambda|=q,\lambda\in\Lambda\}$
is an orthonormal basis of $\CM^{\odot q}$.

As we introduced in Section~\secref{sec_20150519021523},
$\hermitePoly{q}$ denotes the $q$-th Hermite polynomial.
The $q$-th Wiener chaos $\WienerChaos{q}$
is defined as the closed subspace spanned by
$\{\hermitePoly{q}(X(h));h\in\CM,\|h\|_\CM=1\}$ in $\LebSp{2}{\probSp}{\RealNum}$.
For $\lambda\in\Lambda$, set
\begin{align*}
	\WienerHermitePoly{\lambda}
	=
	\frac{1}{\sqrt{\lambda!}}
	\prod_{n=1}^{\infty}
		\hermitePoly{\lambda_n}(X(e_n)).
\end{align*}
Then, $\{\WienerHermitePoly{\lambda};|\lambda|=q,\lambda\in\Lambda\}$
is an orthonormal basis of $\WienerChaos{q}$.

The $q$-th Wiener integral $\WienerInt{q}$ is defined by
$
	\WienerInt{q}(e^\lambda)=\WienerHermitePoly{\lambda}
$
and is extend by linearity.
The mapping $\WienerInt{q}:\CM^{\odot q}\to\WienerChaos{q}$
provides a real linear isometry between $\CM^{\odot q}$ and $\WienerChaos{q}$.

Finally, we summarize results on Malliavin calculus.
Let $\mathcal{S}$ be the totality of all smooth functionals
which have the form of $F=f(X(h_1),\dots,X(h_\alpha))$,
where $h_\beta\in\CM$ and
$f\in\SmoothFunc[poly]{\infty}{\RealNum^\alpha}{\RealNum}$.
The Malliavin derivative $DF$ of $F\in\mathcal{S}$ is
an $\CM$-valued random variable and defined by
\begin{align*}
	DF
	=
	\sum_{\beta=1}^\alpha
		\frac{\partial f}{\partial\xi_\beta}(X(h_1),\dots,X(h_\alpha))
		h_\beta.
\end{align*}
By the iteration, one can define $n$-th derivative $D^nF$, which is an
$\CM^{\odot n}$-valued random variable, by
\begin{align*}
	D^nF
	=
	\sum_{\beta_1,\dots,\beta_n=1}^\alpha
		\frac{\partial^n f}{\partial\xi_{\beta_1}\cdots\partial\xi_{\beta_n}}
			(X(h_1),\dots,X(h_\alpha))
			h_{\beta_1}\otimes\dots\otimes h_{\beta_n}.
\end{align*}
As usual, for $n\in\NaturalNum$ and $1<p<\infty$, we define the Sobolev space $\SobSp{n}{p}{\probSp}{\RealNum}$ by
the completion of $\mathcal{S}$ by the norm
\begin{align*}
	\|F\|_{\SobSp{n}{p}{\probSp}{\RealNum}}^p
	=
	\sum_{k=0}^n
		\expect[\|D^k F\|_{\CM^{\odot k}}^p].
\end{align*}
We set $\SobSp{n}{\infty-}{\probSp}{\RealNum}=\bigcap_{1<p<\infty}\SobSp{n}{p}{\probSp}{\RealNum}$.

Since the derivative operator $D$ is a continuous operator from $\SobSp{1}{2}{\probSp}{\RealNum}$
to $\LebSp{2}{\probSp}{\CM}$, there exists its adjoint operator $\delta$, which is
called the divergence operator or the Skorokhod integral. Notice that the duality
relationship
\begin{align*}
	\expect[F\delta(u)]=\expect[\innerProd[\CM]{DF}{u}]
\end{align*}
holds for any $F\in\SobSp{1}{2}{\probSp}{\RealNum}$ and $u$ belonging to the domain of $\delta$.
By the iteration, we see that there exists an operator $\delta^n$ such that
\begin{align}\label{eq_duality_relationship}
	\expect[F\delta^n(u)]=\expect[\innerProd[\CM^{\otimes n}]{D^nF}{u}]
\end{align}
for any $F\in\SobSp{n}{2}{\probSp}{\RealNum}$ and $u$ belonging to the domain of $\delta^n$.
Notice that $h\in\CM^{\odot q}$ belongs to the domain of $\delta^q$ and
$\delta^q(h)=\WienerInt{q}(h)$. From the It\^{o}-Wiener expansion and the Stroock formula,
we obtain the product formula:
\begin{align}\label{eq_product_formula}
	\WienerInt{p}(h^{\odot p})\WienerInt{q}(k^{\odot q})
	=
	\sum_{r=0}^{p\wedge q}
		r!
		\binom{p}{r}
		\binom{q}{r}
		(h,k)_\CM^r
		\WienerInt{p+q-2r}(h^{\odot p-r}\odot k^{\odot q-r})
\end{align}
for every $h,k\in\CM$.

In what follows, we assume that fBm $B$ is defined
on the canonical probability space $(\probSp,\sigmaField,\prob)$,
that is, $B(\omega)=\omega$ for $\omega\in\probSp$ is fBm under the probability measure $\prob$.
In this setting, we can apply Gaussian analysis and Malliavin calculus to fBm.
In particular, since $h\in\CM$ is given by $h_t=\expect[ZB_t]$ for some square-integrable random variable $Z$,
we see
$
	|h_t-h_s|
	\leq
		\expect[Z^2]^{1/2}
		\expect[(B_t-B_s)^2]^{1/2}
	=
		\expect[Z^2]^{1/2}
		(t-s)^\Hurst
$,
which implies
$
	\CM
	\subset
		\HolFunc[0]{\Hurst}{[0,1]}{\RealNum}
	\subset
		\HolFunc[0]{\Hurst-\epsilon}{[0,1]}{\RealNum}
$.
From \pref{prop_20140929105014} and the inclusion
$\CM\subset\HolFunc[0]{\Hurst-\epsilon}{[0,1]}{\RealNum}$,
the functional $\omega\mapsto X_t(\omega)$ is Fr\'{e}chet differentiable
in $\CM$ and the derivative is integrable.
Hence we see that $X_t$ is Malliavin differentiable
and have
$
	\innerProd[\CM]{D X_t}{h}
	=
		\nabla_h
		X_t
$
for any $h\in\CM$.
More precisely, we obtain the following proposition.
\begin{proposition}\label{prop_20150520004520}
	Let $b,\sigma\in\SmoothFunc[bdd]{n+1}{\RealNum}{\RealNum}$ for $n\geq 1$.
	Assume that \hyporef{hypo_ellipticity} is satisfied.
	Then $X_t\in\SobSp{n}{\infty-}{\probSp}{\RealNum}$ and
	\begin{align*}
		|\innerProd[\CM^{\odot \nu}]{D^\nu X_t}{h^1\odot\cdots\odot h^\nu}|
		\leq
			\const[\nu]
			\|h^1\|_\infty
			\cdots
			\|h^n\|_\infty,
	\end{align*}
	for any $h^1,\dots,h^\nu\in\CM$ and $1\leq \nu\leq n$.
	Here $\const[\nu]$ is a positive constant depending only on $b,\sigma$ and $\nu$.
\end{proposition}

In what follows, we set
\begin{align*}
	\delta_{st}
	&=
		\mathscr{R}\indicator{[s,t)},&
	\zeta_{st}
	&=
		\mathscr{R}
		\left[
			\frac{1}{2}
			(t-s)
			\indicator{[s,t)}
			-
			\int_s^t
				\indicator{[s,v)}\,
				dv
		\right]
\end{align*}
for $0\leq s<t\leq 1$.
Note
\begin{gather*}
	\hermitePoly{q}
		(
			2^{m\Hurst} B_{\dyadicPart[m]{k-1}\dyadicPart[m]{k}}
		)
	=
		\WienerInt{q}
			((2^{m\Hurst} \delta_{\dyadicPart[m]{k-1}\dyadicPart[m]{k}})^{\odot q})
	=
		2^{mq\Hurst}
		\WienerInt{q}
			(\delta_{\dyadicPart[m]{k-1}\dyadicPart[m]{k}}^{\odot q}),\\
	2^{m(\Hurst+1)}
	\left(
		\frac{1}{2\cdot 2^m}
		B_{\dyadicPart[m]{k-1}\dyadicPart[m]{k}}
		-
		\int_{\dyadicPart[m]{k-1}}^{\dyadicPart[m]{k}}
			B_{\dyadicPart[m]{k-1}u}\,
			du
	\right)
	=
		2^{m(\Hurst+1)}
		\WienerInt{1}(\zeta_{\dyadicPart[m]{k-1}\dyadicPart[m]{k}}).
\end{gather*}
The functions $\delta_{st}$ and $\zeta_{st}$ are bounded functions as follows:
\begin{proposition}\label{prop_20150520010300}
	For any $0\leq s<t\leq 1$, we have
	\begin{gather*}
		\|\delta_{st}\|_\infty
		\leq
			\begin{cases}
				(t-s)^{2\Hurst},&0<\Hurst< 1/2,\\
				2\Hurst (t-s),&1/2\leq\Hurst<1,\\
			\end{cases}\\
		\|\zeta_{st}\|_\infty
		\leq
			\begin{cases}
				\displaystyle
					{
						\left(
							\frac{1}{2}
							+
							\frac{1}{2\Hurst+1}
						\right)
						(t-s)^{2\Hurst+1},
					}
				&
					0<\Hurst< 1/2,\\
				2\Hurst (t-s)^2,
				&
					1/2\leq\Hurst<1.\\
			\end{cases}
	\end{gather*}
\end{proposition}
\begin{proof}
	Note
	\begin{align*}
		|\expect[(B_t-B_s)B_u]|
		\leq
		\begin{cases}
			(t-s)^{2\Hurst},&0<\Hurst< 1/2,\\
			2\Hurst (t-s),&1/2\leq\Hurst<1,\\
		\end{cases}
	\end{align*}
	for any $0\leq s<t\leq 1$ and $0\leq u\leq 1$.
	We can find this estimate in \cite[Lemma~5,6]{NourdinNualartTudor2010}.
	The first assertion follows from this estimate and the identification
	$
		\delta_{st}(u)
		=
			[\mathscr{R}\indicator{[s,t)}](u)
		=
			\expect[(B_t-B_s)B_u]
	$.
	We see the second one from the expression
	\begin{align*}
		\zeta_{st}(u)
		=
			\frac{1}{2}
				(t-s)
				\expect[(B_t-B_s)B_u]
				-
				\int_s^t
					\expect[(B_v-B_s)B_u]\,
					dv.
	\end{align*}
	The proof is completed.
\end{proof}

\section{Proof of \pref{prop_20141102222020}}\label{sec_20150107074711}

In this section, we prove \pref{prop_20141102222020}.
The result of convergence of $(B,\hermiteVar{q}{m})$ can be found in \cite{NourdinNualartTudor2010}.
Main contribution in this section is proof of convergence of $\TRVar{m}$.

Throughout this section, we use the following notation:
\begin{gather*}
	a_{k,l}
	=
		\expect
			\left[
				\left(
					\frac{1}{2}
					B_{k-1,k}
					-
					\int_{k-1}^k
						B_{k-1,u}\,
						du
				\right)
				\left(
					\frac{1}{2}
					B_{l-1,l}
					-
					\int_{l-1}^l
						B_{l-1,v}\,
						dv
				\right)
			\right],\\
	a^\dagger_{k,l}
	=
		\expect
			\left[
				B_{k-1,k}
				\left(
					\frac{1}{2}
					B_{l-1,l}
					-
					\int_{l-1}^l
						B_{l-1,u}\,
						du
				\right)
			\right]
\end{gather*}
for $k,l\geq 1$.
It follows from the stationary increments of fBm that
\begin{align}
	\label{eq_20150518052410}
	a_{k,l}=a_{1,l-k+1},\\
	\label{eq_20150518053842}
	a^\dagger_{k,l}=a^\dagger_{1,l-k+1}
\end{align}
for $1\leq k\leq l$. For the same reason, we have
\begin{align}\label{eq_20150519005630}
	a_{k,k}
	=
		a_{1,1}
	=
		\frac{1}{4}
		\frac{1-\Hurst}{1+\Hurst}.
\end{align}

\subsection{Key estimates}

Before starting to prove \pref{prop_20141102222020},
we show the next three propositions:
\begin{proposition}\label{prop_20150518041212}
	It holds that
	\begin{align*}
		|a_{k,l}|
		\leq
			\const
			\begin{cases}
				|k-l|^{2\Hurst-4},& |k-l|\geq 1,\\
				1,& |k-l|=0
			\end{cases}
	\end{align*}
	for any $k$ and $l$.
\end{proposition}

\begin{proposition}\label{prop_20141122044947}
	It holds that
	\begin{align*}
		|a^\dagger_{k,l}|
		\leq
			\const
			\begin{cases}
				|k-l|^{2\Hurst-3},& |k-l|\geq 1,\\
				1,& |k-l|=0,
			\end{cases}
	\end{align*}
	for any $k,l\geq 1$.
\end{proposition}

\begin{proposition}\label{prop_20141122044103}
	It holds that
	$
		a^\dagger_{k,l}+a^\dagger_{l,k}
		=
			0
	$
	for any $k,l\geq 1$.
\end{proposition}

The following is a key lemma to prove \prefs{prop_20150518041212}{prop_20141122044947}:
\begin{lemma}\label{lem_20141120004521}
	It holds that
	\begin{multline*}
		\expect
			[
				(B_{x+k-1}-B_{s+k-1})
				(B_{y+l-1}-B_{t+l-1})
			]\\
		=
			\frac{1}{2}
			|k-l|^{2\Hurst}
			\left\{
				\binom{2\Hurst}{2}
				\frac{b_2(x,s,y,t)}{(k-l)^2}
				+
				\binom{2\Hurst}{3}
				\frac{b_3(x,s,y,t)}{(k-l)^3}
				+
				R(k-l;x,s,y,t)
			\right\}
	\end{multline*}
	for any $0\leq x,s,y,t\leq 1$ and $k,l\in\NaturalNum$ with $|k-l|\geq 2$.
	Here
	\begin{align*}
		b_2(x,s,y,t)
		&=
			2(xy-xt-sy+st),\\
		b_3(x,s,y,t)
		&=
			3(x^2y-xy^2-x^2t+xt^2-s^2y+sy^2+s^2t-st^2)
	\end{align*}
	and $R$ satisfies $|R(k-l;x,s,y,t)|\leq \const|k-l|^{-4}$ for some positive constant $\const$.
\end{lemma}
\begin{proof}
	From \eqref{eq_20150219103359}, we have
	\begin{multline*}
		\expect
			[
				(B_{x+k-1}-B_{s+k-1})
				(B_{y+l-1}-B_{t+l-1})
			]\\
		\begin{aligned}
			&=
				\frac{1}{2}
				\left\{
					-|x-y+k-l|^{2\Hurst}
					+|x-t+k-l|^{2\Hurst}
					+|s-y+k-l|^{2\Hurst}
					-|s-t+k-l|^{2\Hurst}
				\right\}\\
			&=
				\frac{1}{2}
				|k-l|^{2\Hurst}
				\left\{
					-
					\left|
						1+\frac{x-y}{k-l}
					\right|^{2\Hurst}
					+
					\left|
						1+\frac{x-t}{k-l}
					\right|^{2\Hurst}
					+
					\left|
						1+\frac{s-y}{k-l}
					\right|^{2\Hurst}
					-
					\left|
						1+\frac{s-t}{k-l}
					\right|^{2\Hurst}
				\right\}.
		\end{aligned}
	\end{multline*}
	Applying the binomial theorem, we obtain
	\begin{multline*}
		\expect
			[
				(B_{x+k-1}-B_{s+k-1})
				(B_{y+l-1}-B_{t+l-1})
			]\\
		=
			\frac{1}{2}
			|k-l|^{2\Hurst}
			\left\{
				\sum_{\nu=0}^3
					\binom{2\Hurst}{\nu}
					a_\nu
						\left(
							\frac{x-y}{k-l},\frac{x-t}{k-l},\frac{s-y}{k-l},\frac{s-t}{k-l}
						\right)
				+
				R(k-l;x,s,y,t)
			\right\},
	\end{multline*}
	where
	$
		a_\nu(z_1,z_2,z_3,z_4)=-z_1^\nu+z_2^\nu+z_3^\nu-z_4^\nu
	$
	and $R$ is defined by
	\begin{align*}
		R(k-l;x,s,y,t)
		=
			\left\{
				-
				r_3
					\left(
						\frac{x-y}{k-l}
					\right)
				+
				r_3
					\left(
						\frac{x-t}{k-l}
					\right)
				+
				r_3
					\left(
						\frac{s-y}{k-l}
					\right)
				-
				r_3
					\left(
						\frac{s-t}{k-l}
					\right)
			\right\}
	\end{align*}
	with the remainder term $r_3$.
	Note
	$
		|r_3(\xi)|
		\leq
			\const
			|\xi|^4
	$.
	Expanding the polynomials $a_\nu$, we see
	\begin{gather*}
		\begin{aligned}
			a_0
				\left(
					\frac{x-y}{k-l},\frac{x-t}{k-l},\frac{s-y}{k-l},\frac{s-t}{k-l}
				\right)
			&=
				0,
			&
			a_1
				\left(
					\frac{x-y}{k-l},\frac{x-t}{k-l},\frac{s-y}{k-l},\frac{s-t}{k-l}
				\right)
			&=
				0,
		\end{aligned}\\
		\begin{aligned}
			a_2
				\left(
					\frac{x-y}{k-l},\frac{x-t}{k-l},\frac{s-y}{k-l},\frac{s-t}{k-l}
				\right)
			&=
				\frac{1}{(k-l)^2}
				\cdot
				b_2(x,s,y,t),
		\end{aligned}\\
		\begin{aligned}
			a_3
				\left(
					\frac{x-y}{k-l},\frac{x-t}{k-l},\frac{s-y}{k-l},\frac{s-t}{k-l}
				\right)
			=
				\frac{1}{(k-l)^3}
				\cdot
				b_3(x,s,y,t).
		\end{aligned}
	\end{gather*}
	The proof is completed.
\end{proof}

\begin{proof}[Proof of \pref{prop_20150518041212}]
	The assertion for $|k-l|=0,1$ follows
	from the H\"{o}lder inequality and \eqref{eq_20150519005630}.
	We prove the assertion for $|k-l|\geq 2$.
	Note
	\begin{align*}
		\frac{1}{2}
		B_{k-1,k}
		-
		\int_{k-1}^k
			B_{k-1,u}\,
			du
		&=
			\frac{1}{2}
			(B_k-B_{k-1})
			-
			\int_{k-1}^k
				(B_u-B_{k-1})\,
				du\\
		&=
			\int_{k-1}^k
				du
			\int_{k-1}^k
				\mu_k(d\xi)\,
				(B_\xi-B_u)\\
		&=
			\int_0^1
				ds
			\int_0^1
				\mu_1(dx)\,
				(B_{x+k-1}-B_{s+k-1}).
	\end{align*}
	Here we set $\mu_k=(\delta_k+\delta_{k-1})/2$ by using the Dirac delta function $\delta_a$.
	From this equality, we see
	\begin{align*}
		a_{k,l}
		=
			\int_0^1
				ds
			\int_0^1
				\mu_1(dx)
			\int_0^1
				dt
			\int_0^1
				\mu_1(dy)\,
				\expect
					[
						(B_{x+k-1}-B_{s+k-1})
						(B_{y+l-1}-B_{t+l-1})
					].
	\end{align*}
	Note that $b_2$ and $b_3$ in \lref{lem_20141120004521} satisfy
	\begin{gather*}
		\int_0^1
			ds
		\int_0^1
			\mu_1(dx)
		\int_0^1
			dt
		\int_0^1
			\mu_1(dy)\,
			b_\nu(x,s,y,t)
		=
			0.
	\end{gather*}
	From \lref{lem_20141120004521}, we have
	\begin{align*}
		|a_{k,l}|
		&=
			\left|
				\int_0^1
					ds
				\int_0^1
					\mu_1(dx)
				\int_0^1
					dt
				\int_0^1
					\mu_1(dy)\,
					\frac{1}{2}
					|k-l|^{2\Hurst}
					R(k-l;x,s,y,t)
			\right|\\
		&\leq
			\const|k-l|^{2\Hurst-4},
	\end{align*}
	which implies the conclusion for $|k-l|\geq 2$.
	The proof is completed.
\end{proof}

\begin{proof}[Proof of \pref{prop_20141122044947}]
	The assertion for $|k-l|=0,1$ follows
	from the H\"{o}lder inequality and \eqref{eq_20150519005630}.
	We prove the assertion for $|k-l|\geq 2$.
	We have
	\begin{align*}
		a^\dagger_{k,l}
			&=
				\expect
					\left[
						(B_k-B_{k-1})
						\left(
							\frac{1}{2}
							(B_l-B_{l-1})
							-
							\int_0^1
								(B_{y+l-1}-B_{l-1})\,
								dy
						\right)
					\right]\\
			&=
				\frac{1}{2}
				\expect[(B_k-B_{k-1})(B_l-B_{l-1})]
				-
				\int_0^1
					\expect[(B_k-B_{k-1})(B_{y+l-1}-B_{l-1})]\,
					dy.
	\end{align*}
	From \lref{lem_20141120004521}, we have
	\begin{align*}
		\expect[(B_k-B_{k-1})(B_l-B_{l-1})]
		&=
			\frac{1}{2}
			|k-l|^{2\Hurst}
			\left\{
				\binom{2\Hurst}{2}
				\frac{2}{(k-l)^2}
				+
				R(k-l;1,0,1,0)
			\right\}
	\end{align*}
	and
	\begin{multline*}
		\int_0^1
			\expect[(B_k-B_{k-1})(B_{y+l-1}-B_{l-1})]\,
			dy\\
		\begin{aligned}
			&=
				\frac{1}{2}
				|k-l|^{2\Hurst}
				\left\{
					\binom{2\Hurst}{2}
					\frac{1}{(k-l)^2}
					\int_0^1
						2y\,
						dy
					+
					\binom{2\Hurst}{3}
					\frac{1}{(k-l)^3}
					\int_0^1
						3(y-y^2)\,
						dy
					+
					\int_0^1
						R(k-l;1,0,y,0)\,
						dy
				\right\}\\
			&=
				\frac{1}{2}
				|k-l|^{2\Hurst}
				\left\{
					\binom{2\Hurst}{2}
					\frac{1}{(k-l)^2}
					+
					\binom{2\Hurst}{3}
					\frac{1}{(k-l)^3}
					\frac{1}{2}
					+
					\int_0^1
						R(k-l;1,0,y,0)\,
						dy
				\right\}
		\end{aligned}
	\end{multline*}
	From these equality, we have
	\begin{align*}
		a^\dagger_{k,l}
		&=
			\frac{1}{2}
			|k-l|^{2\Hurst}
			\left\{
				-
				\frac{1}{2}
				\binom{2\Hurst}{3}
				\frac{1}{(k-l)^3}
				+
				\frac{1}{2}
				R(k-l;1,0,1,0)
				-
				\int_0^1
					R(k-l;1,0,y,0)\,
					dy
			\right\}\\
		&=
			-
			\frac{1}{4}
			\binom{2\Hurst}{3}
			\frac{|k-l|^{2\Hurst}}{(k-l)^3}
			+
			\frac{1}{2}
			|k-l|^{2\Hurst}
			\left\{
				\frac{1}{2}
				R(k-l;1,0,1,0)
				-
				\int_0^1
					R(k-l;1,0,y,0)\,
					dy
			\right\}.
	\end{align*}
	Recalling that $R$ satisfies $|R(k-l;x,s,y,t)|\leq \const|k-l|^{-4}$ for some positive constant $\const$,
	we obtain the conclusion.
\end{proof}

\begin{proof}[Proof of \pref{prop_20141122044103}]
	A direct computation yields
	\begin{align}\label{eq_20150508075555}
		a^\dagger_{k,l}
		&=
			\frac{1}{4}
			\left\{
				-|k-l+1|^{2\Hurst}
				+|k-l-1|^{2\Hurst}
			\right\}
			-
			\int_0^1
				\frac{1}{2}
				\left\{
					-|k-l+1-s|^{2\Hurst}
					+|k-l-s|^{2\Hurst}
				\right\}\,
				ds
	\end{align}
	and
	\begin{align}\label{eq_20150508075943}
		a^\dagger_{l,k}
		&=
			\frac{1}{4}
			\left\{
				-|l-k+1|^{2\Hurst}
				+|l-k-1|^{2\Hurst}
			\right\}
			{}
			-
			\int_0^1
				\frac{1}{2}
				\left\{
					-|k-l-t|^{2\Hurst}
					+|k-1+(1-t)|^{2\Hurst}
				\right\}\,
				dt.
	\end{align}
	The assertion follows from these two equalities.

	We see \eqref{eq_20150508075555} as follows:
	\begin{align*}
		a^\dagger_{k,l}
		&=
			\frac{1}{2}
			\expect[(B_k-B_{k-1})(B_l-B_{l-1})]
			-
			\int_0^1
				\expect
					[
						(B_k-B_{k-1})
						(B_{s+l-1}-B_{l-1})
					]\,
				ds\\
		&=
			\frac{1}{2}
			\frac{1}{2}
			\left\{
				|k-l+1|^{2\Hurst}
				+|k-l-1|^{2\Hurst}
				-2|k-l|^{2\Hurst}
			\right\}\\
		&\phantom{=}\quad\qquad
			{}
			-
			\int_0^1
				\frac{1}{2}
				\left\{
					-|k-(s+l-1)|^{2\Hurst}
					+|k-(l-1)|^{2\Hurst}
				\right.\\
		&\phantom{=}\quad\qquad\qquad\qquad\qquad
				\left.{}
					+|(k-1)-(s+l-1)|^{2\Hurst}
					-|(k-1)-(l-1)|^{2\Hurst}
				\right\}\,
				ds.
	\end{align*}
	In order to prove \eqref{eq_20150508075943},
	we exchange $k$ and $l$ in \eqref{eq_20150508075555}
	and obtain
	\begin{align*}
		a^\dagger_{l,k}
		&=
			\frac{1}{4}
			\left\{
				-|l-k+1|^{2\Hurst}
				+|l-k-1|^{2\Hurst}
			\right\}
			-
			\int_0^1
				\frac{1}{2}
				\left\{
					-|l-k+1-s|^{2\Hurst}
					+|l-k-s|^{2\Hurst}
				\right\}\,
				ds.
	\end{align*}
	From the integration by substitution $t=1-s$,
	we see that the integral is equal to
	\begin{align*}
			\int_1^0
				\frac{1}{2}
				\left\{
					-|l-k+t|^{2\Hurst}
					+|l-k-(1-t)|^{2\Hurst}
				\right\}
				(-1)\,
				dt
		=
			\int_0^1
				\frac{1}{2}
				\left\{
					-|k-l-t|^{2\Hurst}
					+|k-1+(1-t)|^{2\Hurst}
				\right\}\,
				dt.
	\end{align*}
	These two equalities imply \eqref{eq_20150508075943}.
\end{proof}

\subsection{Relative compactness and convergence in fdds}
We are ready to prove \pref{prop_20141102222020}.
We show relative compactness and convergence in the sense of finite-dimensional distributions (fdds).
\begin{lemma}\label{lem_20150519013003}
	Under the assumption of \pref{prop_20141102222020},
	the sequence $\{(B,\hermiteVar{q}{m},\TRVar{m})\}_{m=1}^\infty$ is relative compact
	in the Skorokhod topology.
\end{lemma}

\begin{lemma}\label{lem_20150519013014}
	Under the assumption of \pref{prop_20141102222020},
	the sequence $\{(B,\hermiteVar{q}{m},\TRVar{m})\}_{m=1}^\infty$ converges in the sense of fdds.
	More precisely, we have, for $0\leq s_1<t_1\leq\dots\leq s_d<t_d\leq 1$,
	\begin{multline*}
		\lim_{m\to\infty}
			\left(
				B_{t_1}-B_{s_1},\hermiteVar{q}{m}(t_1)-\hermiteVar{q}{m}(s_1),\TRVar{m}(t_1)-\TRVar{m}(s_1),
				\dots,
			\right.\\
			\left.
				B_{t_d}-B_{s_d},\hermiteVar{q}{m}(t_d)-\hermiteVar{q}{m}(s_d),\TRVar{m}(t_d)-\TRVar{m}(s_d)
			\right)\\
			=
				\left(
					B_{t_1}-B_{s_1},
					\sigma_\Hurst (W_{t_1}-W_{s_1}),
					\tilde{\sigma}_\Hurst (\tilde{W}_{t_1}-\tilde{W}_{s_1}),
					\dots,
				\right.\\
				\left.
					B_{t_d}-B_{s_d},
					\sigma_\Hurst (W_{t_d}-W_{s_d}),
					\tilde{\sigma}_\Hurst (\tilde{W}_{t_d}-\tilde{W}_{s_d})
				\right)
	\end{multline*}
	weakly in $(\RealNum^d)^3$,
	where
	$W$ and $\tilde{W}$ are standard Brownian motions
	and $B$, $W$ and $\tilde{W}$ are independent.
\end{lemma}

Before beginning our discussion, we note that, for any $0\leq s<t\leq 1$ and $0\leq u<v\leq 1$,
\begin{align}\label{eq_20150518045935}
	\expect[\{\TRVar{m}(t)-\TRVar{m}(s)\}\{\TRVar{m}(v)-\TRVar{m}(u)\}]
	=
		\frac{1}{2^m}
		\sum_{k=\intPart{2^m s}+1}^{\intPart{2^m t}}
		\sum_{l=\intPart{2^m u}+1}^{\intPart{2^m v}}
			a_{k,l}.
\end{align}
Applying \eqref{eq_20150518052410} to \eqref{eq_20150518045935}, we see
\begin{multline}\label{eq_20150518051514}
	\expect[\{\TRVar{m}(t)-\TRVar{m}(s)\}^2]\\
	\begin{aligned}
		&=
			\frac{\intPart{2^m t}-\intPart{2^m s}}{2^m}
			\left(
				a_{1,1}
				+
				2
				\sum_{j=1}^{\intPart{2^m t}-\intPart{2^m s}-1}
					a_{1,j+1}
			\right)
			-
			\frac{2}{2^m}
			\sum_{j=1}^{\intPart{2^m t}-\intPart{2^m s}-1}
				ja_{1,j+1}.
	\end{aligned}
\end{multline}

\begin{proof}[Proof of \lref{lem_20150519013003}]
	The assertion follows from
	\begin{gather*}
		\expect[\{\hermiteVar{q}{m}(t)-\hermiteVar{q}{m}(s)\}^4]
		\leq
			\const
			\left(
				\frac{\intPart{2^m t}-\intPart{2^m s}}{2^m}
			\right)^2,\\
		\expect[\{\TRVar{m}(t)-\TRVar{m}(s)\}^4]
		\leq
			\const
			\left(
				\frac{\intPart{2^m t}-\intPart{2^m s}}{2^m}
			\right)^2
	\end{gather*}
	for any $0\leq s<t\leq 1$ and some constant $\const$.
	The first estimate is proved in \cite{NourdinNualartTudor2010}.
	Combining \eqref{eq_20150518051514} and \pref{prop_20150518041212}, we see
	\begin{align*}
		\expect[\{\TRVar{m}(t)-\TRVar{m}(s)\}^2]
		&\leq
			\const
			\frac{\intPart{2^m t}-\intPart{2^m s}}{2^m}.
	\end{align*}
	Since $\TRVar{m}(t)-\TRVar{m}(s)$ is a Gaussian random variable,
	we have the second estimate.
\end{proof}

\begin{proof}[Proof of \lref{lem_20150519013014}]
	We show
	\begin{gather}
		\label{eq_20150514085853}
		\lim_{m\to\infty}
			\expect[\{\hermiteVar{q}{m}(t)-\hermiteVar{q}{m}(s)\}^4]
		=
			3\sigma_\Hurst^4(t-s)^2,\\
		\label{eq_20150514085924}
		\lim_{m\to\infty}
			\expect
				[
					\{\hermiteVar{q}{m}(t)-\hermiteVar{q}{m}(s)\}^2
				]
		=
			\sigma_\Hurst^2(t-s),\\
		\label{eq_20150514090122}
		\lim_{m\to\infty}
			\expect
				[
					\{\TRVar{m}(t)-\TRVar{m}(s)\}^2
				]
		=
			\tilde{\sigma}_\Hurst^2(t-s),\\
		\label{eq_20150514090838}
		\lim_{m\to\infty}
			\expect
				[
					\{B_t-B_s\}
					\{\hermiteVar{q}{m}(t)-\hermiteVar{q}{m}(s)\}
				]
		=
			0,\\
		\label{eq_20150514090844}
		\lim_{m\to\infty}
			\expect
				[
					\{\hermiteVar{q}{m}(t)-\hermiteVar{q}{m}(s)\}
					\{\TRVar{m}(t)-\TRVar{m}(s)\}
				]
		=
			0,\\
		\label{eq_20150514091605}
		\lim_{m\to\infty}
			\expect
				[
					\{B_t-B_s\}
					\{\TRVar{m}(t)-\TRVar{m}(s)\}
				]
		=
			0
	\end{gather}
	for $0\leq s<t\leq 1$ and
	\begin{gather}
		\label{eq_20150514090229}
		\lim_{m\to\infty}
			\expect
				[
					\{\hermiteVar{q}{m}(t)-\hermiteVar{q}{m}(s)\}
					\{\hermiteVar{q}{m}(v)-\hermiteVar{q}{m}(u)\}
				]
		=
			0,\\
		\label{eq_20150514090334}
		\lim_{m\to\infty}
			\expect
				[
					\{\TRVar{m}(t)-\TRVar{m}(s)\}
					\{\TRVar{m}(v)-\TRVar{m}(u)\}
				]
		=
			0,\\
		\label{eq_20150514090939}
		\lim_{m\to\infty}
			\expect
				[
					\{B_t-B_s\}
					\{\hermiteVar{q}{m}(v)-\hermiteVar{q}{m}(u)\}
				]
		=
			0,\\
		\label{eq_20150514091004}
		\lim_{m\to\infty}
			\expect
				[
					\{\hermiteVar{q}{m}(t)-\hermiteVar{q}{m}(s)\}
					\{\TRVar{m}(v)-\TRVar{m}(u)\}
				]
		=
			0,\\
		\label{eq_20150514091618}
		\lim_{m\to\infty}
			\expect
				[
					\{B_t-B_s\}
					\{\TRVar{m}(v)-\TRVar{m}(u)\}
				]
		=
			0
	\end{gather}
	for $0\leq s<t\leq 1$ and $0\leq u<v\leq 1$ with $(s,t)\cap(u,v)=\emptyset$.
	From these convergence and the fourth moment theorem in \cite{PeccatiTudor2005},
	we see the assertion.

	The convergence \eqref{eq_20150514085853}, \eqref{eq_20150514085924}
	and \eqref{eq_20150514090229} are proved in \cite{NourdinNualartTudor2010}.

	We consider \eqref{eq_20150514090122} and \eqref{eq_20150514090334}.
	Both convergence follows from \eqref{eq_20150518051514} and \pref{prop_20150518041212}.
	In particular, \eqref{eq_20150514090122} is a direct consequence from them.
	We show \eqref{eq_20150514090334} for $s<t\leq u<v$.
	From \eqref{eq_20150518045935} and \eqref{eq_20150518052410},
	we have
	\begin{multline*}
		|
			\expect[\{\TRVar{m}(t)-\TRVar{m}(s)\}\{\TRVar{m}(v)-\TRVar{m}(u)\}]
		|\\
		\leq
			\frac{1}{2^m}
			\sum_{k=\intPart{2^m s}+1}^{\intPart{2^m t}}
			\sum_{l=\intPart{2^m u}+1}^{\intPart{2^m v}}
				|a_{1,l-k+1}|
		\leq
			\frac{1}{2^m}
			\sum_{j=\intPart{2^m u}+1-\intPart{2^m t}}^{\intPart{2^m v}-\intPart{2^m s}-1}
				j|a_{1,j+1}|
		\leq
			\frac{1}{2^m}
			\sum_{j=1}^{2^m}
				j|a_{1,j+1}|.
	\end{multline*}
	Combining this estimate and \pref{prop_20150518041212},
	we obtain \eqref{eq_20150514090334}.

	We study the equalities \eqref{eq_20150514090838}, \eqref{eq_20150514090844},
	\eqref{eq_20150514090939} and \eqref{eq_20150514091004}.
	Since $B_t-B_s$, $\hermiteVar{q}{m}(t)-\hermiteVar{q}{m}(s)$ and $\TRVar{m}(t)-\TRVar{m}(s)$
	belongs to first, $q$-th, first Wiener chaos,
	the expectations in  \eqref{eq_20150514090838} and \eqref{eq_20150514090844} are equal to $0$.
	The same reason yields \eqref{eq_20150514090939} and \eqref{eq_20150514091004}.

	We prove \eqref{eq_20150514091605} and \eqref{eq_20150514091618}.
	Set
	$
		B^{(m)}_t
		=
			B_{\intPart{2^m t}/2^m}
		=
			\sum_{k=1}^{\intPart{2^m t}}
				B_{\dyadicPart[m]{k-1}\dyadicPart[m]{k}}
	$.
	We decompose
	$
		\expect
			[\{B_t-B_s\}\{\TRVar{m}(v)-\TRVar{m}(u)\}]
	$
	into
	$
		I^{(m)}
		+
			\expect
				[
					\{B^{(m)}_t-B^{(m)}_s\}
					\{\TRVar{m}(v)-\TRVar{m}(u)\}
				]
		+
			J^{(m)}
	$, where
	\begin{align*}
		I^{(m)}
		&=
			\expect
				[
					\{B_t-B^{(m)}_t\}
					\{\TRVar{m}(v)-\TRVar{m}(u)\}
				],\\
		J^{(m)}
		&=
			\expect
				[
					\{B^{(m)}_s-B_s\}
					\{\TRVar{m}(v)-\TRVar{m}(u)\}
				].
	\end{align*}

	We can show convergence of $I^{(m)}$ and $J^{(m)}$ easily.
	In fact, we see
	\begin{align*}
		|I^{(m)}|
		&\leq
			\expect[\{B_t-B_{\intPart{2^m t}/2^m}\}^2]^{1/2}
			\expect[\{\TRVar{m}(v)-\TRVar{m}(u)\}^2]^{1/2}\\
		&\leq
			\left(
				t-\frac{\intPart{2^m t}}{2^m}
			\right)^\Hurst
			\left(
				\const
				\frac{\intPart{2^m u}-\intPart{2^m v}}{2^m}
			\right)^{1/2}.
	\end{align*}
	The same inequality holds for $J^{(m)}$.
	Hence we see the convergences.

	We consider convergence of
	$
		\expect
			[
				\{B^{(m)}_t-B^{(m)}_s\}
				\{\TRVar{m}(v)-\TRVar{m}(u)\}
			]
	$.
	Note
	\begin{align*}
		\expect[\{B^{(m)}_t-B^{(m)}_s\}\{\TRVar{m}(v)-\TRVar{m}(u)\}]
		=
			2^{-m(1/2+\Hurst)}
			\sum_{k=\intPart{2^m s}+1}^{\intPart{2^m t}}
			\sum_{l=\intPart{2^m u}+1}^{\intPart{2^m v}}
				a^\dagger_{k,l}.
	\end{align*}
	In the case that $s=u$ and $t=v$, we see
	\begin{multline*}
		\expect[\{B^{(m)}_t-B^{(m)}_s\}\{\TRVar{m}(t)-\TRVar{m}(s)\}]\\
		=
			2^{-m(1/2+\Hurst)}
			\sum_{\intPart{2^m s}+1}^{\intPart{2^m t}}
				a^\dagger_{k,k}
			+
			2^{-m(1/2+\Hurst)}
			\sum_{\intPart{2^m s}+1\leq k<l\leq\intPart{2^m t}}
				(a^\dagger_{k,l}+a^\dagger_{l,k})
		=
			0.
	\end{multline*}
	In the last line, we used \pref{prop_20141122044103}.
	From this, we see \eqref{eq_20150514091605}.

	In the case that $0\leq s<t\leq u<v\leq 1$,
	by noting \eqref{eq_20150518053842},
	we have
	\begin{multline*}
		|\expect[\{B^{(m)}_t-B^{(m)}_s\}\{\TRVar{m}(u)-\TRVar{m}(v)\}]|\\
		\begin{aligned}
			\leq
				2^{-m(1/2+\Hurst)}
				\sum_{k=\intPart{2^m s}+1}^{\intPart{2^m t}}
				\sum_{l=\intPart{2^m u}+1}^{\intPart{2^m v}}
					|a^\dagger_{1,l-k+1}|
			\leq
				2^{-m(1/2+\Hurst)}
				\sum_{j=\intPart{2^m u}-\intPart{2^m t}+1}^{\intPart{2^m v}-\intPart{2^m s}-1}
					j|a^\dagger_{1,j+1}|.
		\end{aligned}
	\end{multline*}
	From \pref{prop_20141122044947}, we see
	\begin{align*}
		|\expect[\{B^{(m)}_t-B^{(m)}_s\}\{\TRVar{m}(u)-\TRVar{m}(v)\}]|
		\leq
			\const
			2^{-m(1/2+\Hurst)}
			\sum_{j=1}^{2^m}
				j
				\cdot
				j^{2\Hurst-3}.
	\end{align*}
	In the case that $0\leq u<v\leq s<t\leq 1$, we obtain the same inequality.
	We complete the proof of \eqref{eq_20150514091618}.
\end{proof}

\section{Proof of convergence of variation functionals}
\label{sec_20170519021523}

\subsection{Estimate on $\tilde{U}_m$}

In this subsection, we prove \tref{thm_20141123062022}.
At the beginning, we give an estimate of
$\expect[|\wTRVar{}{m}(t)-\wTRVar{}{m}(s)|^2]$.
\begin{proposition}\label{prop_20150518032235}
	There exists a positive constant $\const$ independent of $m$ such that
	\begin{align*}
		\left|
			\expect
				[
					g(X_s)
					g(X_t)
					\WienerInt{2}
						(
							\zeta_{\dyadicPart[m]{k-1}\dyadicPart[m]{k}}
							\odot
							\zeta_{\dyadicPart[m]{l-1}\dyadicPart[m]{l}}
						)
				]
		\right|
		\leq
			\const
			\begin{cases}
				2^{-m(4\Hurst+2)},
				&
					0<\Hurst< 1/2,\\
				2^{-4m},
				&
					1/2\leq\Hurst<1\\
			\end{cases}
	\end{align*}
	for any $0\leq s,t\leq 1$ and $1\leq k, l\leq 2^m$.
\end{proposition}
\begin{proof}
	From the duality relationship \eqref{eq_duality_relationship}, we have
	\begin{align*}
		\expect
			[
				g(X_s)
				g(X_t)
				\WienerInt{2}
					(
						\zeta_{\dyadicPart[m]{k-1}\dyadicPart[m]{k}}
						\odot
						\zeta_{\dyadicPart[m]{l-1}\dyadicPart[m]{l}}
					)
			]
		&=
			\expect
				\left[
					\innerProdLarge[\CM^{\odot 2}]
						{
							D^2
								\left\{
									g(X_s)
									g(X_t)
								\right\}
						}
						{
							\zeta_{\dyadicPart[m]{k-1}\dyadicPart[m]{k}}
							\odot
							\zeta_{\dyadicPart[m]{l-1}\dyadicPart[m]{l}}
						}
				\right]
    \end{align*}
    and the Leibniz rule implies
    \begin{align*}
        D^2
            \left\{
                g(X_s)
                g(X_t)
            \right\}
        &=
            g'(X_s)
            g(X_t)
            D^2 X_s
            +
            g''(X_s)
            g(X_t)
            (DX_s)^{\odot 2}\\
        &\phantom{=}\quad
            +
            2
            g'(X_s)
            g'(X_t)
            D X_s
            \odot
            D X_t
            +
            g(X_s)
            g''(X_t)
            (DX_t)^{\odot 2}
            +
            g(X_s)
            g'(X_t)
            D^2 X_t.
    \end{align*}
	From the H\"{o}lder inequality and \pref{prop_20150520004520}, we have
	\begin{align*}
		\expect
			\left[
				g'(X_s)
				g(X_t)
				\innerProdLarge[\CM^{\odot 2}]
						{
							D^2 X_s
						}
						{
							\zeta_{\dyadicPart[m]{k-1}\dyadicPart[m]{k}}
							\odot
							\zeta_{\dyadicPart[m]{l-1}\dyadicPart[m]{l}}
						}
			\right]
		&\leq
			\expect
				[
					|
						g'(X_{\dyadicPart[m]{k-1}})
						g(X_{\dyadicPart[m]{l-1}})
					|^2
				]^{1/2}
			\cdot
			\const
			\|\zeta_{\dyadicPart[m]{k-1}\dyadicPart[m]{k}}\|_\infty
			\|\zeta_{\dyadicPart[m]{l-1}\dyadicPart[m]{l}}\|_\infty\\
		&\leq
			\const'
			\begin{cases}
				\displaystyle
					{
						\left(
							\frac{1}{2}
							+
							\frac{1}{2\Hurst+1}
						\right)^2
						(2^{-m(2\Hurst+1)})^2,
					}
				&
					0<\Hurst< 1/2,\\
				(2\Hurst)^2 (2^{-2m})^2,
				&
					1/2\leq\Hurst<1.\\
			\end{cases}
	\end{align*}
	In the last line, we used \pref{prop_20150520010300}
	and the constant $\const$ and $\const'$ are independent of $m$.
	Since the other terms in the above also admit similar estimates, we see the assertion.
\end{proof}

\begin{proposition}\label{prop_20141104000212}
	There exists a positive constant $\const$ independent of $m$ such that
	\begin{align*}
		\expect
			[|\wTRVar{}{m}(t)-\wTRVar{}{m}(s)|^2]
		\leq
			\const
			\cdot
			\frac{\intPart{2^m t}-\intPart{2^m s}}{2^m}
			\cdot
			\begin{cases}
				2^{-4m\Hurst},	&	0<\Hurst<1/2,\\
				2^{-2m},		&	1/2\leq\Hurst<1\\
			\end{cases}
	\end{align*}
	for any $0\leq s,t\leq 1$.
\end{proposition}

\begin{proof}
	From the product formula \eqref{eq_product_formula}, we have
	\begin{align*}
		|\wTRVar{}{m}(t)-\wTRVar{}{m}(s)|^2
		=
			\sum_{k,l=\intPart{2^m s}+1}^{\intPart{2^m t}}
				g(X_{\dyadicPart[m]{k-1}})
				g(X_{\dyadicPart[m]{l-1}})
				\WienerInt{1}(\zeta_{\dyadicPart[m]{k-1}\dyadicPart[m]{k}})
				\WienerInt{1}(\zeta_{\dyadicPart[m]{l-1}\dyadicPart[m]{l}})
		=
			S+T,
	\end{align*}
	where
	\begin{align*}
		S
		&=
			\sum_{k,l=\intPart{2^m s}+1}^{\intPart{2^m t}}
				g(X_{\dyadicPart[m]{k-1}})
				g(X_{\dyadicPart[m]{l-1}})
				\innerProd[\CM]
					{\zeta_{\dyadicPart[m]{k-1}\dyadicPart[m]{k}}}
					{\zeta_{\dyadicPart[m]{l-1}\dyadicPart[m]{l}}},\\
		T
		&=
			\sum_{k,l=\intPart{2^m s}+1}^{\intPart{2^m t}}
				g(X_{\dyadicPart[m]{k-1}})
				g(X_{\dyadicPart[m]{l-1}})
				\WienerInt{2}
					(
						\zeta_{\dyadicPart[m]{k-1}\dyadicPart[m]{k}}
						\odot
						\zeta_{\dyadicPart[m]{l-1}\dyadicPart[m]{l}}
					).
	\end{align*}
	We estimate the expectations $\expect[S]$ and $\expect[T]$.

	The expectation $|\expect[S]|$ is estimated by
	\begin{align*}
		|\expect[S]|
		&=
			\left|
				\sum_{k,l=\intPart{2^m s}+1}^{\intPart{2^m t}}
					\expect
						[
							g(X_{\dyadicPart[m]{k-1}})
							g(X_{\dyadicPart[m]{l-1}})
						]
					\innerProd[\CM]
						{\zeta_{\dyadicPart[m]{k-1}\dyadicPart[m]{k}}}
						{\zeta_{\dyadicPart[m]{l-1}\dyadicPart[m]{l}}}
			\right|\\
		&\leq
			\left(
				\sup_{0\leq t\leq 1}
					\expect[|g(X_t)|^2]
			\right)
			\sum_{k,l=\intPart{2^m s}+1}^{\intPart{2^m t}}
				|
					\innerProd[\CM]
						{\zeta_{\dyadicPart[m]{k-1}\dyadicPart[m]{k}}}
						{\zeta_{\dyadicPart[m]{l-1}\dyadicPart[m]{l}}}
				|.
	\end{align*}
	Combining the self-similarity of fBm and \pref{prop_20150518041212}, we have
	\begin{align*}
		|\expect[S]|
		&\leq
			\left(
				\sup_{0\leq t\leq 1}
					\expect[|g(X_t)|^2]
			\right)
			\cdot
			2^{-m(2\Hurst+2)}
			\cdot
			\const
			(
				\intPart{2^m t}-\intPart{2^m s}
			)\\
		&=
			\const
			\left(
				\sup_{0\leq t\leq 1}
					\expect[|g(X_t)|^2]
			\right)
			\frac{\intPart{2^m t}-\intPart{2^m s}}{2^m}
			\cdot
			2^{-m(2\Hurst+1)}.
	\end{align*}

	We evaluate the expectation $\expect[T]$.
	From \pref{prop_20150518032235}, we obtain
	\begin{align*}
		\left|
			\expect[T]
		\right|
		&\leq
			(\intPart{2^m t}-\intPart{2^m s})^2
			\cdot
			\const
			\begin{cases}
				2^{-2m(4\Hurst+2)},		&	0<\Hurst< 1/2,\\
				2^{-4m},				&	1/2\leq\Hurst<1,\\
			\end{cases}\\
		&\leq
			\const
			\frac{\intPart{2^m t}-\intPart{2^m s}}{2^m}
			\cdot
			\begin{cases}
				2^{-4m\Hurst},
				&
					0<\Hurst< 1/2,\\
				2^{-2m},
				&
					1/2\leq\Hurst<1.\\
			\end{cases}
	\end{align*}
	The proof is completed.
\end{proof}

\begin{proof}[Proof of \pref{prop_20141117062354}]
	From \pref{prop_20141104000212}, we have
	\begin{align*}
		\expect
			[
				|
					2^{mr}
					\wTRVar{}{m}(t)
					-
					2^{mr}
					\wTRVar{}{m}(s)
				|^2
			]
		&=
			2^{2mr}
			\expect[|\wTRVar{}{m}(t)-\wTRVar{}{m}(s)|^2]\\
		&\leq
			\const
			\frac{\intPart{2^m t}-\intPart{2^m s}}{2^m}
			\cdot
			\begin{cases}
				2^{-2m(2\Hurst-r)},	&	0<\Hurst<1/2,\\
				2^{-2m(1-r)},			&	1/2\leq\Hurst<1.\\
			\end{cases}
	\end{align*}
	This inequality implies convergence of in the sense of fdds
	and relative compactness. For relative compactness, see \cite[Cororally~2.2]{BurdzySwanson2010}.
	The proof is completed.
\end{proof}

\begin{proof}[Proof of \tref{thm_20141123062022}]
	The assertion follows from convergence of in the sense of fdds
	and relative compactness of $\{(B,2^{-m/2}\wHerVar{q}{m},2^{m}\wTRVar{}{m})\}_{m=1}^\infty$,
	that is, we obtain \tref{thm_20141123062022}
	from the following \lrefs{prop_20141125002248}{prop_20141125002312}.
\end{proof}

\begin{lemma}\label{prop_20141125002248}
	Let $0\leq t_1<\dots<t_d\leq 1$.
	Under the assumption of \tref{thm_20141123062022},
	we have
	\begin{multline}\label{eq_20141125004539}
		\lim_{m\to\infty}
			\left(
				B_{t_1},2^{-m/2}\wHerVar{q}{m}(t_1),2^{m}\wTRVar{}{m}(t_1),
				\dots,
				B_{t_d},2^{-m/2}\wHerVar{q}{m}(t_d),2^{m}\wTRVar{}{m}(t_d)
			\right)\\
		\begin{aligned}
			&=
				\left(
					B_{t_1},
					\sqrt{q!}
					\int_0^{t_1}
						f(X_s)\,
						dW_s,
					\frac{1}{\sqrt{12}}
					\int_0^{t_1}
						g(X_s)\,
						d\tilde{W}_s,
					\dots,
				\right.\\
			&\phantom{=}\quad\qquad\qquad\qquad
				\left.
					B_{t_d},
					\sqrt{q!}
					\int_0^{t_d}
						f(X_s)\,
						dW_s,
					\frac{1}{\sqrt{12}}
					\int_0^{t_d}
						g(X_s)\,
						d\tilde{W}_s,
				\right)
		\end{aligned}
	\end{multline}
	weakly in $(\RealNum^d)^3$,
	where
	$W$ and $\tilde{W}$ are standard Brownian motions
	and $B$, $W$ and $\tilde{W}$ are independent.
\end{lemma}
\begin{lemma}\label{prop_20141125002312}
	Under the assumption of \tref{thm_20141123062022},
	$\{(B,2^{-m/2}\wHerVar{q}{m},2^{m}\wTRVar{}{m})\}_{m=1}^\infty$
	is relative compact in the Skorokhod topology.
\end{lemma}

\begin{proof}[Proof of \lref{prop_20141125002248}]
	We decompose $\wHerVar{q}{m}(t)$ and $\wTRVar{}{m}(t)$ into
	$
		\wHerVar{q}{m,n}(t)
		+
		R^{(m,n)}(t)
	$
	and
	$
		\wTRVar{}{m,n}(t)
		+
		\tilde{R}^{(m,n)}(t)
	$ for $m\geq n$, respectively, where
	\begin{align*}
		\wHerVar{q}{m,n}(t)
		&=
			\sum_{k=1}^{\intPart{2^m t}}
				f(X_{\lowerPart[n]{\dyadicPart[m]{k-1}}})
				\hermitePoly{q}(2^{m\Hurst}B_{\dyadicPart[m]{k-1}\dyadicPart[m]{k}}),\\
		R^{(m,n)}(t)
		&=
			\sum_{k=1}^{\intPart{2^m t}}
				\left\{
					F_{\dyadicPart[m]{k-1}\dyadicPart[m]{k}}(X)
					-
					f(X_{\lowerPart[n]{\dyadicPart[m]{k-1}}})
				\right\}
				\hermitePoly{q}(2^{m\Hurst}B_{\dyadicPart[m]{k-1}\dyadicPart[m]{k}}),\\
		\wTRVar{}{m,n}(t)
		&=
			\sum_{k=1}^{\intPart{2^m t}}
				g(X_{\lowerPart[n]{\dyadicPart[m]{k-1}}})
				\WienerInt{1}(\zeta_{\dyadicPart[m]{k-1}\dyadicPart[m]{k}}),\\
		\tilde{R}^{(m,n)}(t)
		&=
			\sum_{k=1}^{\intPart{2^m t}}
				\left\{
					g(X_{\dyadicPart[m]{k-1}})
					-
					g(X_{\lowerPart[n]{\dyadicPart[m]{k-1}}})
				\right\}
				\WienerInt{1}(\zeta_{\dyadicPart[m]{k-1}\dyadicPart[m]{k}}).
	\end{align*}
	Here
	$
		\lowerPart[n]{t}
		=
			\sup
				\{
					\dyadicPart[n]{k};
					\dyadicPart[n]{k}\leq t,k=0,\dots,2^n-1
				\}
	$.
	We prove
	\begin{enumerate}
		\item	\label{item_1506912770}
				The sequence
				$
					\{
						\{
							(
								B_{t_\alpha},
								2^{-m/2}\wHerVar{q}{m,n}(t_\alpha),
								2^m \wTRVar{}{m,n}(t_\alpha)
							)_{\alpha=1}^d
						\}_{m=n}^\infty
					\}_{n=1}^\infty
				$
				converges to the right-hand side of \eqref{eq_20141125004539}
				as $m\to\infty$ and $n\to\infty$.
		\item	\label{item_1506912803}
				$
					\displaystyle
						{
							\lim_{n\to\infty}
							\limsup_{m\to\infty}
								\expect
									[|2^{-m/2}R^{(m,n)}(t_\alpha)|^2]
							=
								0
						}
				$
				for $\alpha=1,\dots,d$,
		\item	\label{item_1506912832}
				$
					\displaystyle
						{
							\lim_{n\to\infty}
							\limsup_{m\to\infty}
								\expect
									[|2^m\tilde{R}^{(m,n)}(t_\alpha)|^2]
							=
								0
						}
				$
				for $\alpha=1,\dots,d$.
	\end{enumerate}

	Assertion~(\ref{item_1506912770}) is a direct consequence of \pref{prop_20141102222020}.
	To show Assertion~(\ref{item_1506912803}), we use the product formula \eqref{eq_product_formula}
	and estimate the expectations. For detail, see {\cite[Lemmas~22 and 23]{Naganuma2015}}.

	In the rest of this proof we show Assertion~(\ref{item_1506912832})
	by using independent increments of the standard Brownian motion $B$.
	Set
	$
		\tilde{Y}^{(m,n)}_k
		=
			\{
				g(X_{\dyadicPart[m]{k-1}})
				-
				g(X_{\lowerPart[n]{\dyadicPart[m]{k-1}}})
			\}
			\WienerInt{1}(\zeta_{\dyadicPart[m]{k-1}\dyadicPart[m]{k}})
	$
    and
    $
        \sigmaField_{t}
        =
            \sigma(B_u;0\leq u\leq t)
    $.
	Then, for $k<l$, random variables $\tilde{Y}^{(m,n)}_k$
	and
	$
		g(X_{\dyadicPart[m]{k-1}})
		-
		g(X_{\lowerPart[n]{\dyadicPart[m]{k-1}}})
	$
	are $\sigmaField_{\dyadicPart[m]{l-1}}$-measurable.
	In addition,
	$\WienerInt{1}(\zeta_{\dyadicPart[m]{l-1}\dyadicPart[m]{l}})$
	is independent of $\sigmaField_{\dyadicPart[m]{l-1}}$.
    This implies
    $
        \expect
            [
                \WienerInt{1}(\zeta_{\dyadicPart[m]{l-1}\dyadicPart[m]{l}})
                |\sigmaField_{\dyadicPart[m]{l-1}}
            ]
        =
            \expect
                [
                    \WienerInt{1}(\zeta_{\dyadicPart[m]{l-1}\dyadicPart[m]{l}})
                ]
        =
            0
    $
    a.s.
	Hence, we have
	\begin{align*}
		\expect
			[
				\tilde{Y}^{(m,n)}_k
				\tilde{Y}^{(m,n)}_l
				|\sigmaField_{\dyadicPart[m]{l-1}}
			]
		=
			\tilde{Y}^{(m,n)}_k
            \{
				g(X_{\dyadicPart[m]{k-1}})
				-
				g(X_{\lowerPart[n]{\dyadicPart[m]{k-1}}})
			\}
			\expect
				[
					\WienerInt{1}(\zeta_{\dyadicPart[m]{l-1}\dyadicPart[m]{l}})
					|\sigmaField_{\dyadicPart[m]{l-1}}
				]
        =
            0
	\end{align*}
	a.s.\ for $k<l$. From this, we obtain
	$
		\expect
			[
				\tilde{Y}^{(m,n)}_k
				\tilde{Y}^{(m,n)}_l
			]
		=
			0
	$
	for $k\neq l$. In addition, we have
	\begin{align*}
		\expect
			[
				|\tilde{Y}^{(m,n)}_k|^2
			]
		\leq
			\expect
				[
					\{
						g(X_{\dyadicPart[m]{k-1}})
						-
						g(X_{\lowerPart[n]{\dyadicPart[m]{k-1}}})
					\}^4
				]^{1/2}
			\expect
				[
					\WienerInt{1}(\zeta_{\dyadicPart[m]{k-1}\dyadicPart[m]{k}})^4
				]^{1/2}
		\leq
			\const
			2^{-n}
			2^{-3m}
	\end{align*}
	for some constant $\const$.
	From these, we obtain
	\begin{align*}
		\expect[|2^m\tilde{R}^{(m,n)}(t)|^2]
		=
			2^{2m}
			\sum_{k=1}^{\intPart{2^m t}}
				\expect[|\tilde{Y}^{(m,n)}_k|^2]
		\leq
			2^{2m}
			\cdot
			2^m
			\cdot
			\const
			2^{-n}
			2^{-3m}
		=
			\const
			2^{-n},
	\end{align*}
	which implies the third assertion.

	The proof is completed.
\end{proof}

\begin{proof}[Proof of \lref{prop_20141125002312}]
	We can prove the assertion in the same way as {\cite[Proposition~18]{Naganuma2015}}.
	In the proof, we shall show that the processes satisfy some kind of moment condition for relative compactness.
\end{proof}

\subsection{Weighted Hermite and power variations}
\label{sec_20150225090808}
In this subsection, we prove \tref[thm_20141203094309]{thm_20141117071529}.

At the beginning, we give an estimate of
$\expect[|\wHerVar{q}{m}(t)-\wHerVar{q}{m}(s)|^2]$
\begin{proposition}\label{prop_20141117073846}
	Let $\mu$ and $\nu$ be probability measure on $[0,1]$
	and $f,g\in\SmoothFunc[poly]{2q}{\RealNum}{\RealNum}$.
	Then there exists a constant $\const$ such that
	\begin{align*}
		\left|
			\expect
				\left[
					\WienerInt{2q-2r}
						(
							\delta_{\dyadicPart[m]{k-1}\dyadicPart[m]{k}}^{\odot q-r}
							\odot
							\delta_{\dyadicPart[m]{l-1}\dyadicPart[m]{l}}^{\odot q-r}
						)
					F^{f,\mu}_{st}(X)
					F^{g,\nu}_{uv}(X)
				\right]
		\right|
		\leq
			\const
			\begin{cases}
				2^{-4m(q-r)\Hurst},&0<\Hurst<1/2,\\
				2^{-2m(q-r)},&1/2\leq\Hurst<1,\\
			\end{cases}
	\end{align*}
	for any $0\leq s<t\leq 1$, $0\leq u<v\leq 1$
	and $1\leq k,l\leq 2^m$.
\end{proposition}
\begin{proof}
	From the duality relationship \eqref{eq_duality_relationship} and
	the Leibniz rule, we see
	\begin{multline*}
		\expect
			\left[
				\WienerInt{2q-2r}
					(
						\delta_{\dyadicPart[m]{k-1}\dyadicPart[m]{k}}^{\odot q-r}
						\odot
						\delta_{\dyadicPart[m]{l-1}\dyadicPart[m]{l}}^{\odot q-r}
					)
				F_{st}^{f,\mu}(X)
				F_{uv}^{g,\nu}(X)
			\right]\\
		\begin{aligned}
			&=
			\expect
				\left[
					\innerProdLarge[\CM^{\odot 2q-2r}]
						{
							\delta_{\dyadicPart[m]{k-1}\dyadicPart[m]{k}}^{\odot q-r}
							\odot
							\delta_{\dyadicPart[m]{l-1}\dyadicPart[m]{l}}^{\odot q-r}
						}
						{
							D^{2q-2r}
							\left\{
								F_{st}^{f,\mu}(X)
								F_{uv}^{g,\nu}(X)
							\right\}
						}
				\right]\\
			&=
			\sum_{a+b=2q-2r}
				\frac{(2q-2r)!}{a!b!}
				\expect
					\left[
						\innerProdLarge[\CM^{\odot 2q-2r}]
						{
							\delta_{\dyadicPart[m]{k-1}\dyadicPart[m]{k}}^{\odot q-r}
							\odot
							\delta_{\dyadicPart[m]{l-1}\dyadicPart[m]{l}}^{\odot q-r}
						}
						{
							D^a	F_{st}^{f,\mu}(X)
							\odot
							D^b	F_{uv}^{g,\nu}(X)
						}
					\right].
		\end{aligned}
	\end{multline*}
	From \pref{prop_20150520010300}, we see that
	\begin{align*}
		\expect
			[
				|
					\innerProd[\CM^{\odot a}]
						{D^a F^{f,\mu}_{st}(X)}
						{h^1\odot\cdots\odot h^a}
				|^r
			]^{1/r}
		\leq
			\const
			\|h^1\|_\infty\cdots\|h^a\|_\infty
		\leq
			\const
			\begin{cases}
				2^{-2ma\Hurst},&0<\Hurst< 1/2,\\
				(2\Hurst)^a 2^{-ma},&1/2\leq\Hurst<1,\\
			\end{cases}
	\end{align*}
	for
	$
		h^1,\dots,h^a
		\in
			\{
				\delta_{\dyadicPart[m]{k-1}\dyadicPart[m]{k}},
				\delta_{\dyadicPart[m]{l-1}\dyadicPart[m]{l}}
			\}
	$.
	Combining them, the proof is completed.
\end{proof}
\begin{proposition}\label{prop_20141203092733}
	Let $q\geq 2$.
	There exists a positive constant $\const$ such that
	\begin{align*}
		\expect
			[|\wHerVar{q}{m}(t)-\wHerVar{q}{m}(s)|^2]
		\leq
			\const
			(\intPart{2^m t}-\intPart{2^m s})
			\begin{cases}
				2^{m(1-2q\Hurst)}, & 0<\Hurst\leq 1/2q,\\
				1, & 1/2q<\Hurst<1-1/2q,\\
				m, & \Hurst=1-1/2q,\\
				2^{m\{1-2q(1-\Hurst)\}}, & 1-1/2q<\Hurst<1,\\
			\end{cases}
	\end{align*}
	for any $0\leq s<t\leq 1$.
\end{proposition}
\begin{proof}
	We can prove this proposition in the same way as \cite[Proposition~21]{Naganuma2015}
	by using \pref{prop_20141117073846} instead of \cite[Proposition~19]{Naganuma2015}.
	In more detail, we use \eqref{eq_product_formula}
	to rewrite $|\wHerVar{q}{m}(t)-\wHerVar{q}{m}(s)|^2$
	by the It\^o-Wiener integrals.
	Then we see that it is expressed by the summation of the integrand in \pref{prop_20141117073846}.
	From \pref{prop_20141117073846}, we see the conclusion.
\end{proof}

We prove \tref[thm_20141203094309]{thm_20141117071529}.

\begin{proof}[Proof of \tref{thm_20141203094309}]
	Recall the identity
    $
		\xi^q
		=
			\sum_{r=0}^q
				\binom{q}{r}
				\expect[Z^{q-r}]
				\hermitePoly{r}(\xi)
    $
	for any $\xi\in\RealNum$, where $Z$ is a standard Gaussian random variable.
	Applying this identity, we see
	\begin{align*}
		2^{m(q\Hurst-1)}
		\sum_{k=1}^{\intPart{2^m\cdot}}
			F_{\dyadicPart[m]{k-1}\dyadicPart[m]{k}}(X)
			(B_{\dyadicPart[m]{k-1}\dyadicPart[m]{k}})^q
		&=
			2^{-m}
			\sum_{k=1}^{\intPart{2^m\cdot}}
				F_{\dyadicPart[m]{k-1}\dyadicPart[m]{k}}(X)
				(2^{m\Hurst}B_{\dyadicPart[m]{k-1}\dyadicPart[m]{k}})^q\\
		&=
			2^{-m}
			\sum_{r=0}^q
				\binom{q}{r}
				\expect[Z^{q-r}]
				\sum_{k=1}^{\intPart{2^m\cdot}}
					F_{\dyadicPart[m]{k-1}\dyadicPart[m]{k}}(X)
					\hermitePoly{r}(2^{m\Hurst}B_{\dyadicPart[m]{k-1}\dyadicPart[m]{k}})\\
		&=
			\expect[Z^q]
			\cdot
			2^{-m}
			\sum_{k=1}^{\intPart{2^m\cdot}}
				F_{\dyadicPart[m]{k-1}\dyadicPart[m]{k}}(X)
			+
			\sum_{r=2}^q
				\binom{q}{r}
				\expect[Z^{q-r}]
				\cdot
				2^{-m}
				\wHerVar{r}{m}.
	\end{align*}
	We prove convergence of the first and second term in the following.

	We consider the first term.
	Note
	\begin{multline*}
		2^{-m}
		\sum_{k=1}^{\intPart{2^m t}}
			F_{\dyadicPart[m]{k-1}\dyadicPart[m]{k}}(X)
		-
		\int_0^t
			f(X_s)\,
			ds\\
		\begin{aligned}
			&=
				\sum_{k=1}^{\intPart{2^m t}}
					\int_{\dyadicPart[m]{k-1}}^{\dyadicPart[m]{k}}
						F_{\dyadicPart[m]{k-1}\dyadicPart[m]{k}}(X)\,
						ds
				-
				\sum_{k=1}^{\intPart{2^m t}}
					\int_{\dyadicPart[m]{k-1}}^{\dyadicPart[m]{k}}
						f(X_s)\,
						ds
				-
				\int_{\intPart{2^m t}}^t
					f(X_s)\,
					ds\\
			&=
				\sum_{k=1}^{\intPart{2^m t}}
					\int_{\dyadicPart[m]{k-1}}^{\dyadicPart[m]{k}}
						\{
							F_{\dyadicPart[m]{k-1}\dyadicPart[m]{k}}(X)
							-
							f(X_s)
						\}\,
						ds
				-
				\int_{\intPart{2^m t}}^t
					f(X_s)\,
					ds.
		\end{aligned}
	\end{multline*}
	Since $X$ is $(\Hurst-\epsilon)$-H\"{o}lder continuous,
	we see that the absolute value of the above has an upper bound
	\begin{align*}
		\sum_{k=1}^{\intPart{2^m t}}
			\int_{\dyadicPart[m]{k-1}}^{\dyadicPart[m]{k}}
				|
					F_{\dyadicPart[m]{k-1}\dyadicPart[m]{k}}(X)
					-
					f(X_s)
				|\,
				ds
		\leq
			\sum_{k=1}^{\intPart{2^m t}}
			\int_{\dyadicPart[m]{k-1}}^{\dyadicPart[m]{k}}
				\const[X]
				2^{-m(\Hurst-\epsilon)}\,
				ds
		=
			\const[X]
			2^{-m(\Hurst-\epsilon)},
	\end{align*}
	where $\const[X]$ is a random variable.
	Hence
	\begin{align*}
		\lim_{m\to\infty}
			2^{-m}
			\sum_{k=1}^{\intPart{2^m \cdot}}
				F_{\dyadicPart[m]{k-1}\dyadicPart[m]{k}}(X)
		=
			\int_0^\cdot
				f(X_s)\,
				ds
	\end{align*}
	almost surely with respect to the uniform norm.

	We prove convergence of the process $2^{-m}\wHerVar{r}{m}$ to the process $0$ for $r=2,\dots,q$.
	It follows from \pref{prop_20141203092733} that
	\begin{align*}
		\expect[|2^{-m}\wHerVar{r}{m}(t)-2^{-m}\wHerVar{r}{m}(s)|^2]
		&\leq
			\const
			\frac{\intPart{2^m t}-\intPart{2^m s}}{2^m}
			\begin{cases}
				2^{-2rm\Hurst}, & 0<\Hurst\leq 1/2r,\\
				2^{-m}, & 1/2r<\Hurst<1-1/2r,\\
				m2^{-m}, & \Hurst=1-1/2r,\\
				2^{-2rm(1-\Hurst)}, & 1-1/2r<\Hurst<1,\\
			\end{cases}\\
		&\leq
			\const
			\left(
				\frac{\intPart{2^m t}-\intPart{2^m s}}{2^m}
			\right)^{1+\kappa}
			\begin{cases}
				2^{-\kappa m}, & 0<\Hurst\leq 1/2r,\\
				2^{-\kappa m}, & 1/2r<\Hurst<1-1/2r,\\
				m2^{-\kappa m}, & \Hurst=1-1/2r,\\
				2^{-\kappa m}, & 1-1/2r<\Hurst<1,\\
			\end{cases}
	\end{align*}
	where
	\begin{align*}
		\kappa
		=
			\begin{cases}
				r\Hurst, & 0<\Hurst\leq 1/2r,\\
				1/2, & 1/2r<\Hurst<1-1/2r,\\
				1/2, & \Hurst=1-1/2r,\\
				r(1-\Hurst), & 1-1/2r<\Hurst<1.\\
			\end{cases}
	\end{align*}
	This inequality implies convergence of $2^{-m}\wHerVar{q}{m}$ to the zero process.

	The proof is completed.
\end{proof}

\begin{proof}[Proof of \tref{thm_20141117071529}]
	The assertion is proved in the same way as \cite[Theorem~15]{Naganuma2015}
	by using \pref{prop_20141117073846} instead of \cite[Proposition~19]{Naganuma2015}.
	In this proof, we use \pref{prop_20141102222020}.
\end{proof}

\section*{Acknowledgments}
The first author's research was partially supported
by Grant-in-Aid for Scientific Research (B) No.~JP16H03938.
The second author's research was partially supported
by Grant-in-Aid for Young Scientists (B) No.~JP17K14202.

\address{
Graduate School of Mathematical Sciences,\\
The University of Tokyo,\\
Meguro-ku, Tokyo, 153-8914, Japan}
{aida@ms.u-tokyo.ac.jp}

\address{
Graduate School of Engineering Science,\\
Osaka University,\\
Toyonaka, Osaka, 560-8531, Japan}
{naganuma@sigmath.es.osaka-u.ac.jp}

\end{document}